\documentclass[a4paper, 11 pt, oneside, reqno]{amsart}
\usepackage{amsfonts, amssymb, amsmath, braket, color, eucal, amsthm, mathdots}
\usepackage{geometry}

\usepackage[utf8]{inputenc}
\DeclareFontFamily{U}{min}{}
\DeclareFontShape{U}{min}{m}{n}{<-> udmj30}{}

\usepackage{tikz,tikz-cd}
\usetikzlibrary{decorations.pathreplacing}

\usepackage{color}
\definecolor{myurlcolor}{rgb}{0.1,0.1,0.8}
\definecolor{mylinkcolor}{rgb}{0.05,0.05,0.4}

\usepackage{hyperref}
\hypersetup{colorlinks,
	linkcolor=mylinkcolor,
	citecolor=mylinkcolor,
	urlcolor=myurlcolor}

\usepackage{cleveref}

\newtheorem{lemma}{Lemma}
\numberwithin{lemma}{section}
\newtheorem{proposition}[lemma]{Proposition}
\newtheorem{theorem}[lemma]{Theorem}
\newtheorem{corollary}[lemma]{Corollary}

\theoremstyle{definition}
\newtheorem{definition}[lemma]{Definition}
\newtheorem{example}[lemma]{Example}
\newtheorem{remark}[lemma]{Remark}

  \newlength\squareheight
  \newcommand\squareslash{\tikz{\draw (0,0) rectangle (\squareheight,\squareheight);\draw(0,0) -- (\squareheight,\squareheight)}}
  \DeclareMathOperator\squarediv{\squareslash}

  \newlength\subsquareheight
\newcommand{\demph}[1]{\emph{#1}}
\newcommand{\xto}{\xrightarrow}

\newcommand{\Nv}{\mathcal{N}}

\newcommand{\OP}{\mathcal{OP}}
\newcommand{\mathbbZ}{\mathbb{Z}}
\newcommand{\cn}[1]{\mathbf{#1}}

\newcommand{\cl}[1]{\mathcal{#1}}
\newcommand{\colim}{\mathrm{colim}}

\DeclareMathOperator{\sign}{sign}
\DeclareMathOperator{\tor}{Tor}
\newcommand{\m}{\mathfrak{m}}

\newcommand{\NMet}{\mathbb{N}\cn{Met}}
\newcommand{\DiGraph}{\cn{DiGraph}}

\newcommand{\bpc}{\mathrm{RC}}

\newcommand{\RC}{\mathrm{RC}}
\newcommand{\RH}{\mathrm{RH}}
\newcommand{\MC}{\mathrm{MC}}
\newcommand{\MH}{\mathrm{MH}}
\newcommand{\PH}{\mathrm{PH}}
\newcommand{\OPH}{\mathrm{PH}}
\newcommand{\MPSS}{E}

\title[Bigraded path homology and the MPSS]{Bigraded path homology and the magnitude-path spectral sequence}

\author{Richard Hepworth}
\address{Institute of Mathematics\\ University of Aberdeen\\ Scotland}
\email{r.hepworth@abdn.ac.uk}

\author{Emily Roff}
\address{School of Mathematics\\ University of Edinburgh\\ Scotland}
\email{emily.roff@ed.ac.uk}



\begin{document}

\begin{abstract}
Two important invariants of directed graphs, namely magnitude homology and path homology, have recently been shown to be intimately connected: there is a \emph{magnitude-path spectral sequence} or \emph{MPSS} in which magnitude homology appears as the first page, and in which path homology appears as an axis of the second page.
In this paper we study the homological and computational properties of the spectral sequence, and in particular of the full second page, which we now call \emph{bigraded path homology}.
We demonstrate that every page of the MPSS deserves to be regarded as a homology theory in its own right, satisfying excision and K\"unneth theorems (along with a homotopy invariance property already established by Asao), and that magnitude homology and bigraded path homology also satisfy Mayer--Vietoris theorems.
We construct a homotopy theory of graphs (in the form of a cofibration category structure) in which weak equivalences are the maps inducing isomorphisms on bigraded path homology, strictly refining an existing structure based on ordinary path homology.
And we provide complete computations of the MPSS for two important families of graphs---the directed and bi-directed cycles---which demonstrate the power of both the MPSS, and bigraded path homology in particular, to distinguish graphs that ordinary path homology cannot.
\end{abstract}

\maketitle

\setcounter{tocdepth}{1}
\tableofcontents

\section{Introduction}

There is, by now, an abundance of homological invariants of directed graphs, including path homology~\cite{GLMY2013,PH-homotopy}, magnitude homology \cite{HepworthWillerton2017} and reachability homology \cite{HepworthRoff2023}. In the context of undirected graphs, even more invariants are applicable, including clique homology~\cite{DeMeyerDeMeyer, PetriEtAl} and discrete cubical homology \cite{Barcelo-etal-2014}. 
These theories have some notable successes,
such as work of Asao relating magnitude homology to curvature of metric spaces~\cite{Asao-curvature};
work of Asao, Hiraoka and Kanazawa relating magnitude homology to girth of graphs~\cite{Asao-curvature,AsaoGirth};
work of Tajima and Yoshinaga developing a corresponding
magnitude homotopy type~\cite{TajimaYoshinaga};
a version of discrete Morse theory relevant to path homology developed by Lin, Wang and Yau~\cite{PH-discretemorse};
and path homology analogues of classical geometric results developed by Kempton, M\"unch and Yau \cite{PH-curvature}.
Many of these homology theories possess formal properties analogous to the Eilenberg--Steenrod axioms, yet they tend to disagree when evaluated on even simple classes of graphs. It is desirable, then, to understand the relationships among them, where such relationships exist.

Asao has recently shown that two of these homology theories---namely magnitude homology and path homology---are indeed closely related, appearing on consecutive pages in a certain spectral sequence \cite{Asao-path}. We follow Di \emph{et al}~\cite{DIMZ} in referring to that spectral sequence as the \emph{magnitude-path spectral sequence} or \emph{MPSS}. Page 1 of the MPSS is exactly magnitude homology, while page 2 contains path homology on its horizontal axis, and the target of the spectral sequence is reachability homology \cite{HepworthRoff2023}. The MPSS thus encompasses three existing invariants of directed graphs and clarifies their relationships, while adding infinitely many new ones. In particular, it extends path homology (the horizontal axis of the second page) to a bigraded theory (the entire second page) that we now call \emph{bigraded path homology}.

The objective of this paper is to demonstrate the properties, strength and usefulness of the MPSS, and of bigraded path homology in particular. Asao has shown that each page of the MPSS has a homotopy-invariance property, whose strength increases as one turns through the pages of the sequence~\cite{Asao-filtered}. Building on that observation, we demonstrate that the spectral sequence as a whole possesses formal properties that justify calling \emph{every} page a homology theory for directed graphs---each one distinct, but systematically related. Concerning the MPSS, our main results are as follows.
    \begin{itemize}
    \item[(A)] Every page of the MPSS satisfies a K\"unneth theorem with respect to the box product (\Cref{thm:kunneth_box_SS}).
    \item[(B)] Every page of the MPSS satisfies an excision theorem with respect to a class of subgraph inclusions first studied in \cite{CDKOSW} (\Cref{excision-all}).
\end{itemize}
In particular these two properties all hold for magnitude homology and bigraded path homology. Additionally, magnitude homology and bigraded path homology each satisfy a Mayer--Vietoris theorem (Theorems \ref{Mayer--Vietoris-magnitude} and \ref{Mayer--Vietoris-path}) that we are able to deduce from the excision property using straightforward homological algebra.

The diversity of homological viewpoints on directed graphs has motivated a recent drive towards consolidation using formal homotopy theory, and \Cref{sec:cofibration category} of this paper contributes to that development. In \cite{CDKOSW}, Carranza \textit{et al} prove that the category of directed graphs carries a cofibration category structure for which the weak equivalences are maps inducing isomorphisms on path homology. By specializing properties (A) and (B) to bigraded path homology, we can prove that their structure admits a natural refinement:
\begin{itemize}
    \item[(C)]
    The category of directed graphs 
    carries a cofibration category structure 
    in which the cofibrations are those of \cite{CDKOSW} 
    and the weak equivalences are maps 
    inducing isomorphisms on bigraded path homology (\Cref{thm:cofib_cat}).
\end{itemize}
That this structure is indeed strictly finer than the one in \cite{CDKOSW} is demonstrated by complete computations of bigraded path homology and the MPSS for two important families of directed graphs. These are the \emph{directed cycles} $Z_m$ for $m\geq 1$, and the \emph{bi-directed cycles} $C_{m,n}$ for $m,n\geq 1$, depicted in~\Cref{fig:cycles}.
\begin{figure}
    \[
        \begin{tikzpicture}[baseline=0]
            \node(u) at (-135:1.5) {};
            \node(v) at (-90:1.5) {};
            \node(w) at (-45:1.5) {};
            \node(y) at (0:1.5) {};
            \node(z) at (225:1.5) {};
            \node(x) at (180:1.5) {$\bullet$};
            \node(a2) at (135:1.5) {$\bullet$};
            \node (a3) at (90:1.5) {$\bullet$};
            \node (a4) at (45:1.5) {$\bullet$};
    
            \draw[-stealth,dashed] (z) -- (x);
            \draw[-stealth] (x) -- (a2);
            \draw[-stealth] (a2) -- (a3);
            \draw[-stealth] (a3) -- (a4);
            \draw[-stealth,dashed] (a4) -- (y);
            \draw[-stealth,dashed] (y) -- (w);
            \draw[-stealth,dashed] (w) -- (v);
            \draw[-stealth,dashed] (v) -- (u);
    
            \node () at (0:0) {$Z_m$};
           \draw [
            decorate, 
            decoration = 
                {
                    brace,
                    raise=5,
                    amplitude=5pt,
                }
            ] 
                (1.6,1.6) --  (1.6,-1.6)
                node[pos=0.5,right=9pt]
                {$m$ edges};
        \end{tikzpicture}
        \qquad\qquad
        \begin{tikzpicture}[baseline=0]
            \node(x) at (180:1.5) {$\bullet$};
            \node(y) at (0:1.5) {$\bullet$};
    
            \node (a2) at (135:1.5) {$\bullet$};
            \node (a3) at (90:1.5) {};
            \node (a4) at (45:1.5) {$\bullet$};
    
            \node (b2) at (216:1.5) {$\bullet$};
            \node (b3) at (252:1.5) {};
            \node (b4) at (288:1.5) {};
            \node (b5) at (324:1.5) {$\bullet$};
            
            \draw[-stealth] (x) -- (a2);
            \draw[-stealth,dashed] (a2) -- (a3);
            \draw[-stealth,dashed] (a3) -- (a4);
            \draw[-stealth] (a4) -- (y);
    
            \draw[-stealth] (x) -- (b2);
            \draw[-stealth,dashed] (b2) -- (b3);
            \draw[-stealth,dashed] (b3) -- (b4);
            \draw[-stealth,dashed] (b4) -- (b5);
            \draw[-stealth] (b5) -- (y);
            
            \node () at (0:0) {$C_{m,n}$};
       \draw [
        decorate, 
        decoration = 
            {
                brace,
                raise=5,
                amplitude=5pt,
            }
        ] 
            (1.6,1.6) --  (1.6,0.1)
            node[pos=0.5,right=9pt]
            {$m$ edges};
       \draw [
        decorate, 
        decoration = 
            {
                brace,
                raise=5,
                amplitude=5pt,
            }
        ] 
            (1.6,-0.1) --  (1.6,-1.6)
            node[pos=0.5,right=9pt]
            {$n$ edges};
        \end{tikzpicture}
    \]
    \caption{The directed cycle $Z_m$ and the bi-directed cycle $C_{m,n}$}
    \label{fig:cycles}
\end{figure}
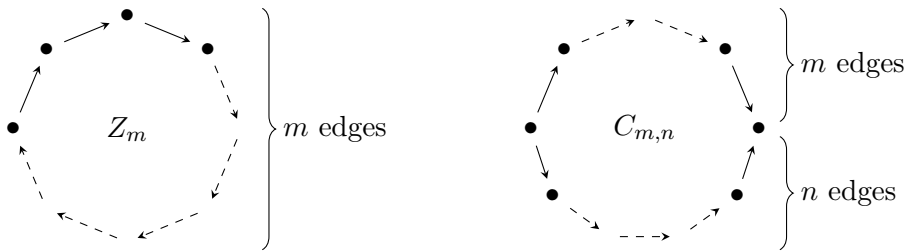

Path homology can only distinguish $Z_1$ and $Z_2$
(which it sees as `contractible') from the $Z_m$ for $m\geq 3$ (which it
sees as `circle-like').
Similarly, path homology cannot distinguish any of the $C_{m,n}$ from
one another so long as $\max(m,n)\geq 3$.
We find, however, that the MPSS, and even bigraded path homology,
can do much better:
\begin{itemize}
	\item[(D)]
	Bigraded path homology can distinguish all of the $Z_m$
	for $m\geq 2$ by inspecting the bidegrees in which generators lie.
	Further, the MPSS of $Z_m$ characterises $m\geq 1$ 
        as the first value of $r$  for which the $E^r$-page is trivial, i.e.~concentrated in bidegree $(0,0)$. 
        (\Cref{theorem-directed-cycles-omnibus}
        and the subsequent paragraph.)
	\item[(E)]
	The bigraded path homology of $C_{m,n}$ depends only on the value of
        $\m=\max(m,n)$, and for $\m\geq 2$ one can determine the value of 
        $\m$ by inspecting the bidegrees in which generators lie.
	Further, the MPSS of $C_{m,n}$ characterises $\m\geq 2$ 
        as the first value of
	$r$ for which the $E^r$-page is trivial, i.e.~concentrated 
	in bidegree $(0,0)$.
        (\Cref{proposition-independent} 
        and~\Cref{theorem-mn-cycles-omnibus}.)
\end{itemize}
We believe that the results (A)--(E) demonstrate clearly the theoretical and computational strength of the MPSS, and of bigraded path homology in particular. Indeed, it appears that bigraded path homology shares many of the strengths and advantages of path homology, whilst being a more sensitive and informative invariant. On a technical front, our methods illustrate a useful principle: that properties of path homology are frequently (though not always) `inherited' from corresponding properties of magnitude homology---and that what holds true for either of these will often hold true throughout the MPSS. Indeed, proofs of results about path homology that use the standard construction  of \cite{GLMY2013,PH-homotopy}  are often rather involved, and we believe that passage to the bigraded theory and the MPSS can serve to clarify such proofs, rather than complicating them.

Moreover, it is our belief that the MPSS as a whole will eventually cast more light on the homotopy theory of directed graphs. Elaborating on speculations made by Asao in the introduction to~\cite{Asao-filtered}, it is tempting to conjecture that the cofibration category we describe may belong to a nested family of structures, one for each page of the sequence, with \emph{$r$-homotopy} (\Cref{def:r-homotopy}) functioning as the relevant notion of homotopy for the theory in which weak equivalences induce isomorphisms on the $(r+1)$-page. Zooming in a little, consider the excision result mentioned above, which shows that  certain maps of pairs obtained from cofibrations induce isomorphisms on the MPSS from the $E^1$ page onwards. This is in contrast to the situation for homotopies, where $r$-homotopic maps induce equal maps on the MPSS from only the $(r+1)$-page onwards. We wonder whether there may be a notion of~\emph{$r$-cofibration}, more general than the existing notion of cofibration, which induces excision isomorphisms from the $E^{r+1}$ page onwards; such cofibrations might feature in a homotopy theory in which weak equivalences induce isomorphisms  on the $(r+1)$-page. The development of this idea is left to future work.

\subsection*{Open questions}
Magnitude homology, path homology and the MPSS
are all relatively new invariants, and as such there
are many open questions and opportunities for 
further research.  
Here, we highlight just a few that we anticipate will
be fruitful:
\begin{itemize}
    \item To what extent can entries of the MPSS be nonzero, 
    and to what extent can this be controlled?
    For example, given an arbitrary location $E^r_{i,j}$ in the MPSS, 
    can we find a graph $G$ for which $E^r_{i,j}(G)\neq 0$? 
    And can the diameter of such $G$ be chosen to be $r$?
    (The answer to such questions will often be `no',
    if only because the MPSS is concentrated in a specific octant,
    but we anticipate further restrictions still.
    Our computations for $Z_m$, for instance, provide many examples,
    but still only succeed in occupying bidegrees $(a,b)$ for which
    $\frac{a}{2}|b$ if $a$ is even, or $\frac{a-1}{2}|(b-1)$ if $a$ is odd.)
    
    \item To what extent can torsion appear in the MPSS, 
    and to what extent can it be controlled?  
    For example, 
    given a specific finitely-generated torsion abelian group,
    and a specific location $E^r_{i,j}$ in the MPSS,
    can we find a graph whose MPSS features that torsion 
    group in that location?
    (We do not know whether this question has been investigated in full
    even for magnitude homology.
    Again, the answer to this question will often be `no',
    but the nature and extent of any restrictions will 
    themselves be of interest.)
    \item To what extent do the results and methods of this paper 
    extend to the category $\NMet$ of generalised metric spaces
    and short maps with distances in $\mathbb{N}\cup\{\infty\}$?
    We anticipate that extending to this category will elucidate and
    facilitate any attacks on `realisation problems'
    like the two above.
    
    \item  Kaneta and Yoshinaga gave, under certain conditions,
    a decomposition of magnitude homology as a direct sum indexed
    by certain \emph{frames}~\cite[Theorem~3.12]{KanetaYoshinaga}.
    To what extent does this decomposition extend to the MPSS?
    In particular, how does it interact with the differentials $d^1$
    of the MPSS?

    \item Is there a Mayer--Vietoris theorem 
    applying to pages $E^r$ for $r\geq 3$?
    (We prove Mayer--Vietoris theorems for $E^1$ and $E^2$, which take the
    form of a short and long exact sequence respectively,
    the latter obtained from the former by passing to homology
    with respect to the $d^1$ differential.
    However, applying homology in a long exact sequence does not produce
    a long exact sequence, but rather a spectral sequence converging to $0$.
    A significant sub-question, then, is to identify the relevant 
    algebraic structures, generalising short and long exact sequences,
    that will relate the $E^r$-pages.)
    
    \item
    Are there versions of magnitude corresponding to the later pages
    of the MPSS of a graph $G$
    that can be described directly from the graph,
    without reference to homological algebra?
    (The notion of magnitude of a finite graph preceded the 
    introduction of magnitude homology,
    and has a simple and direct definition~\cite{LeinsterGraph}.)
\end{itemize}

\subsection*{Acknowledgements}
We thank Yasuhiko Asao, Sergei O.~Ivanov, Maru Sarazola and Masahiko Yoshinaga for interesting and helpful conversations.
This work was partially supported by JSPS Postdoctoral Fellowships for Research in Japan.
%

\section{The magnitude-path spectral sequence}\label{sec:MPS}

We begin by fixing terminology and notation concerning the categories of directed graphs and metric spaces, before proceeding to describe the magnitude-path spectral sequence. We will assume familiarity with spectral sequences in 
general, but for the reader who is new to the subject, 
we recommend Sections~5.1 and 5.2 of~\cite{Weibel},
and Section~2.1 of~\cite{McCleary}.

\subsection{The category of directed graphs}
\label{sec:digraph_met}

This paper is concerned with directed graphs, which are permitted to contain loops but not parallel edges. Thus, a directed graph \(G\) consists of a set \(V(G)\) of vertices and a set \(E(G) \subseteq V(G) \times V(G)\) of (directed) edges. We depict an edge \((x,y)\) by an arrow \(x \to y\). A \demph{(directed) path} from a vertex \(x\) to a vertex \(y\) in \(G\) is a sequence of consistently-oriented edges
\[x = x_0 \to x_1 \to \cdots \to x_k = y.\]
(Note that we do not require the vertices in a path to be distinct.) A \demph{map of directed graphs} \(G \to H\) is a function \(f\colon V(G) \to V(H)\) such that for every edge \(x \to y\) in \(G\), either \(f(x) = f(y)\) or there is an edge \(f(x) \to f(y)\) in \(H\) (or both). The category of directed graphs and maps of directed graphs is denoted by \(\cn{DiGraph}\).

We will also be interested in the category \(\cn{Met}\) of \demph{(generalized) metric spaces} and \demph{short maps}. A generalized metric space is a set \(X\) equipped with a function \(d_X\colon X \times X \to [0, +\infty]\) such that \(d_X(x,x)=0\) for every \(x \in X\) and the triangle inequality is satisfied. We will always call the function \(d_X\) a \demph{metric}, though it may not be separated or symmetric. A short map is a function \(f\colon X \to Y\) such that \(d_Y(f(x),f(x')) \leq d_X(x,x')\) for every \(x,x' \in X\).

Every directed graph \(G\) carries a metric on its vertex set, in which \(d(x,y)\) is the minimal number of edges in a directed path from \(x\) to \(y\), or \(+\infty\) if no such path exists. This is called the \demph{shortest path metric} on \(G\). The operation of equipping a directed graph with the shortest path metric extends to a functor
\[M\colon \cn{DiGraph} \hookrightarrow \cn{Met}\]
which is full and faithful, making \(\cn{DiGraph}\) into a full subcategory of \(\cn{Met}\).

\subsection{Reachability chains and the length filtration}

Throughout the paper we work over a fixed commutative ground ring $R$.
The \emph{reachability chain complex} of a directed graph $G$
is the chain complex $\RC_\ast(G)$ of $R$-modules 
defined as follows.
In degree $k$ the \(R\)-module $\RC_k(G)$ has basis
given by all tuples $(x_0,\ldots,x_k)$ of vertices of $G$ in which
consecutive entries are distinct, and in which there is a directed
path in $G$ from each entry to the next.
The differential of $\RC_\ast(G)$ is given by
\begin{equation}\label{eq-differential}
    \partial(x_0,\ldots,x_k)
    =
    \sum_{i=0}^k
    (-1)^i
    (x_0,\ldots,\widehat{x_i},\ldots,x_k)
\end{equation}
where any term with repeated consecutive entries is omitted.
Note that in the summand
$(x_0,\ldots,\widehat{x_i},\ldots,x_k)$ a path between
the adjacent terms $x_{i-1}$ and $x_{i+1}$ can be obtained by 
concatenating a path from $x_{i-1}$ to $x_i$ with one from
$x_i$ to $x_{i+1}$, these existing due to the original assumption
on $(x_0,\ldots,x_k)$.
The \emph{reachability homology} $\RH_\ast(G)$ of $G$
is simply the homology of $\RC_\ast(G)$,
and was studied in detail in~\cite{HepworthRoff2023}.

The \emph{length} of a tuple $(x_0,\ldots,x_k)$ of vertices in \(G\) is defined, using the shortest path metric, by
\[
    \ell(x_0,\ldots,x_k)
    =
    d(x_0,x_1)+\cdots + d(x_{k-1},x_k).
\]
The triangle inequality guarantees that 
$\ell(x_0,\ldots,\widehat{x_i},\ldots,x_k)$ is at 
most $\ell(x_0,\ldots,x_k)$.
This allows us to define a filtration 
\[
    F_0\RC_\ast(G)
    \subseteq
    F_1\RC_\ast(G)
    \subseteq
    F_2\RC_\ast(G)
    \subseteq
    \cdots
\]
of $\RC_\ast(G)$ 
by defining $F_\ell\RC_\ast(G)$ to be the subcomplex spanned
by the generators of length at most $\ell$.

\begin{remark}\label{rmk:nerve}
    The reachability chains of a directed graph $G$ can also be described as the normalized simplicial chains of a certain simplicial set---a perspective that will be useful in \Cref{sec:E-Z}. We follow \cite[Section 1.2]{DIMZ} in calling this simplicial set the `(filtered) nerve'; it is defined as follows.

    The \demph{nerve} of \(G\) is the simplicial set \(\Nv(G)\) whose set of \(k\)-simplices is
    \[\Nv_k(G) = \left\{(x_0, \ldots, x_k) \middle\vert 
    \begin{array}{l}
    x_0 \ldots x_k \in V(G) \text{ and for } 0 \leq i < k \text{ there} \\
    \text{exists a directed path from } x_i \text{ to } x_{i+1} \text{ in } G
    \end{array}
    \right\}.\]
    (Observe that here we do not insist that adjacent terms be distinct.) For \(0 \leq i \leq k\) the face operator \(\delta_i\colon \Nv_k(G) \to \Nv_{k-1}(G)\) discards the \(i^\mathrm{th}\) vertex, and the degeneracy operator \(\sigma_i\colon \Nv_k(G) \to \Nv_{k+1}(G)\) duplicates the \(i^\mathrm{th}\) vertex. For each \(\ell \in \mathbb{N}\), there is a sub-simplicial set $F_\ell\Nv(G)$ of \(\Nv(G)\) whose set of \(k\)-simplices comprises all tuples in \(\Nv_k(G)\) of length at most $\ell$. 
    These make \(\Nv(G)\) into a filtered simplicial set called the \demph{filtered nerve} of \(G\). 
    
    Given a simplicial set $A$, we write $N(A)$ or simply $NA$ for the complex of normalized simplicial chains on $A$; if $A$ is filtered, we equip $NA$ with the filtration defined by $F_\ell NA = N(F_\ell A)$. The reachability complex of \(G\) is precisely the normalized complex of simplicial chains in the nerve of \(G\) \cite[Section 4]{HepworthRoff2023}---
    \[
    \RC_\ast(G) = N(\Nv(G))
    \]
    ---and with the filtration by length, this becomes an equality of filtered chain complexes.
\end{remark} 

To any filtered chain complex we can associate a spectral
sequence that, roughly speaking, allows us to understand the
homology of the original complex in terms of the homology
of its filtration quotients.
(For details, see Section~5.4 of~\cite{Weibel}, 
and Section~2.2 of~\cite{McCleary}.)
In the case of the reachability chain complex with its length
filtration, this spectral sequence
is called the \emph{magnitude-path spectral sequence},
or MPSS for short, and denoted by
$\{E^r_{\ast,\ast}(G),d^r\}_{r\geq 0}$.
Since $\RC_\ast(-)$ and its filtration are functorial with respect
to maps of graphs, each group $E^r_{i,j}(-)$ determines a functor
from directed graphs to $R$-modules.
The MPSS was first mentioned, for undirected graphs,
in the original paper on magnitude 
homology~\cite[Remark~8.7]{HepworthWillerton2017},
but it was Asao who first demonstrated its importance
in~\cite{Asao-path}.
For more details on aspects of what follows, see~\cite{Asao-path} and Section~1 of~\cite{DIMZ}.

\subsection{The $E^0$-page}\label{subsection-E0}
The $E^0$-page of the spectral sequence associated to a filtered chain complex
is given by the filtration quotients, and the differential $d^0$
is induced from the differential on the original complex.
In the case of the MPSS, this recovers the \emph{magnitude chains}
$\MC_{\ast,\ast}(G)$ of a directed graph $G$~\cite{HepworthWillerton2017}.
This is the graded chain complex spanned in
bidegree $k,\ell$ by the generators of $\RC_\ast(G)$ that have 
degree $k$ and length precisely $\ell$, and whose differential
is given by~\eqref{eq-differential}, but with
any terms of length less than $\ell$ omitted.
(Note in particular that the first and last
terms of the sum are always omitted.)
It is easy to see that 
$\MC_{\ast,\ell}(G)=F_\ell\RC_\ast(G)/F_{\ell-1}\RC_\ast(G)$,
so that taking degree shifts into account, we find that
\[
    E^0_{i,j}(G)
    = F_i\RC_{i+j}(G) / F_{i-1}\RC_{i+j}(G)
    =\MC_{i+j,i}(G)
\]
with the magnitude chains differential.

Certain bounds on $k$ and $\ell$ must be satisfied in order
for magnitude chain groups to be nonzero.  
It follows that  the $E^0$-page of the MPSS is concentrated 
in bidegrees $i,j$ for which $i\geq 0$ and $-i\leq j\leq 0$; see \Cref{fig:E0}.
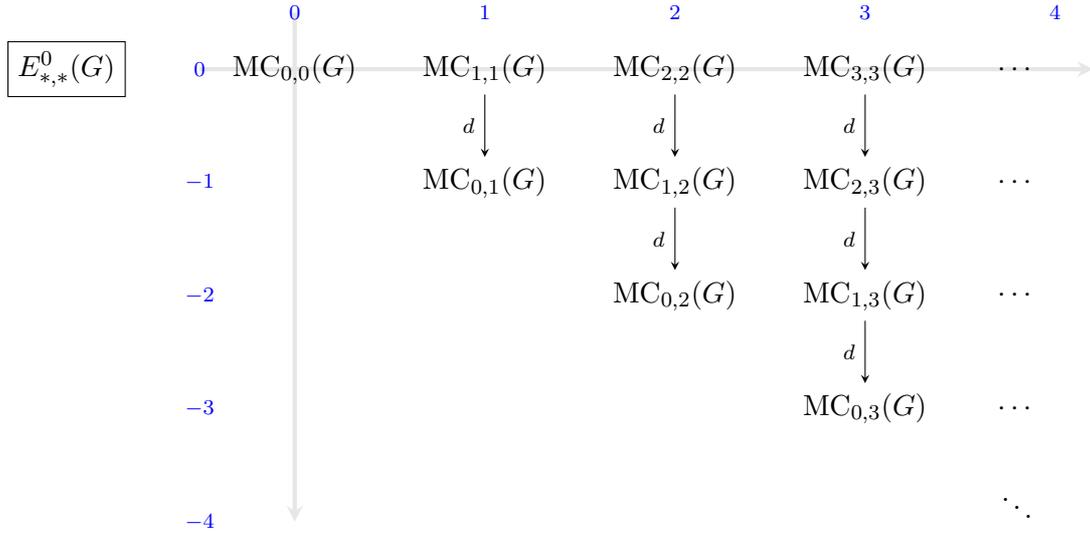
\begin{figure}[h!]
    \centering
    \begin{tikzpicture}[baseline=0,xscale=2.5,yscale=1.5]
        \node[draw]() at (-1.2,0) {$E^0_{\ast,\ast}(G)$};
        \draw[-stealth, ultra thick,white!90!black] (-0.5,0) to (4.2,0);
        \draw[-stealth, ultra thick,white!90!black] (0,0.5) to (0,-4);
        \node () at (0,0) {$\MC_{0,0}(G)$};
        \node (a) at (1,0) {$\MC_{1,1}(G)$};
        \node (b) at (1,-1) {$\MC_{0,1}(G)$};
        \node (c) at (2,0) {$\MC_{2,2}(G)$};
        \node (d) at (2,-1) {$\MC_{1,2}(G)$};
        \node (e) at (2,-2) {$\MC_{0,2}(G)$};
        \node (f) at (3,0) {$\MC_{3,3}(G)$};
        \node (g) at (3,-1) {$\MC_{2,3}(G)$};
        \node (h) at (3,-2) {$\MC_{1,3}(G)$};
        \node (i) at (3,-3) {$\MC_{0,3}(G)$};
        \node () at (3.8,-3.8) {$\ddots$};
        \node () at (3.8,0) {$\cdots$};
        \node () at (3.8,-1) {$\cdots$};
        \node () at (3.8,-2) {$\cdots$};
        \node () at (3.8,-3) {$\cdots$};
        \node[blue]() at (0,0.5) {$\scriptstyle 0$};
        \node[blue]() at (1,0.5) {$\scriptstyle 1$};
        \node[blue]() at (2,0.5) {$\scriptstyle 2$};
        \node[blue]() at (3,0.5) {$\scriptstyle 3$};
        \node[blue]() at (4,0.5) {$\scriptstyle 4$};
        \node[blue]() at (-0.5,0) {$\scriptstyle 0$};
        \node[blue]() at (-0.5,-1) {$\scriptstyle -1$};
        \node[blue]() at (-0.5,-2) {$\scriptstyle -2$};
        \node[blue]() at (-0.5,-3) {$\scriptstyle -3$};
        \node[blue]() at (-0.5,-4) {$\scriptstyle -4$};
        \draw[-stealth] (a) to node[left]{$\scriptstyle d$}(b);
        \draw[-stealth] (c) to node[left]{$\scriptstyle d$}(d);
        \draw[-stealth] (d) to node[left]{$\scriptstyle d$}(e);
        \draw[-stealth] (f) to node[left]{$\scriptstyle d$}(g);
        \draw[-stealth] (g) to node[left]{$\scriptstyle d$}(h);
        \draw[-stealth] (h) to node[left]{$\scriptstyle d$}(i);
    \end{tikzpicture}
    \caption{Page \(E^0\) of the MPSS is the magnitude chain complex.}
    \label{fig:E0}
\end{figure}
If $G$ has an upper bound on its finite lengths,
i.e.~if there is $K$ such that $d(x,y)\leq K$ or $d(x,y)=\infty$
for all vertices $x,y$,
then $j$ is constrained to lie in the smaller range
$-\frac{K-1}{K}\cdot i\leq j\leq 0$.
And the terms $\MPSS^0_{i,-i}=\MC_{0,i}(G)$ on the negative diagonal all
vanish for $i>0$.
(See~\cite[Proposition~2.10]{HepworthWillerton2017}
and  \cite[Theorem~4.1]{LeinsterShulman}.)

Observe that we are now working with two different 
kinds of bigrading, sometimes on the same object.
First, the bigrading of $\MC_{k,\ell}(G)$, which is well established
in the literature on magnitude homology, and which exactly picks out
the homological degree and length.
Second, the bigrading of $E^r_{i,j}(G)$, which has been established
in the literature on spectral sequences of this type since their
classical times.
Although it is awkward to use two different bigradings,
it seems that we are stuck with both.
We will always use notation to distinguish which bigrading system
is at play, writing $E^1_{i,j}(G)$ on the one hand and 
$\MC_{k,\ell}(G)$ on the other.

\subsection{The $E^1$-page}

The $E^1$-page of the spectral sequence associated to a filtration
is precisely the homology of the associated filtration quotients.
In the case of the MPSS, this recovers the \emph{magnitude homology}
$\MH_{\ast,\ast}(G)$ of the graph $G$,
defined to be precisely the homology of the magnitude chains,
$\MH_{k,\ell}(G) = H_k(\MC_{\ast,\ell}(G))$.
Thus, taking the degree shifts into account, we have
\[
    E^1_{i,j}(G)
    = 
    \MH_{i+j,i}(G),
\]
which we depict as in \Cref{fig:E1}.
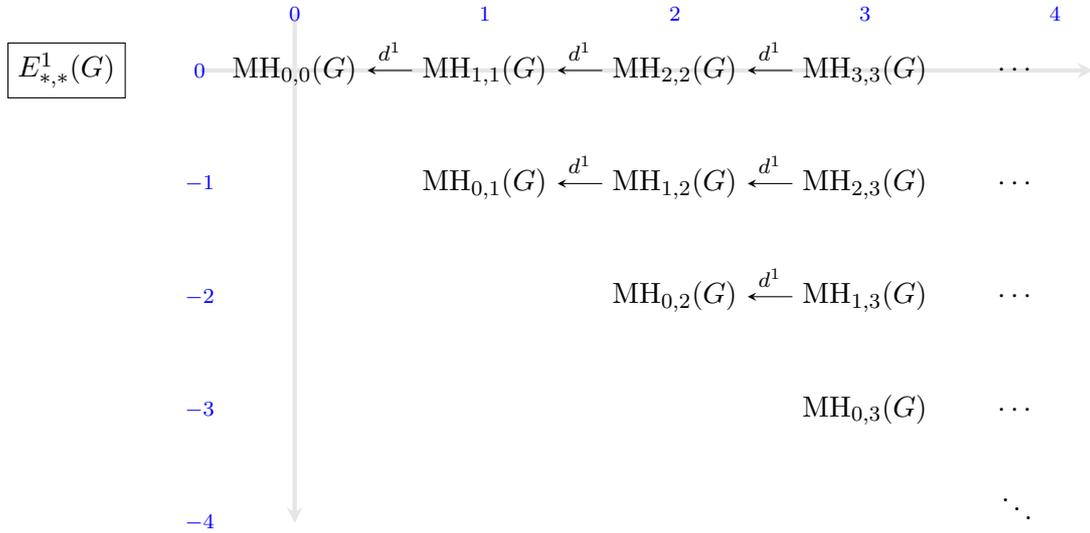
\begin{figure}[h!]
    \centering
    \begin{tikzpicture}[baseline=0,xscale=2.5,yscale=1.5]
        \node[draw]() at (-1.2,0) {$E^1_{\ast,\ast}(G)$};
        \draw[-stealth, ultra thick,white!90!black] (-0.5,0) to (4.2,0);
        \draw[-stealth, ultra thick,white!90!black] (0,0.5) to (0,-4);
        \node (aa) at (0,0) {$\MH_{0,0}(G)$};
        \node (a) at (1,0) {$\MH_{1,1}(G)$};
        \node (b) at (1,-1) {$\MH_{0,1}(G)$};
        \node (c) at (2,0) {$\MH_{2,2}(G)$};
        \node (d) at (2,-1) {$\MH_{1,2}(G)$};
        \node (e) at (2,-2) {$\MH_{0,2}(G)$};
        \node (f) at (3,0) {$\MH_{3,3}(G)$};
        \node (g) at (3,-1) {$\MH_{2,3}(G)$};
        \node (h) at (3,-2) {$\MH_{1,3}(G)$};
        \node (i) at (3,-3) {$\MH_{0,3}(G)$};
        \node () at (3.8,-3.8) {$\ddots$};
        \node () at (3.8,0) {$\cdots$};
        \node () at (3.8,-1) {$\cdots$};
        \node () at (3.8,-2) {$\cdots$};
        \node () at (3.8,-3) {$\cdots$};
        \node[blue]() at (0,0.5) {$\scriptstyle 0$};
        \node[blue]() at (1,0.5) {$\scriptstyle 1$};
        \node[blue]() at (2,0.5) {$\scriptstyle 2$};
        \node[blue]() at (3,0.5) {$\scriptstyle 3$};
        \node[blue]() at (4,0.5) {$\scriptstyle 4$};
        \node[blue]() at (-0.5,0) {$\scriptstyle 0$};
        \node[blue]() at (-0.5,-1) {$\scriptstyle -1$};
        \node[blue]() at (-0.5,-2) {$\scriptstyle -2$};
        \node[blue]() at (-0.5,-3) {$\scriptstyle -3$};
        \node[blue]() at (-0.5,-4) {$\scriptstyle -4$};
        \draw[-stealth] (f) to node[above]{$\scriptstyle d^1$}(c);
        \draw[-stealth] (c) to node[above]{$\scriptstyle d^1$}(a);
        \draw[-stealth] (a) to node[above]{$\scriptstyle d^1$}(aa);
        \draw[-stealth] (g) to node[above]{$\scriptstyle d^1$}(d);
        \draw[-stealth] (d) to node[above]{$\scriptstyle d^1$}(b);
        \draw[-stealth] (h) to node[above]{$\scriptstyle d^1$}(e);
    \end{tikzpicture}
    \caption{Page \(E^1\) of the MPSS is magnitude homology.}
    \label{fig:E1}
\end{figure}

The differential $d^1$ in the $E^1$ term of the spectral sequence
associated to a filtered chain complex $C_\ast$
is given by applying the differential of $C_\ast$
to appropriate representatives~\cite[5.4.6]{Weibel}.
In the case of the MPSS, this amounts to the following.
\begin{itemize}
    \item
    Take an element $x\in\MH_{i+j,i}(G)$.
    \item
    Represent $x$ by a cycle in $\MC_{i+j,i}(G)$.
    \item
    Regard the cycle as an element of $\RC_{i+j}(G)$.
    It is a combination of generators of length $i$.
    \item
    Apply the differential of $\RC_{i+j}(G)$
    to obtain an element of $\RC_{i+j-1}(G)$.
    This is a combination of generators of length at most $i-1$.
    \item
    Discard all generators of length $i-2$ or less 
    and regard the result as an element of $\MC_{i+j-1,i-1}(G)$.
    This element is a cycle.
    \item
    Then $d^1(x)$ is the associated homology class in
    $\MH_{i+j-1,i-1}(G)$.
\end{itemize}

\subsection{The $E^2$-page and beyond}\label{subsection:E2}
The $E^2$-page of the MPSS does not admit so direct a
description as the preceding pages.
Nevertheless, Asao showed  that it contains an important
invariant, namely the path homology, as its horizontal axis:
\[
    E^2_{i,0}(G) = \PH_i(G)
\]
See~\cite[Theorem~1.2]{Asao-path}. Though it is standard in some parts of the literature to denote path homology simply by \(H_\ast\), we write \(\PH_\ast\) to distinguish it more clearly from the other homology theories at play.

Path homology has an important homotopy-invariance property~\cite[Theorem~3.3]{PH-homotopy}. 
Asao showed that this extends to the rest of the $E^2$-page
and, with increasing strength, to the subsequent pages,
as we now explain.

\begin{definition}\label{def:r-homotopy}
    Let \(f,g\colon G \to H\) be maps of directed graphs, and let \(r \geq 0\). We say \demph{there is an \(r\)-homotopy from \(f\) to \(g\)}, and write \(f \rightsquigarrow_r g\), if every vertex $x$ of $G$ satisfies \(d(f(x),g(x)) \leq r\). We say \demph{there is a long homotopy from \(f\) to \(g\)}, and write \(f \rightsquigarrow_\infty g\), if every vertex $x$ of $G$ satisfies \(d(f(x),g(x)) < \infty\).
\end{definition}

Thus, for example, there is a \(1\)-homotopy from $f$ to $g$ if, for each $x$, either $f(x)$ and $g(x)$ are equal or there is a directed edge from the former to the latter. In general, \(r\)-homotopy (or long homotopy) is a condition on the pair of maps \(f\) and \(g\) which requires the existence of certain paths in \(H\), but does not demand us to make a particular choice of such paths. The relation \(\rightsquigarrow_r\) is not symmetric; nor, when \(r \neq 0, \infty\), is it transitive. However, Asao proved the following, which we state explicitly here
for future reference.

\begin{proposition}[{Asao~\cite[Theorem~1.3]{Asao-filtered}}]
\label{r-homotopy-invariance}
    If there is an \(r\)-homotopy from $f$ to $g$, then the induced maps 
    $E^s(f),E^s(g)\colon E^s_{\ast,\ast}(G)\to E^s_{\ast,\ast}(H)$ 
    are equal for $s\geqslant r+1$.
\end{proposition}

In particular, this says that the entirety of the $E^2$-term of the MPSS is invariant under $1$-homotopy. (See also~\cite[Proposition~5.7]{Asao-path} for this statement.) Meanwhile, reachability homology is invariant under long homotopy:

\begin{proposition}[{\cite[Theorem 4.4]{HepworthRoff2023}}]
    If there is a long homotopy from \(f\) to \(g\), then the induced maps \(\RH_\ast(f), \RH_\ast(g)\colon \RH_\ast(G) \to \RH_\ast(H)\) are equal.
\end{proposition}

The notion of \(r\)-homotopy and long homotopy give rise immediately to corresponding notions of \(r\)-homotopy and long homotopy \emph{equivalence} for pairs of directed graphs. In particular, one can define \(r\)-contractibility for any \(r \geq 0\). For this, note first that the terminal object in the category \(\DiGraph\) is the directed graph with a unique vertex; we denote it by \(\bullet\). Its MPSS is trivial: for every \(r\), the page \(E^r(\bullet)\) is concentrated in bidegree \((0,0)\), where it is given by a single copy of the ground ring \(R\).

\begin{definition}
    Directed graphs \(G\) and \(H\) are said to be \demph{\(r\)-homotopy equivalent} (respectively, \demph{long homotopy equivalent}) if there exist maps \(f\colon G \rightleftarrows H : g\) such that \(g \circ f\) is related to the identity on \(G\) by a zig-zag of \(r\)-homotopies (resp.~long homotopies) and \(f \circ g\) is related to the identity on \(H\) by a zig-zag of \(r\)-homotopies (resp.~long homotopies). A directed graph \(G\) is said to be \demph{\(r\)-contractible} if the terminal map \(G \to \bullet\) is part of an \(r\)-homotopy equivalence.
\end{definition}

\begin{corollary}
    If \(G\) and \(H\) are \(r\)-homotopy equivalent, then \(E^s(G) \cong E^s(H)\) for all \(s \geq r + 1\). \qed
\end{corollary}

For instance, if \(G\) has diameter \(r\) (meaning that \(r = \sup_{g,g' \in V(G)} d(g,g')\))
then \(G\) is \(r\)-contractible; it follows that its magnitude-path spectral sequence is trivial from page \(E^{r+1}\) onwards. For recent progress in the study of \(r\)-homotopy equivalence for directed graphs and metric spaces, we refer the reader to Ivanov \cite{SOIvanov}.

\subsection{The $E^\infty$-page}
Let us now consider the $E^\infty$-page of the MPSS.
The target of the MPSS is the homology of the reachability
chains from which it was constructed,
i.e.~the reachability homology $\RH_\ast(G)$.
In order to guarantee convergence,
let us assume that the graph $G$ has an
upper bound on its finite distances,
i.e.~that there is $K$ such that for each pair of vertices $x,y$
either $d(x,y)\leq K$ or $d(x,y)=\infty$.
Then the filtration of $\RC_\ast(G)$ is bounded:
in each degree $n$, $F_p\RC_n(G)$ vanishes for $p\leq n-1$,
and coincides with $\RC_n(G)$ for $p\geq n\cdot K$.
(See page~123 of~\cite{Weibel}.)
It follows that  in each bidegree $i,j$ 
the terms $E^r_{i,j}(G)$ are eventually independent of $r$,
and that their common value $E^\infty_{i,j}(G)$ is isomorphic
to the relevant filtration quotient of $\RH_\ast(G)$
in the filtration it inherits from the length filtration:
\[
    E^\infty_{i,j}(G)
    =
    F_i\RH_{i+j}(G)/F_{i-1}\RH_{i+j}(G).
\] 
Note that the above condition on $G$ always holds for
finite graphs.
However, it seems possible that 
other conditions on $G$ may guarantee that the filtration
is bounded, and certainly other conditions besides boundedness can
guarantee convergence of a spectral sequence.

\section{Bigraded path homology}\label{sec:BPH}

We saw in the last section that the $E^2$-page of the MPSS
of a directed graph $G$ contains the path homology $\PH_\ast(G)$
as its horizontal axis, and that the entire page has the same 
homotopy invariance property that $\PH_\ast(G)$ does,
namely invariance under $1$-homotopy.
This motivates the following definition.

\begin{definition}
    Let $G$ be a directed graph.
    The \emph{bigraded path homology} of $G$,
    denoted $\PH_{\ast,\ast}(G)$, is defined by
    \begin{equation}\label{eq-PH-MPSS}
        \PH_{k,\ell}(G) =  E^2_{\ell,k-\ell}(G)
    \end{equation}
    for all $k,\ell$,
    so that we have
    \[
        E^2_{i,j}(G)
        =
        \PH_{i+j,i}(G)
    \]
    in precise analogy with the relationship between $E^1$ 
    and $\MH$, and also
    \[
        \PH_{k,k}(G) = \PH_k(G).
    \]
    We may depict the former as in \Cref{fig:E2}. Just as the MPSS is functorial with respect to
    maps of graphs, so the same holds for the
    bigraded path homology.
\begin{figure}[h!]
    \centering
    \begin{tikzpicture}[baseline=0,xscale=2.5,yscale=1.5]
        \node[draw]() at (-1.2,0) {$E^2_{\ast,\ast}(G)$};
        \draw[-stealth, ultra thick,white!90!black] (-0.5,0) to (4.2,0);
        \draw[-stealth, ultra thick,white!90!black] (0,0.5) to (0,-4);
        \node (aa) at (0,0) {$\PH_{0,0}(G)$};
        \node (a) at (1,0) {$\PH_{1,1}(G)$};
        \node (b) at (1,-1) {$\PH_{0,1}(G)$};
        \node (c) at (2,0) {$\PH_{2,2}(G)$};
        \node (d) at (2,-1) {$\PH_{1,2}(G)$};
        \node (e) at (2,-2) {$\PH_{0,2}(G)$};
        \node (f) at (3,0) {$\PH_{3,3}(G)$};
        \node (g) at (3,-1) {$\PH_{2,3}(G)$};
        \node (h) at (3,-2) {$\PH_{1,3}(G)$};
        \node (i) at (3,-3) {$\PH_{0,3}(G)$};
        \node () at (3.8,-3.8) {$\ddots$};
        \node () at (3.8,0) {$\cdots$};
        \node () at (3.8,-1) {$\cdots$};
        \node () at (3.8,-2) {$\cdots$};
        \node () at (3.8,-3) {$\cdots$};
        \node[blue]() at (0,0.5) {$\scriptstyle 0$};
        \node[blue]() at (1,0.5) {$\scriptstyle 1$};
        \node[blue]() at (2,0.5) {$\scriptstyle 2$};
        \node[blue]() at (3,0.5) {$\scriptstyle 3$};
        \node[blue]() at (4,0.5) {$\scriptstyle 4$};
        \node[blue]() at (-0.5,0) {$\scriptstyle 0$};
        \node[blue]() at (-0.5,-1) {$\scriptstyle -1$};
        \node[blue]() at (-0.5,-2) {$\scriptstyle -2$};
        \node[blue]() at (-0.5,-3) {$\scriptstyle -3$};
        \node[blue]() at (-0.5,-4) {$\scriptstyle -4$};
    \end{tikzpicture}
    \caption{Page \(E^2\) of the MPSS is bigraded path homology.}
    \label{fig:E2}
\end{figure}
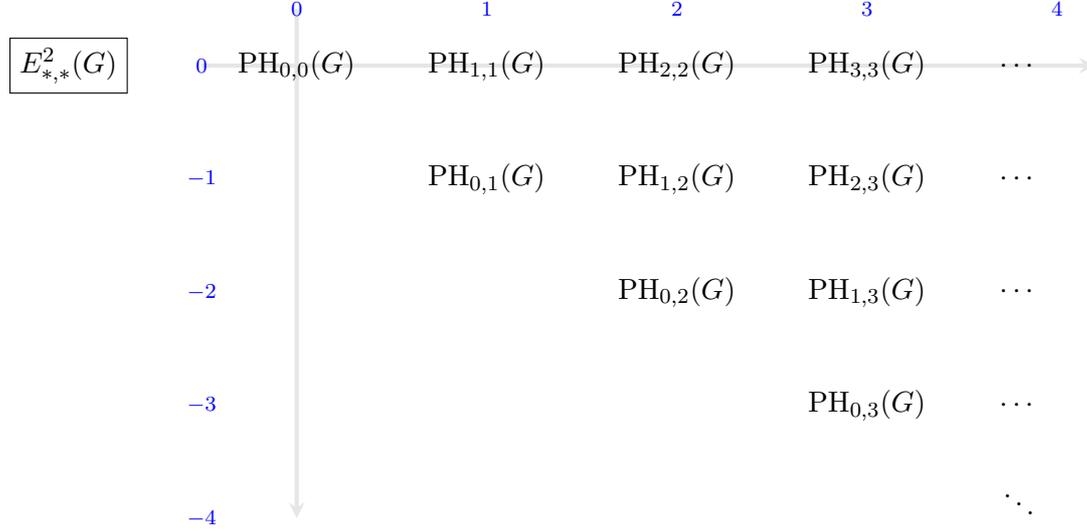
\end{definition}

We will see later in the paper that the bigraded path homology
groups satisfy many of the same formal properties as 
path homology, but that they contain strictly more information.

\begin{definition}
\label{def:relative_MPSS}
    Let $X$ be a directed graph and let $A$ be a subgraph of $X$.
    Then $\RC(A)$ is a subcomplex of $\RC(X)$,
    and we define the \emph{relative reachability chains}
    of the pair $(X,A)$ to be the quotient chain complex
    \[
        \RC(X,A) = \RC(X)/\RC(A).
    \]
    We equip this with the filtration inherited from the
    filtration on $\RC(X)$, so that
    $F_\ell\RC(X,A)$ is the image of $F_\ell\RC(X)$
    in $\RC(X,A)$.
    This results in the 
    \emph{relative magnitude path-spectral sequence}
    $\{E^r(X,A),d^r\}_{r\geq 0}$
    and associated \emph{magnitude chains},
    \emph{magnitude homology} and
    \emph{bigraded path homology} groups of the pair,
    defined by
    \begin{align*}
        \MC_{k,\ell}(X,A) &=  E^0_{\ell,k-\ell}(X,A)
        \\
        \MH_{k,\ell}(X,A) &=  E^1_{\ell,k-\ell}(X,A)
        \\
        \PH_{k,\ell}(X,A) &=  E^2_{\ell,k-\ell}(X,A)
    \end{align*}
    for all $k,\ell$.
\end{definition}

Recall that a subgraph \(A\) of \(X\) is said to be \demph{convex} if for every pair of vertices \(a,a'\) in \(A\) we have \(d_A(a,a') = d_X(a,a')\) \cite[Definition 4.2]{LeinsterGraph}.

\begin{theorem}[Exact sequences of a pair]
\label{thm:les-pair}
    Let $X$ be a graph and let $A$ be a subgraph of $X$.
    If $A$ is convex in $X$, then there is 
    a short exact sequence of magnitude chains: 
    \begin{equation}\label{eq-ses-mc-pair}
        0
        \to
        \MC(A)
        \to
        \MC(X)
        \to
        \MC(X,A)
        \to
        0
    \end{equation}
    Consequently there is a long exact sequence
    of magnitude homology groups:
    \begin{equation}\label{eq-les-mh-pair}
        \cdots
        \to
        \MH_{\ast,\ast}(A)
        \to
        \MH_{\ast,\ast}(X)
        \to
        \MH_{\ast,\ast}(X,A)
        \to
        \MH_{\ast-1,\ast}(X,A)
        \to
        \cdots
    \end{equation}
    If, in addition, there are no edges from $X\setminus A$ into $A$,
    then~\eqref{eq-ses-mc-pair} is split by a chain map,
    \eqref{eq-les-mh-pair} splits into short exact sequences,
    and we obtain a long exact sequence
    of bigraded path homology groups:
    \begin{equation}\label{eq-les-ph-pair}
        \cdots
        \to
        \PH_{\ast,\ast}(A)
        \to
        \PH_{\ast,\ast}(X)
        \to
        \PH_{\ast,\ast}(X,A)
        \to
        \PH_{\ast-1,\ast-1}(X,A)
        \to
        \cdots
    \end{equation}  
\end{theorem}

\begin{proof}
    We begin with the proof that if $A\subseteq X$ is convex,
    then we obtain the sequence~\eqref{eq-ses-mc-pair}.
    Convexity means that the length of a generator 
    of $\RC(A)$ does not depend on whether we regard it
    as a generator of $\RC(A)$ or $\RC(X)$.
    The result can now be proved directly from the 
    explicit description of magnitude chains given
    in~\Cref{subsection-E0}.
    Alternatively, recall that a map of filtered chain complexes
    $f\colon C\to D$ is called \emph{strict}
    if for each $\ell$ we have 
    $f(F_\ell C) = f(C)\cap F_\ell D$.
    Note that the latter condition is equivalent to
    $f^{-1}(F_\ell D) = F_\ell C + \ker(f)$.
    If $C\hookrightarrow D$ is a strict inclusion of 
    chain complexes, then one obtains a short
    exact sequence of filtration quotients:
    \[
        0
        \to
        \frac{F_\ell C }
        {F_{\ell-1} C }
        \to
        \frac{F_\ell D }
        {F_{\ell-1} D }
        \to
       \frac{ F_\ell\left(
             D / C 
        \right)}
        { F_{\ell-1}\left(
             D / C 
        \right)}
        \to
        0
    \]
    See~\cite[Section 0120]{Stacks} or~\cite[Lemme~1.1.9]{Deligne}.
    Our assumption ensures that the 
    inclusion map $\RC(A)\hookrightarrow \RC(X)$ is strict, 
    so the result follows from the last paragraph.

    Applying homology to the short exact 
    sequence~\eqref{eq-ses-mc-pair} now gives the long exact
    sequence~\eqref{eq-les-mh-pair}.

    Let us assume for the rest of the proof that
    there are no edges into $A$ from $X\setminus A$.
    This `no-entry' condition on $A$ means that the only vertices
    of $X$ that admit paths into $A$ are those that already lie in
    $A$.
    (In particular, the only paths in $X$ between vertices of $A$
    are those that lie wholly in $A$, so this condition alone ensures that 
    $A\subseteq X$ is convex.)
    
    We now show that~\eqref{eq-ses-mc-pair} is split.
    Observe that, by the no-entry condition, a generator of
    $\MC(X)$ lies in $\MC(A)$ if and only if its final entry 
    lies in $A$.
    Thus we may define a map $p\colon \MC(X)\to\MC(A)$ by 
    the following rule.
    \[
        p(x_0,\ldots,x_k)
        =
        \begin{cases}
            (x_0,\ldots,x_k) & \text{if }x_k\in A
            \\
            0 & \text{if }x_k\not\in A
        \end{cases}
    \]
    Then $p$ is a chain map because the differential of magnitude 
    chains sends a generator to a linear combination of generators
    with the same start and end points.
    And it is a splitting because it sends the 
    generators of $\MC(A)$ to themselves.

    So we have shown that our stronger assumption gives us a
    splitting of~\eqref{eq-ses-mc-pair}.
    It follows that the connecting maps of~\eqref{eq-les-mh-pair}
    vanish and~\eqref{eq-les-mh-pair} splits into short exact 
    sequences.
    The maps in these short exact sequences are obtained from maps
    of filtered complexes, and so they commute with the differential
    $d^1$ of the magnitude-path spectral sequence.
    In terms of magnitude chains, 
    the differential $d^1$ has the form
    $d^1\colon\MH_{\ast,\ast}(-)\to\MH_{\ast-1,\ast-1}(-)$,
    and consequently  we obtain the long exact 
    sequence~\eqref{eq-les-ph-pair} with the specified degree shifts.
\end{proof}


\section{Eilenberg--Zilber theorems}\label{sec:E-Z}

This section and the next concern the behaviour of the pages of the magnitude-path spectral sequence with respect to the box product of directed graphs.

\begin{definition}
\label{def:box}
The \demph{box product} of directed graphs \(G\) and \(H\) is the directed graph \(G \square H\) with vertex set \(V(G \square H) = V(G) \times V(H)\), and \(((g_1,h_1),(g_2,h_2)) \in E(G \square H)\) if either \(g_1 = g_2\) and \((h_1,h_2) \in E(H)\), or \(h_1 = h_2\) and \((g_1,g_2) \in E(G)\).
\end{definition}

In this section we will prove two Eilenberg--Zilber-style theorems. The first of these, \Cref{thm:filt_EZ}, says that for any directed graphs \(G\) and \(H\) the filtered chain complex \(\RC(G \square H)\) is naturally chain homotopy equivalent to the tensor product \(\RC(G) \otimes \RC(H)\) via filtration-preserving chain maps and chain homotopies. The second, \Cref{thm:EZ_box_SS}, relates the magnitude-path spectral sequence $E(G\square H)$ to the spectral sequences \(\MPSS(G)\) and \(\MPSS(H)\) of the factors. 

Before we can state those results formally and prove them (which we do in \Cref{sec:EZ_proofs}), we must collect some terminology and facts concerning \emph{pairings} and \emph{tensor products} of spectral sequences.

\subsection{Pairings and tensor products of spectral sequences}

First, let us briefly recall the relevant facts about the homology of tensor products of ordinary chain complexes. Given any chain complexes \(C\) and \(D\) over a ring \(R\), there is a map
\begin{equation}
\label{eq:homology_prod}
    \alpha\colon H_\ast(C) \otimes H_\ast(D) \to H_\ast(C \otimes D)
\end{equation}
determined by \([c] \otimes [d] \mapsto [c \otimes d]\). We shall refer to this map as the \demph{homology product}. If the ring \(R\) is a principal ideal domain (P.I.D.) and \(C\) happens to consist of flat \(R\)-modules, then the classical K\"unneth theorem for chain complexes says that the homology product fits into a short exact sequence, as follows.

\begin{theorem}[Algebraic K\"unneth theorem]
\label{thm:alg_kunneth}
    Let \(R\) be a P.I.D. and let \(C\) be a chain complex of flat \(R\)-modules. Then, given any chain complex \(D\) of \(R\)-modules, we have for each \(n \in \mathbb{N}\) a short exact sequence
    \begin{align*}
        0 \to \bigoplus_k H_k(C) \otimes H_{n-k}(D) \xto{\alpha} H_n(C \otimes D) \to \bigoplus_k \tor(H_k(C), H_{n-k-1}(D)) \to 0
    \end{align*}
    natural in \(D\) and with respect to chain maps \(C \to C'\) where \(C'\) is also flat. The sequence splits, but not naturally.
\end{theorem}

A proof of \Cref{thm:alg_kunneth} can be found in \cite[VI.3.3]{CartanEilenberg}. (The statement there is for hereditary rings, which includes the case of P.I.D.s \cite[p.13]{CartanEilenberg}.)

The algebraic K\"unneth theorem has the following corollary.

\begin{corollary}
\label{cor:alg_kunneth}
    If \(C\) and \(D\) are chain complexes over a field, then the homology product \(\alpha\colon H_\ast(C) \otimes H_\ast(D) \to H_\ast(C \otimes D)\) is an isomorphism. \qed
\end{corollary}

Now, fix a ground ring \(R\) (not necessarily a P.I.D.) and let \(E\) and \({}'E\) be any two spectral sequences of \(R\)-modules. For each \(r \geq 0\), we can form the tensor product \(E^r \otimes {}'E^r\) of bigraded \(R\)-modules---
\[(E^r \otimes {}'E^r)_{pq} = \bigoplus_{\substack{s+u=p \\ t+v = q}} E^r_{st} \otimes {}'E^r_{uv}\]
---and endow it with the differential \(d^r_\otimes(x \otimes y) = d_E^r(x) \otimes y + (-1)^{s+t} x \otimes d_{{}'E}^r(y)\). In general, the family \((E^r \otimes {}'E^r, d_\otimes^r)_{r \geq 0}\) need not be a spectral sequence, for, although the homology product always provides a natural map
\[
E^{r+1} \otimes {}'E^{r+1} \cong H_\ast(E^r) \otimes H_\ast ({}'E^r) \xto{\alpha} H_\ast(E^r \otimes {}'E^r),
\]
this need not be an isomorphism. If the ring \(R\) is a field, however, then \Cref{cor:alg_kunneth} implies we have a \demph{tensor product spectral sequence} \(\{E^r \otimes {}'E^r, d_\otimes^r\}\).

\begin{definition}
    Let \(E, {}'E\) and \({}''E\) be spectral sequences of \(R\)-modules. A \demph{pairing}
    \(\phi^\ast\colon (E,{}'E) \to {}''E\)
    is a sequence of maps of bigraded \(R\)-modules \(\phi^r\colon E^r \otimes {}'E^r \to {}''E^r\) with the following properties:
    \begin{enumerate}
        \item Each \(\phi^r\) is a chain map with respect to the differentials on page \(r\).
        \item For every \(r\) this diagram commutes:
        \[
        \begin{tikzcd}
        E^{r+1} \otimes {}'E^{r+1} \arrow{r}{\phi^{r+1}} \arrow[swap]{d}{\cong} & {}''E^{r+1} \\
        H_\ast(E^r) \otimes H_\ast({}'E^r) \arrow[swap]{d}{\alpha} \\
        H_\ast(E^r \otimes {}'E^r) \arrow[swap]{r}{H_\ast(\phi^r)} & H_\ast({}''E^r) \arrow[swap]{uu}{\cong}.
        \end{tikzcd}
        \]
    \end{enumerate}
    If the ground ring \(R\) is a field, a pairing is precisely a map of spectral sequences.
\end{definition}

\begin{definition}
    Let \(A\) and \(B\) be filtered chain complexes. Their \demph{(filtered) tensor product} is the chain complex \(A \otimes B\) equipped with the filtration in which
    \[F_\ell (A \otimes B) = \sum_{s+t=\ell} F_s A \otimes F_t B.\]
\end{definition}

A chain map \(f\colon A \otimes B \to C\) is filtered if and only if for every \(s\) and \(t\) we have \(f(F_sA \otimes F_tB) \subseteq F_{s+t}C\). Any such map induces a map
\[\bigoplus_{p+q=n} \frac{F_p A}{F_{p-1} A} \otimes \frac{F_q B}{F_{q-1} B} \to \frac{F_{n} C}{F_{n-1} C}\]
since \(f(F_{p-1}A \otimes F_q B) \subseteq F_{p+q-1} C\) and  \(f(F_{p}A \otimes F_{q-1} B) \subseteq F_{p+q-1} C \). It follows that there is an induced map \(\overline{f}\colon E^0(A) \otimes E^0(B) \to E^0(C)\). Indeed, \(f\) induces a map \(E^r(A) \otimes E^r(B) \to E^r(C)\) for every \(r\), and these maps comprise a pairing of spectral sequences. (For a detailed proof, see \cite[Lemma 3.5.2]{Helle}.)

\begin{lemma}\label{lem:helle}
Let \(A,B\) and \(C\) be filtered chain complexes over \(R\). Any filtered chain map \(f\colon A \otimes B \to C\) induces a pairing \(\phi\colon (E(A), E(B)) \to E(C)\) with \(\phi^0 = \overline{f}\). If \(R\) is a field, then \(f\) induces a map of spectral sequences \(\phi\colon E(A) \otimes E(B) \to E(C)\).
\qed
\end{lemma}

We are now equipped to state and prove the Eilenberg--Zilber theorems.

\subsection{The Eilenberg--Zilber theorems}
\label{sec:EZ_proofs}

In Theorem 5.1 of \cite{HepworthRoff2023}, 
the classical Eilenberg--Zilber theorem for simplicial sets is applied to prove 
that there is a chain homotopy equivalence
\[\nabla\colon \bpc(G) \otimes \bpc(H) \rightleftarrows \bpc(G \square H): \Delta.\]
Here we extend this to an equivalence of 
\emph{filtered} chain complexes, from which we will obtain, via \Cref{lem:helle}, an induced pairing of spectral sequences
\[\nabla^\ast\colon \MPSS(G) \otimes \MPSS(H) \to \MPSS(G \square H).\]

\begin{theorem}[Filtered Eilenberg--Zilber for the reachability chain complex]
\label{thm:filt_EZ}
    Let \(G\) and \(H\) be directed graphs. 
    Then $\RC(G)\otimes\RC(H)$ and $\RC(G\square H)$
    are chain homotopy equivalent, via maps and chain
    homotopies that are natural and that 
    respect the filtrations of the two sides.
    The chain homotopy equivalence
    \[\nabla\colon \bpc(G) \otimes \bpc(H) \to \bpc(G \square H)\]
    can be described explicitly on generators as follows:
    \begin{equation}\label{eq:filt_EZ}
        \nabla
        \left(
        (g_0,\ldots,g_p)
        \otimes
        (h_0,\ldots,h_q)
        \right)
        =
        \sum_\sigma
        \sign(\sigma)
        ((g_{i_0},h_{j_0}),\ldots,(g_{i_r},h_{j_r})).
    \end{equation}
    Here $r=k+k'$, and $\sigma$ runs over all sequences
    $((i_0,j_0),\ldots,(i_r,j_r))$ in which 
    $0\leq i_s\leq k$, $0\leq j_s\leq k'$,
    and in which each term $(i_{s+1},j_{s+1})$ 
    is obtained from $(i_s,j_s)$ by increasing 
    exactly one of the components by $1$. 
    The coefficient $\sign(\sigma)$ is defined to be $(-1)^n$ where $n$ is the number of pairs $(i,j)$ for which $i=i_k\implies j< j_k$.
\end{theorem}

In the statement of \Cref{thm:filt_EZ}, the sequences $\sigma$
can be regarded as the paths in $\mathbb{Z}^2$ from $(0,0)$
to $(k,k')$ that may only go upwards or to the right.
With this in mind, $\sign(\sigma)$ is $(-1)^n$
where $n$ is the number of lattice points 
that are on or above the $x$-axis,
and strictly below the path itself.
And if we regard each $(g_i,h_j)$ as being a label on lattice
point $(i,j)$, then the summand 
$((g_0,h_0),\ldots,(g_k,h_{k'}))$
associated to $\sigma$ is precisely the list of vertices
visited by the path.

To prove \Cref{thm:filt_EZ}, we will make use of a filtered variant of the classical Eilenberg--Zilber theorem. The classical theorem says that
we have, for any simplicial sets $A$ and $B$, 
the Eilenberg--Zilber map
$\nabla\colon NA\otimes NB\to N(A\times B)$ and the Alexander--Whitney map
$\Delta\colon N(A\times B)\to NA\otimes NB$, which satisfy \(\Delta \circ \nabla = \mathrm{Id}_{NA \otimes NB}\),
and a chain homotopy $\mathrm{SHI}$
between $\nabla\circ\Delta$ 
and $\mathrm{Id}_{N(A\times B)}$.
(For details, see Section~5 of~\cite{EM1953}
or Section~2 of~\cite{GonzalezDiazReal};
in particular the latter reference contains
explicit descriptions of 
$\Delta$, $\nabla$ and $\mathrm{SHI}$.) For present purposes, it is important to know that if \(A\) and \(B\) are filtered, then all these maps are automatically filtration-preserving.

\begin{lemma}\label{lemma-EZ-filtered}
    Let $A$ and $B$ be filtered simplicial sets. We equip their product $A\times B$ with the filtration given by $F_\ell (A\times B) = \bigcup_{i+j=\ell} F_iA\times F_jB$.
    Then the Alexander--Whitney map
    $\Delta\colon N(A\times B)\to NA\otimes NB$, the Eilenberg--Zilber map
    $\nabla\colon NA\otimes NB\to N(A\times B)$,
    and the chain homotopy $\mathrm{SHI}$
    between $\nabla\circ\Delta$ 
    and $\mathrm{Id}_{N(A\times B)}$
    are all filtration-preserving.
    Thus, $NA\otimes NB$ and $N(A\times B)$ 
    are chain homotopy equivalent as filtered chain
    complexes.
\end{lemma}

\begin{proof}
    Observe that 
    $F_p[NA\otimes NB]$ is the span of the images
    of the maps $N(F_iA)\otimes N(F_jB)\to N(A)\otimes N(B)$
    for $i+j=p$,
    and $F_p[N(A\times B)]$ is the span of the images
    of the maps $N(F_iA\times F_jB)\to N(A\otimes B)$
    for $i+j=p$.
    Thus we can prove that $\Delta$, say, preserves 
    filtrations by showing that the following
    diagram commutes:
    \[
    \begin{tikzcd}
        N(F_iA\times F_jB)
        \arrow{r}{\Delta}
        \arrow{d}
        &
        N(F_iA) \otimes N(F_jB)
        \arrow{d}
        \\
        N(A\times B)\arrow{r}{\Delta}
        &
        N(A) \otimes N(B).
    \end{tikzcd}
    \]
    But this is an immediate consequence of the
    naturality of $\Delta$.
    The proof that $\nabla$ and $\mathrm{SHI}$
    are filtration-preserving is similar.
\end{proof}

We now prove \Cref{thm:filt_EZ}. The proof makes use of the fact that the shortest path metric on the box product \(G \square H\) coincides with the \(\ell_1\)-metric on the product of the vertex sets:
\begin{equation}\label{eq:box_ell1}
    d_{G \square H}((g_0,h_0),(g_1,h_1)) = d_G(g_0,g_1) + d_H(h_0,h_1)
\end{equation}
for each \((g_1,h_1), (g_2, h_2) \in V(G) \times V(H)\). (See, for example, \cite[Corollary 1.3]{DIMZ}.) 
    
\begin{proof}[Proof of \Cref{thm:filt_EZ}]
    By \Cref{rmk:nerve} and \Cref{lemma-EZ-filtered}, there is a chain homotopy equivalence of filtered chain complexes
    \begin{equation}
    \label{eq:filt_EZ1}
        \RC(G) \otimes \RC(H) = N(\Nv(G))\otimes N(\Nv(H))\rightleftarrows N(\Nv(G)\times\Nv(H)).
    \end{equation}
    It remains to identify $N(\Nv(G)\times\Nv(H))$
    with $N(\Nv(G\square H)) = \RC(G \square H)$.
    
    In fact, \(\Nv(G)\times\Nv(H)\) and \(\Nv(G\square H)\) are isomorphic as filtered simplicial sets. This is Proposition 1.4 in \cite{DIMZ}, but for clarity we give some details here. The isomorphism of simplicial sets
    $\Nv(G)\times\Nv(H) \cong \Nv(G\square H)$,
    given by
    \[((g_0,\ldots,g_k),(h_0,\ldots,h_k))
    \leftrightarrow ((g_0,h_0),\ldots,(g_k,h_k)),\]
    was established in the proof of Theorem~5.1
    of~\cite{HepworthRoff2023};
    we just need to show that the tuple
    $((g_0,h_0),\ldots,(g_k,h_k))$
    lies in filtration $p$ if and only if
    the same is true of the pair 
    $((g_0,\ldots,g_k),(h_0,\ldots,h_k))$.
    But this follows from the fact that
    \[
        \ell((g_0,h_0),\ldots,(g_k,h_k))
        =
        \ell(g_0,\ldots,g_k)+\ell(h_0,\ldots,h_k),
    \]
    which holds since, by \eqref{eq:box_ell1}, we have
    \begin{align*}
        \ell\left(
            (g_0, h_0), \ldots, (g_k, h_k) 
        \right) 
        &= 
        \sum_{m=0}^{k-1} 
        d_{G \square H} 
        ((g_m,h_m),(g_{m+1},h_{m+1})) 
        \\
        & = 
        \sum_{m=0}^{k-1} 
        (d_G(g_m, g_{m+1}) 
        + d_H(h_m, h_{m+1})) 
        \\
        &= 
        \sum_{m=0}^{k-1} d_G(g_m, g_{m+1}) 
        + 
        \sum_{m=0}^{k-1} d_H(h_m, h_{m+1}) 
        \\
        &= 
        \ell(g_0, \ldots, g_k) + \ell(h_0, \ldots, h_k).
    \end{align*}
    Thus we have an isomorphism of filtered chain complexes
    \[N(\Nv(G) \times \Nv(H)) \cong N(\Nv(G \square H)) = \RC(G \square H),\]
    and this, combined with \eqref{eq:filt_EZ1}, proves the theorem.
\end{proof}

From \Cref{thm:filt_EZ} we can derive an Eilenberg--Zilber-type theorem for the magnitude-path spectral sequence of the box product, as follows. In \Cref{sec:kunneth} we will use this result to prove K\"unneth theorems for each page of the MPSS.

\begin{theorem}[Eilenberg--Zilber for the MPSS]\label{thm:EZ_box_SS}
    For any directed graphs \(G\) and \(H\) there is a pairing of spectral sequences
    \[
        \nabla^\ast\colon 
        (\MPSS(G), \MPSS(H)) 
        \longrightarrow 
        \MPSS(G \square H)
    \]
    which is natural in $G$ and $H$, and for
    which $\nabla^0$ is a chain homotopy equivalence.
\end{theorem}

\begin{proof}
    \Cref{thm:filt_EZ} gives us a map of filtered chain complexes
    \[\nabla\colon\RC(G)\otimes\RC(H)\to\RC(G\square H)\]
    natural in \(G\) and \(H\). From this, via~\Cref{lem:helle}, we obtain the
    pairing of spectral sequences
    \[
        \nabla^\ast\colon 
        (\MPSS(G), \MPSS(H)) 
        \longrightarrow 
        \MPSS(G \square H)
    \]
    in which the map $\nabla^0$ of the $E^0$-terms
    is the map of filtration quotients induced by $\nabla$. Since \(\nabla\)'s homotopy inverse \(\Delta\) and the chain homotopy \(\mathrm{SHI}\) are also filtration-preserving, they both descend to the filtration quotients, making \(\nabla^0\) a chain homotopy equivalence.
\end{proof}

\begin{remark}
\label{rmk:not_for_strong}
    Though the box product is sometimes referred to as the \demph{cartesian product} of directed graphs, it is not the categorical product in \(\DiGraph\). The categorical product of \(G\) and \(H\) is their \demph{strong product}: the directed graph \(G \squarediv H\) whose vertices are elements of \(V(G) \times V(H)\), with \(((g_1,h_1), (g_2,h_2)) \in E(G \squarediv H)\) if \(g_1 = g_2\) and \((h_1,h_2) \in E(H)\); or \(h_1 = h_2\) and \((g_1,g_2) \in E(G)\); or \((g_1,g_2) \in E(G)\) and \((h_1,h_2) \in E(H)\). The shortest path metric on \(G \squarediv H\) coincides with the \(\ell_\infty\)-metric on the product of the vertex sets:
    \[d_{G \squarediv H}((g_0,h_0),(g_1,h_1)) = \max \{d_G(g_0,g_1), d_H(h_0,h_1)\}\]
    for each \((g_0,h_0), (g_1,h_1) \in V(G) \times V(H)\).

    Just as in the case of the box product, there is an isomorphism of simplicial sets
    $\Nv(G)\times\Nv(H) \cong \Nv(G\squarediv H)$,
    given by
    \[((g_0,\ldots,g_k),(h_0,\ldots,h_k))
    \leftrightarrow ((g_0,h_0),\ldots,(g_k,h_k)).\]
    (This is part of Theorem~5.1 in~\cite{HepworthRoff2023}.) However, this is not an isomorphism of \emph{filtered} simplicial sets: while the function \(\Nv(G)\times\Nv(H) \to \Nv(G\squarediv H)\) is always filtration-preserving, its inverse usually is not.
\end{remark}


\section{K\"unneth theorems}\label{sec:kunneth}

In general, one obtains a K\"unneth theorem
for some homology theory by combining an appropriate 
Eilenberg--Zilber theorem with the classic algebraic
K\"unneth theorem for chain complexes.
The Eilenberg--Zilber theorem usually establishes a chain homotopy
equivalence, so is in some sense as good as can be hoped.
However, relating the homology of a tensor product of chain complexes with the
tensor product of the homologies entails loss of information, as quantified by the 
relevant $\tor$ term in the algebraic K\"unneth theorem.

Our Eilenberg--Zilber theorem for the magnitude-path spectral sequence (\Cref{thm:EZ_box_SS}) is, in this setting, as good as can be hoped:
a pairing of spectral sequences that is a chain homotopy equivalence on the initial
term.
However, in order to access, say, the $E^2$-term $E^2(G\square H)$, 
we need to take homology twice,
and therefore potentially twice encounter the discrepancies 
expressed by the $\tor$ terms.
It may be the case that there is a general framework for encapsulating
and understanding these cascading errors in the setting of a 
spectral sequence, but in the present paper our approach
is to make assumptions in order to ensure that no such cascades arise.

Our strongest K\"unneth theorem holds under the assumption that $R$ is a field.
In this case, the pairing \(\nabla^\ast\) 
in \Cref{thm:EZ_box_SS} is a map of spectral sequences. Since \(\nabla^0\) is a quasi-isomorphism, it follows that \(\nabla^r\) is an isomorphism
\[\MPSS^r(G) \otimes \MPSS^r(H) \xto{\cong} \MPSS^r(G \square H)\]
for every \(r \geq 1\). This gives the following result.

\begin{theorem}[K\"unneth theorem for the MPSS over a field]
\label{thm:kunneth_MPSS_field}
Fix a ground ring \(R\) which is a field. Then for every pair of directed graphs \(G\) and \(H\) there is a map of spectral sequences
\begin{equation*}
\MPSS(G) \otimes \MPSS(H) \to \MPSS(G \square H)
\end{equation*}
natural in \(G\) and \(H\) and consisting of isomorphisms from \(\MPSS^1\) onwards. \qed
\end{theorem}

In particular, \Cref{thm:kunneth_MPSS_field} gives us K\"unneth isomorphisms
for magnitude homology and bigraded path homology with 
coefficients in a field.
In the absence of the assumption that $R$ is a field,
we can still obtain K\"unneth theorems of the usual
form in magnitude homology and in the original
path homology, as we see in the next two results.

Specialized to undirected graphs, the following theorem recovers Theorem 5.3 of \cite{HepworthWillerton2017}. It is in turn a special case of the K\"unneth formula for the magnitude homology of generalized metric spaces \cite[Theorem 4.6]{Roff2023}, which extends that for classical metric spaces proved as Proposition 4.3 in \cite{BottinelliKaiser2021}.

\begin{theorem}[K\"unneth theorem for magnitude homology]
\label{thm:kunneth_box_MH}
Fix a ground ring \(R\) which is a P.I.D. For any directed graphs \(G\) and \(H\) there is short exact sequence
    \begin{equation}
    \label{eq:SES_MH}
    \begin{split}
    0 \to \bigoplus_{\substack{i+j=k \\a+b=\ell}} \MH_{i,a}(G) \otimes & \MH_{j,b}(H) \to \MH_{k,\ell}(G \square H) \\
    & \to \bigoplus_{\substack{i+j=k-1 \\a+b=\ell}} \mathrm{Tor}(\MH_{i,a}(G),\MH_{j,b}(H)) \to 0,
    \end{split}
    \end{equation}
natural in \(G\) and \(H\).
\end{theorem}

\begin{proof}
Apply the algebraic K\"unneth theorem (\Cref{thm:alg_kunneth})
to the tensor product chain complex $E^0(G)\otimes E^0(H)$.
Then use the quasi-isomorphism 
$\nabla^0\colon E^0(G)\otimes E^0(H)
\to E^0(G\square H)$ to replace the middle term---the result is a short exact sequence relating
$E^1(G)$, $E^1(H)$ and $E^1(G\square H)$.
Now use the identification $E_{p,q}^1(-) = \MH_{p+q,p}(-)$ 
to replace these with $\MH(G)$, $\MH(H)$ and $\MH(G\square H)$, and re-index appropriately.
\end{proof}

Using \Cref{thm:kunneth_box_MH} we can recover two known K\"unneth formulae for ordinary path homology with respect to the box product: \cite[Theorem 4.7]{PH-kunneth} and \cite[Theorem 9.5]{IvanovPav2022}. (The latter applies to more general objects than graphs.)

First, we record a fact that will be useful here and later.

\begin{lemma}
\label{lem:diag_free}
    Fix a ground ring \(R\) which is a P.I.D. Let \(G\) be any directed graph. For every \(k \in \mathbb{N}\), the \(R\)-module \(\MH_{kk}(G)\) is freely generated.
\end{lemma}

\begin{proof}
    Since \(\MC_{k+1,k}(G) = 0\) for every \(k\), we have \(\MH_{kk}(G) = \ker(\partial_{kk})\), which is a submodule of the free \(R\)-module \(\MC_{kk}(G)\).
\end{proof}

\begin{theorem}[K\"unneth theorem for ordinary path homology]\label{thm:kunneth_box_OPH}
Fix a ground ring \(R\) which is a P.I.D. For any directed graphs \(G\) and \(H\) there is a short exact sequence 
    \begin{equation}
    \label{eq:SES_OPH}
    \begin{split}
    0 \to \bigoplus_{i+j=k} \OPH_{i}(G) \otimes & \OPH_{j}(H) \to \OPH_{k}(G \square H) \\
    & \to \bigoplus_{i+j=k} \mathrm{Tor}(\OPH_{i}(G),\OPH_{j-1}(H)) \to 0,
    \end{split}
    \end{equation}
    natural in \(G\) and \(H\).
\end{theorem}

\begin{proof}
    Recall from \Cref{subsection:E2}
    that the ordinary path homology of a directed graph is the diagonal of its bigraded path homology or, equivalently, the horizontal boundary row on page \(E^2\) of the magnitude-path spectral sequence:
    \[\OPH_i(G) = \PH_{i,i}(G) = E^2_{i,0}(G).\]
    Applying the algebraic K\"unneth theorem in the horizontal boundary row of page \(E^1\) gives the short exact sequence
    \begin{equation*}
    \begin{split}
    0 \to \bigoplus_{i+j = k} \PH_{ii}(G) \otimes & \PH_{jj}(H) \to H_k \left( E_{\ast 0}^1(G) \otimes E_{\ast 0}^1(H) \right) \\
    & \to \bigoplus_{i+j=k} \mathrm{Tor}(\PH_{ii}(G),\PH_{j-1,j-1}(H)) \to 0;
    \end{split}
    \end{equation*}
    the claim is that the middle term is isomorphic to \(\PH_{kk}(G \square H)\). To see this, we can use the K\"unneth formula for magnitude homology.

    Consider the K\"unneth sequence for magnitude
    homology~\eqref{eq:SES_MH} in the case $k=\ell$.
    Since $\MH_{pq}(-)$ vanishes for $p>q$, 
    the first term
    $\bigoplus \MH_{ia}(G) \otimes \MH_{jb}(H)$
    reduces in this case to just the part involving diagonal terms, 
    i.e.~those terms where $a=i$ and $b=j$.
    For the same reason, the third term
    $\bigoplus
    \mathrm{Tor}(\MH_{ia}(G),\MH_{jb}(H))$
    reduces to just those terms in which $a=i$ and $b=j-1$,
    or $a=i-1$ and $b=j$.
    Thus, in all cases the $\tor$ term features 
    a diagonal group $\MH_{ii}(G)$ or $\MH_{jj}(H)$ 
    as one of its arguments.
    Since the diagonal magnitude homology modules are always free (\Cref{lem:diag_free}), it follows that the third term vanishes in this case,
    yielding for every $k$ an isomorphism
    \begin{equation*}
    \bigoplus_{i+j=k} \MH_{ii}(G) \otimes \MH_{jj}(H) 
    \xrightarrow[\cong]{\ \nabla^1\ }
    \MH_{kk}(G \square H).
    \end{equation*}
    Since \(\MH_{ii}(G) \otimes \MH_{jj}(H) = E^1_{i0}(G) \otimes E^1_{j0}(H)\), taking homology gives
    \[H_k \left( E^1_{\ast 0}(G) \otimes E^1_{\ast 0}(H) \right) \cong H_k \left( \MH_{\ast \ast} (G \square H) \right) = \PH_{kk}(G \square H),\]
    as claimed.
\end{proof}

Looking at the entire second page of the magnitude-path spectral sequence yields a K\"unneth formula for bigraded path homology.
However, to access this we need to make a flatness
assumption.

\begin{theorem}[K\"unneth theorem for bigraded path homology]
\label{thm:kunneth_box_bigrad}
Fix a ground ring \(R\) which is a P.I.D., and let \(G\) be a directed graph with flat magnitude homology. Then for any directed graph \(H\) there is a short exact sequence
    \begin{equation}
    \label{eq:SES_PH}
    \begin{split}
    0 \to \bigoplus_{\substack{i+j=k \\a+b=\ell}}  \PH_{i,a}(G) \otimes &\PH_{j,b}(H) \to \PH_{k, \ell}(G \square H) \\
    & \to \bigoplus_{\substack{i+j=k \\a+b=\ell}} \mathrm{Tor}(\PH_{i,a}(G),\PH_{j-1,b-1}(H)) \to 0,
    \end{split}
    \end{equation}
natural in \(H\) and with respect to maps \(G \to G'\) where \(G'\) also has flat magnitude homology. If \(R\) is a field then for \emph{every} pair of directed graphs \(G\) and \(H\) there is an isomorphism
    \begin{equation}
    \label{eq:kunneth_iso_PH}
    \PH_{k, \ell}(G \square H) \cong \bigoplus_{\substack{i+j=k \\a+b=\ell}} \PH_{i,a}(G) \otimes \PH_{j,b}(H),
    \end{equation}
natural in \(G\) and \(H\).
\end{theorem}

\begin{proof}
    Since \(G\) has flat magnitude homology---meaning that for every \(p,q\) the \(R\)-module \(\MPSS_{pq}^1(G)=\MH_{p+q,p}(G)\) is flat---we can apply the algebraic K\"unneth theorem in each row of \(\MPSS^{1}(G) \otimes \MPSS^{1}(H)\) to obtain, for each \(p\) and \(q\), a short exact sequence
    \begin{equation}
    \label{eq:SES_PH_partial}
    \begin{split}
    0 \to \left(\MPSS^{2}(G) \otimes \MPSS^{2}(H)\right)_{pq} & \xto{} H_p \left( (\MPSS^{1}(G) \otimes \MPSS^{1}(H) )_{\ast q} \right) \\
    &\to \bigoplus_{\substack{m+u=p \\ n+v = q}} \mathrm{Tor}(\MPSS^{2}_{mn}(G),\MPSS^{2}_{u-1,v}(H)) \to 0,
    \end{split}
    \end{equation}
    which has the claimed naturality in \(G\) and \(H\). 
    The assumption that \(G\) has flat magnitude homology ensures that in the statement of \Cref{thm:kunneth_box_MH} the \(\tor\) term vanishes, 
    so that using the identification  \(\MPSS_{pq}^1(-) = \MH_{p+q,p}(-)\)
    we obtain a natural isomorphism
    \[
    (\MPSS^{1}(G) \otimes \MPSS^{1}(H) )_{p,q}
    =
    (\MH(G) \otimes \MH(H))_{p+q,p} \cong \MH(G \square H)_{p+q,p}
    =\MPSS^1(G\square H)_{p,q}.
    \]
    Using this isomorphism we may replace the middle term of~\eqref{eq:SES_PH_partial} with $H_p(\MPSS^1(G\square H)_{\ast,q})
    =\MPSS^2_{p,q}(G\square H)$.
    Using the identification \(\MPSS_{pq}^2(-) = \PH_{p+q,p}(-)\) and reindexing appropriately, this yields the required short exact sequence \eqref{eq:SES_PH}.

    Over a field, the flatness assumption is guaranteed to be satisfied and the torsion term in \eqref{eq:SES_PH} vanishes, yielding the isomorphism \eqref{eq:kunneth_iso_PH}.
\end{proof}

Our final K\"unneth formula holds throughout the magnitude-path spectral sequence, though only under more restrictive flatness assumptions.

\begin{theorem}[K\"unneth theorem for the MPSS over a P.I.D.]
\label{thm:kunneth_box_SS}
    Fix a ground ring \(R\) which is a P.I.D. 
    Let $G$ and $H$ be directed graphs, and suppose that there is $s\geq 1$ such that,
    for each \(0 \leq r < s\), 
    each term of \(\MPSS^r(G)\) is flat. Then for \(1 \leq r < s\) the map \(\nabla^r\) is an isomorphism
    \[
        \bigoplus_{\substack{m+u=p \\ n+v = q}} 
        \MPSS^{r}_{mn}(G) \otimes \MPSS^{r}_{uv}(H) 
        \xto{\ \cong\ } 
        \MPSS^r_{pq}(G \square H),
    \]
    while \(\nabla^{s}\) fits into a short exact sequence
    \begin{align*}
    0 \to \bigoplus_{\substack{m+u=p \\ n+v = q}} \MPSS^{s}_{mn}(G) \otimes & \MPSS^{s}_{uv}(H) \xto{\nabla^{s}} \MPSS^{s}_{pq}(G \square H) \\
    &\to \bigoplus_{\substack{m+u=p \\ n+v = q}} \mathrm{Tor}(\MPSS^{s}_{{mn}}(G),\MPSS^{s}_{{u-s+1,v+s-2}}(H)) \to 0
    \end{align*}
    which is natural in \(H\) and with respect to maps \(G \to G'\) for $G'$ satisfying the same flatness property as \(G\).
\end{theorem}

\begin{proof}
    We prove this by induction on \(s\geq 1\). The case \(s=1\) is just the K\"unneth theorem for magnitude homology (\Cref{thm:kunneth_box_MH}); the case \(s=2\) follows from that theorem and the K\"unneth theorem for bigraded path homology (\Cref{thm:kunneth_box_bigrad}). Now we take \(s > 2\) and assume the statement holds for all \(r < s\).
    
    For \(1 \leq r < s\), the map on page \(r\) is the composite
    \[
        \begin{tikzcd}
        [
        ar symbol/.style = {draw=none,"\textstyle#1" description,sloped},
        isomorphic/.style = {ar symbol={\cong}},
        ]
        \MPSS^r(G) \otimes \MPSS^r(H) \arrow[swap]{d}{\cong} \arrow{r}{\nabla^r} & \MPSS^r(G \square H)\\
        H(\MPSS^{r-1}(G)) \otimes H(\MPSS^{r-1}(H)) \arrow[swap]{d}{\alpha} \\
        H(\MPSS^{r-1}(G) \otimes \MPSS^{r-1}(H)) \arrow[swap]{r}{H(\nabla^{r-1})} & H(\MPSS^{r-1}(G \square H)) 
        \arrow[swap]{uu}{\cong}
        \end{tikzcd}
    \]
    where \(\alpha\) is the homology product. By the assumptions of the theorem, \(\MPSS^{r-1}(G)\) and \(\MPSS^r(G)\) both consist of flat \(R\)-modules, so the algebraic K\"unneth theorem tells us that \(\alpha\) has the claimed naturality and is an isomorphism. By the inductive assumption, the map \(H(\nabla^{r-1})\) is an isomorphism and has the claimed naturality too. It follows that the same holds for \(\nabla^r\).
    
    Applying the algebraic K\"unneth theorem on page \(E^{s-1}\), to the chain complex lying along each line of slope \(-(s-1) / (s-2)\), yields for each \(p,q\) a short exact sequence
    \begin{align*}
    0 \to \left( \MPSS^{s}(G) \otimes \MPSS^{s} (H)  \right)_{pq} & \xto{\alpha} H \left( (\MPSS^{s-1} (G) \otimes \MPSS^{s-1} (H) )_{pq} \right) \\
    &\to \bigoplus_{\substack{m+u=p \\ n+v = q}} \mathrm{Tor}(\MPSS^{s}_{mn}(G),\MPSS^{s}_{{u-s+1,v+s-2}}(H)) \to 0,
    \end{align*}
    which has the claimed naturality. We may use the isomorphism \(H(\nabla^{s-1})\) to replace the middle term by \(\MPSS_{p,q}^{s}(G \square H)\), obtaining the short exact sequence
    \begin{align*}
    0 \to \left( \MPSS^{s}(G) \otimes \MPSS^{s} (H)  \right)_{pq} & \xto{H(\nabla^{s-1}) \circ \alpha} \MPSS_{pq}^{s}(G \square H) \\
    &\to \bigoplus_{\substack{m+u=p \\ n+v = q}} \mathrm{Tor}(\MPSS^{s}_{{mn}}(G),\MPSS^{s}_{{u-s+1,v+s-2}}(H)) \to 0.
    \end{align*}
    Since \(\nabla^s = H(\nabla^{s-1}) \circ \alpha\), this completes the proof.
\end{proof}


\section{Excision and Mayer--Vietoris theorems}\label{sec:BPH_MV}

This section contains three key results. We will prove an excision theorem which holds for every page of the magnitude-path spectral sequence from $\MPSS^1$ onwards (\Cref{excision-all}). From the excision theorem we are able to derive a Mayer--Vietoris theorem for magnitude homology (\Cref{Mayer--Vietoris-magnitude}) and for bigraded path homology (\Cref{Mayer--Vietoris-path}).

It is well known from the literature on 
both magnitude homology and path homology
that Mayer--Vietoris theorems do not hold for arbitrary unions of 
graphs.
Rather, one needs to assume that the union in question
is `nice' in some appropriate sense.
Indeed, here we follow~\cite{CDKOSW} in considering 
pushouts (rather than unions) along a class of subgraph inclusions termed \emph{cofibrations}. (This terminology will be justified in \Cref{sec:cofibration category}.)
Thus, we begin in~\Cref{sec:cofibrations} with a recollection
on cofibrations.
Then in~\Cref{sec:excision-Mayer--Vietoris} we state our main results,
leaving the lengthy proof of~\Cref{excision-all} 
to \Cref{sec:proof-excision}.

\subsection{Cofibrations}\label{sec:cofibrations}
The maps we call `cofibrations' are essentially the same as those defined in \cite[Definition 3.9]{CDKOSW}, except that we have reversed the directions of edges. This superficial modification and the reasons for it are explained in \Cref{rmk:cofib_reverse}. The definition runs as follows.

\begin{definition}
\label{def:reach}
    Let \(X\) be a directed graph, and \(A \subseteq X\) a subgraph. The \demph{reach} of \(A\), denoted \(rA\), is the induced subgraph of \(X\) on the set of all vertices that admit a path from some vertex in \(A\).
\end{definition}

\begin{definition}[{\cite[Definition~3.9]{CDKOSW}}]
\label{def:cofib}
    A \emph{cofibration} of directed graphs is an induced subgraph
    inclusion $A\hookrightarrow X$ for which:
    \begin{enumerate}
        \item \label{cond:cofib1}
        There are no edges from vertices not in $A$
        to vertices in $A$.
        \item \label{cond:cofib2}
        For each $x\in rA$ there is a vertex $\pi(x)\in A$
        with the property that
        \[
            d(a,x)=d(a,\pi(x))+ d(\pi(x),x)
            \qquad
            \text{for every }a\in A.
        \]
    \end{enumerate}
\end{definition}

Before stating the theorems, we make a few remarks on the definition. First, note that condition \eqref{cond:cofib1} is equivalent to saying that there are no paths in \(X\) from vertices outside \(A\) to vertices inside \(A\). This guarantees, in particular, that \(A\) is a convex subgraph of \(X\). Condition \eqref{cond:cofib2} says that \(X\) \demph{projects} to \(A\) in the sense of Leinster \cite[Definition 4.6]{LeinsterGraph} and its precursor (also Leinster) \cite[Definition~2.3.1]{LeinsterMetric2013}. As noted in \cite{LeinsterGraph}, the vertex \(\pi(x)\) is the vertex of \(A\) closest to \(x\), and this determines \(\pi(x)\) uniquely. In particular, taking \(x = a\) in condition \eqref{cond:cofib2} we see that \(\pi(a) = a\) for every \(a \in V(A)\). Thus, we have a \demph{projection function} \(\pi\colon V(rA) \to V(A)\). In general, though, \(\pi\) does not determine a map of graphs.

We are concerned in this section with the behaviour of cofibrations, and the homology of pushouts along cofibrations. The category \(\DiGraph\) is cocomplete, so in particular it has all pushouts; this is explained in~\cite{CDKOSW} before Lemma~2.10. Moreover, it is shown in \cite[Proposition~3.15]{CDKOSW} 
that the class of cofibrations is closed under pushout.
That is,
given a cofibration $i\colon A\to X$ and an arbitrary map of directed graphs 
$f\colon A\to Y$, the map $j$ in the pushout diagram
\begin{equation}\label{pushout}
\begin{tikzcd}
    A \arrow{r}{f} \arrow[hook,d,"i"'] & Y \arrow{d}{j} \\
    X \arrow[r,"g"'] & X \cup_A Y
\end{tikzcd}
\end{equation}
is also a cofibration.

\begin{example}[A cone]\label{cone-cofibration}
    Let \(I\) denote the following directed graph.
    \[
    \begin{tikzcd}
        -1 & 0 & +1 \\ [-20]
        \bullet\arrow{r} & \bullet & \bullet \arrow{l}
    \end{tikzcd}
    \]
    Given an arbitrary directed graph $X$, we define $CX$ to be the
    directed graph obtained from $X\square I$ by identifying the 
    induced subgraph $X\square\{+1\}$ to a single vertex
    that we denote simply $+1$.
    We think of $CX$ as a form of `cone' on $X$.
    We now identify $X$ with the induced subgraph of $CX$ on the
    vertices of form $(x,-1)$ for $x$ a vertex of $X$,
    so that we obtain the induced subgraph inclusion
    \[
        X\lhook\joinrel\longrightarrow CX.
    \]
    This is a cofibration.
    To see this, we verify the two properties of~\Cref{def:cofib}:
    \begin{enumerate}
        \item
        Since there is no edge of $I$ from $0$ to $-1$, 
        there are no edges of $CX$ from $CX\setminus X$ to $X$.
        \item
        The only paths in $I$ that begin at $-1$ are the trivial path
        and the edge $-1\to 0$,
        so that the reach $rX$ of $X$ is the induced subgraph on
        the vertices of the form $(x,0)$ and $(x,-1)$
        for $x$ a vertex of $X$.
        We then define $\pi$ on such vertices by
        \[
            \pi(x,-1) = \pi(x,0) = (x,-1).
        \]
        The required property of $\pi$ then states that,
        for vertices $x,y$ of $X$, and for $j=-1,0$,
        \[
            d((x,-1),(y,j))=d((x,-1),\pi(y,j))+d(\pi(y,j),(y,j))
        \]
        or in other words
        \[
            d((x,-1),(y,j))=d((x,-1),(y,-1))+d((y,-1),(y,j)),
        \]
        and this is immediately verified; see~\Cref{eq:box_ell1}.
    \end{enumerate}
    Note also that $CX$ is 1-contractible: the 
    identity map $CX\to CX$ is $1$-homotopic to the constant map
    with value $+1$.
    Indeed, if we denote by $d\colon CX\to CX$ the map defined
    by $d(+1)=+1$ and $d(x,j) = (x,\max\{j,0\})$ 
    for $x\in X$ and $j=0,-1$, 
    then one can check that there are $1$-homotopies 
    from both the identity map, and the constant map, to $d$.
    This example will be used later in~\Cref{cor:suspension}.
\end{example}

\begin{remark}\label{rmk:cofib_reverse}
    As the observations after \Cref{def:cofib} suggest, the definition of cofibrations in \cite{CDKOSW} is an adaptation, to the directed setting, of the notion of \demph{projecting decomposition} appearing in the literature on the magnitude and magnitude homology of undirected graphs and metric spaces \cite{LeinsterMetric2013, LeinsterGraph, HepworthWillerton2017}. Our \Cref{def:cofib} is taken directly from Definition~3.9 of~\cite{CDKOSW}, except that we have reversed the directionality: \(A \hookrightarrow X\) is a cofibration in the sense of this paper if and only if the corresponding map between the transpose graphs---in which the direction of every edge has been reversed---is a cofibration in the sense of \cite{CDKOSW}.

    This superficial change means that \Cref{def:cofib} can also be seen as a strengthening, suited to this paper's filtered techniques, of the notion of \demph{long cofibration} of directed graphs given in \cite[Definition 6.2]{HepworthRoff2023}: every cofibration in the sense of this paper is also a long cofibration. As is shown in \cite[Proposition 6.6]{HepworthRoff2023}, a long cofibration \(A \hookrightarrow X\) is precisely a map of directed graphs that induces a \demph{Dwyer morphism} between the preorders generated by \(A\) and \(X\). The class of Dwyer morphisms contains all cofibrations in the Thomason model structure on the category of small categories, which is Quillen equivalent to the classical model structure on the category of simplicial sets \cite{Thomason1980,Cisinski}. Thus, the theory being developed here and in \Cref{sec:cofibration category} bears a close relationship to more classical homotopy theoretic constructions. The details of that relationship remain to be explored.
\end{remark}

\subsection{The excision and Mayer--Vietoris theorems}
\label{sec:excision-Mayer--Vietoris}

We can now state our excision and Mayer--Vietoris theorems,
which apply in the context of a pushout along a cofibration.
Our Mayer--Vietoris theorems, which hold for magnitude homology
and bigraded path homology, are fairly direct consequences of 
the excision theorem and so we prove them within this subsection.
The excision theorem, on the other hand, applies to every page
of the MPSS, but its proof is intricate
and so is deferred to the next subsection. 

Recall from \Cref{def:relative_MPSS} that, given any subgraph \(A \subseteq X\), one can consider the relative magnitude-path spectral sequence \(\{E^r_{\ast,\ast}(X,A), d^r\}_{r \geq 0}\). Given a pushout of the form in \eqref{pushout}, our excision theorem says that the relative magnitude-path spectral sequence of the pair \((X,A)\) coincides with that of the pair \((X \cup_A Y, Y)\).

\begin{theorem}[Excision in the magnitude-path spectral sequence]
\label{excision-all}
    Let $i\colon A\to X$ be a cofibration and let 
    $f\colon A\to Y$ be an arbitrary map of directed graphs,
    so that we obtain the pushout diagram of the form in~\eqref{pushout}.
    Then for all $r\geq 1$ the induced map
    \[
        E^r_{\ast,\ast}(X,A)
        \xrightarrow{\ \ \cong\ \ }
        E^r_{\ast,\ast}(X\cup_A Y,Y)
    \]
    is an isomorphism.
    In particular, on magnitude homology and bigraded path
    homology we have
    \[
        \MH_{\ast,\ast}(X,A)
        \xrightarrow{\ \ \cong\ \ }
        \MH_{\ast,\ast}(X\cup_A Y,Y)
    \]
    and 
    \[
        \PH_{\ast,\ast}(X,A)
        \xrightarrow{\ \ \cong\ \ }
        \PH_{\ast,\ast}(X\cup_A Y,Y).
    \]
\end{theorem}

In the classical topological setting, 
the excision theorem for spaces gives rise to the 
Mayer--Vietoris sequence by `stitching together', 
using a standard argument of homological algebra, 
the long exact sequences
of pairs of the form $(X,A)$ and $(X\cup_A Y,Y)$.
The same reasoning can be applied here 
to either the short exact sequences
of magnitude homology groups, 
or the long exact sequences of bigraded path homology groups,
that are associated to a pair by~\Cref{thm:les-pair}.
That is how we obtain the next two theorems, 
which give a short exact Mayer--Vietoris sequence for magnitude homology,
and a long exact Mayer--Vietoris sequence for path homology.

However, we are not at this time able to give a Mayer--Vietoris 
result for the $E^r$-pages for any $r\geq 3$.
Indeed, the snake lemma allows us to 
pass from short exact sequences of chain complexes to
long exact sequences of homology groups,
and this is how we pass from $E^1$ to $E^2$ 
in the proof of~\Cref{thm:les-pair}.
But we do not know of any account of the structure obtained
from \emph{long} exact sequences of chain complexes by taking homology,
and this prevents us from obtaining results of Mayer--Vietoris
type on the later pages of the MPSS.

We now state the theorems formally, before proceeding to the proofs.

\begin{theorem}[Mayer--Vietoris for 
magnitude homology of directed graphs]\label{Mayer--Vietoris-magnitude}
    Suppose given a cofibration $i\colon A\to X$
    and any map $f\colon A\to Y$, so that we have a pushout of the form in~\eqref{pushout}.
    Then we have the \demph{Mayer--Vietoris sequence}
    in magnitude homology of directed graphs, 
    in the form of a short  exact sequence:
    \begin{equation}\label{eq:MH_MV}
        0
        \to
        \MH_{\ast,\ast}(A)
        \xrightarrow{(i_\ast,-f_\ast)}
        \MH_{\ast,\ast}(X)\oplus \MH_{\ast,\ast}(Y)
        \xrightarrow{g_\ast\oplus j_\ast}
        \MH_{\ast,\ast} (X\cup_A Y)
        \to
        0
    \end{equation}
    This sequence is split.
\end{theorem}

\begin{remark}
\label{rmk:MV_comparison}
This theorem can be compared with  the Mayer--Vietoris theorem for magnitude homology of \emph{undirected} graphs that appeared as Theorem~6.6 in Hepworth and Willerton~\cite{HepworthWillerton2017}. The result for undirected graphs applies in the presence of a projecting decomposition; we refer the reader to \cite[Definition 6.3]{HepworthWillerton2017} for the definition.

The category of undirected graphs embeds canonically into \(\DiGraph\) by taking a graph \(G\) and replacing each of its edges \(\{x,y\}\) by a pair of directed edges \(x \rightleftarrows y\). However, given a projecting decomposition \(G = X \cup_A Y\), it is almost never the case that the inclusion of \(A\) into \(X\) (or into \(Y\)) induces, under this operation, a cofibration in the sense of \Cref{def:cofib}. Indeed, that occurs only when \(X\) (respectively, \(Y\)) is the disjoint union of \(A\) and its complement, in which case the split exact sequence \eqref{eq:MH_MV} holds trivially for \(\MH_{\ast,\ast}(X \cup_A Y)\). On the other hand, there are many non-trivial examples of projecting decompositions of undirected graphs; see, for example, Corollaries 4.13 and 4.14 of Leinster \cite{LeinsterGraph}. Thus, our Mayer--Vietoris theorem and that in~\cite{HepworthWillerton2017} are quite independent.
\end{remark}

\begin{theorem}[Mayer--Vietoris for 
bigraded path homology]\label{Mayer--Vietoris-path}
    Suppose given a cofibration $i\colon A\to X$
    and any map $f\colon A\to Y$, so that we have a pushout
    of the form in~\eqref{pushout}.
    Then we have the \demph{Mayer--Vietoris sequence}
    in bigraded path homology, meaning that there is a
    long exact sequence:
    \[
        \cdots
        \to
        \PH_{\ast,\ast}(A)
        \xrightarrow{(i_\ast,-f_\ast)}
        \PH_{\ast,\ast}(X)\oplus\PH_\ast(Y)
        \xrightarrow{g_\ast\oplus j_\ast}
        \PH_{\ast,\ast} (X\cup_A Y)
        \xrightarrow{\partial_\ast}
        \PH_{\ast-1,\ast}(A)
        \to
        \cdots
    \]
\end{theorem}

Theorem~\ref{excision-all} is the hardest to prove of the three 
theorems in this section, and we therefore leave its
proof to the end of the section.
The proofs of the other two theorems are 
straightforward once we have excision, so we tackle them first.

\begin{proof}[Proof of Theorem~\ref{Mayer--Vietoris-magnitude},
assuming Theorem~\ref{excision-all}]
    Extracting long exact sequences of Mayer--Vietoris type 
    from excision theorems
    is common in the algebraic topology literature.
    Indeed, if we consider the commutative diagram 
    \[
    \begin{tikzcd}
        \cdots
        \arrow{r}{}
        &
        \MH_{\ast,\ast}(A)
        \arrow{r}{i}
        \arrow{d}{f}
        &
        \MH_{\ast,\ast}(X)
        \arrow{r}{}
        \arrow{d}{g}
        &
        \MH_{\ast,\ast}(X,A)
        \arrow{r}{}
        \arrow{d}{\cong}
        &
        \cdots
        \\
        \cdots
        \arrow{r}{}
        &
        \MH_{\ast,\ast}(Y)
        \arrow[swap]{r}{j}
        &
        \MH_{\ast,\ast}(X\cup_A Y)
        \arrow{r}{}
        &
        \MH_{\ast,\ast}(X\cup_A Y,Y)
        \arrow{r}{}
        &
        \cdots
    \end{tikzcd}
    \]
    whose rows are obtained from \Cref{thm:les-pair},
    and identify every pair of third terms using the excision 
    isomorphism, 
    then we may apply the result of exercise~38 of~\cite[p.159]{Hatcher}
    to obtain the following long exact sequence:
    \[
    \begin{tikzcd}
        \cdots
        \arrow{r}{}
        &
        \MH_{\ast,\ast}(A)
        \arrow{r}{(i_\ast,-f_\ast)}
        &
        \MH_{\ast,\ast}(X)
        \oplus
        \MH_{\ast,\ast}(Y)
        \arrow{r}{g_\ast\oplus j_\ast}
        &
        \MH_{\ast,\ast}(X\cup_A Y)
        \arrow{r}{}
        &
        \cdots
    \end{tikzcd}
    \]
    Since the inclusion $A\subseteq X$ satisfies the condition
    that there are no edges from $X\setminus A$ to $A$,
    \Cref{thm:les-pair} shows that $i_\ast$ is split, 
    and it follows that $(i_\ast,-f_\ast)$ is also split,
    so that we obtain the short exact sequence of the claim.
\end{proof}

\begin{proof}[Proof of Theorem~\ref{Mayer--Vietoris-path},
assuming Theorem~\ref{Mayer--Vietoris-magnitude}]
    This is similar to the proof of~\Cref{Mayer--Vietoris-magnitude},
    except that one uses the long exact sequences 
    of bigraded path homology groups obtained from~\Cref{thm:les-pair}.
    We leave the details to the reader.
\end{proof}

Let us give here an immediate application of our excision theorem. A further application will be seen in the more substantial computation in \Cref{sec:mn_cycles}.

\begin{definition}
    Let $X$ be a directed graph.
    We define $SX$ to be the graph obtained from $X$ by 
    adding two new vertices, $+1$ and $-1$, together with an
    edge from each of the new vertices to each vertex of $X$.
    Thus $SX$ is a form of `unreduced suspension' of $X$; see \Cref{eg:spheres} for an illustrative example.
\end{definition}

The following result can be compared with Proposition~5.10
of Grigor'yan \textit{et al}~\cite{PH-path-complexes}.

\begin{corollary}[A suspension theorem]
\label{cor:suspension}
    Let $X$ be a nonempty directed graph, and
    $k,\ell\geq 0$.
    Then there is a natural isomorphism
    \[
        \ker\bigl(\PH_{k,\ell}(SX)\to\PH_{k,\ell}(\bullet)\bigr)
        \cong
        \ker\bigl(\PH_{k-1,\ell-1}(X)\to\PH_{k-1,\ell-1}(\bullet)\bigr).
    \]
    In particular,
    $\PH_{k,\ell}(SX)$ and $\PH_{k-1,\ell}(X)$ are 
    naturally isomorphic unless $k=\ell=0$ or $k=\ell=1$,
    in which case they differ only by a single summand of $R$.
\end{corollary}

Note that the kernels appearing in the statement could be called
\emph{reduced} bigraded path homology groups by analogy with 
algebraic topology.
Since reduced groups do not appear elsewhere in the paper,
we restrain ourselves from defining these in general.

\begin{proof}
    Recall from~\Cref{cone-cofibration} the graph
    $I$  given by $-1\to 0\leftarrow +1$,
    and the graph $CX$ obtained from $X\square I$ by identifying
    $X\square\{+1\}$ to a single vertex $+1$.
    We  showed that the inclusion $X\hookrightarrow CX$ 
    that identifies $X$ with the subgraph $X\square\{-1\}$ 
    is a cofibration.
    Now define a map $CX\to SX$ sending $X\square\{0\}$ to $X$,
    $X\square\{+1\}$ to $+1$, and $X\square\{-1\}$ to $-1$.
    Then we have a pushout square
    \[
    \begin{tikzcd}
        X \arrow{r}{} \arrow[hook]{d} & \bullet \arrow[hook]{d}{} \\
        CX \arrow[r] & SX
    \end{tikzcd}
    \]
    where the right-hand map includes the one-vertex graph
    as the subgraph with single vertex $-1$.
    Excision gives an isomorphism
    \[
        \PH_{\ast,\ast}(SX,\bullet)
        \cong
        \PH_{\ast,\ast}(CX,X).
    \]
    Next, the map $\PH_{k,\ell}(SX)\to\PH_{k,\ell}(SX,\bullet)$
    induces an isomorphism
    \[
        \ker\bigl(\PH_{k,\ell}(SX)\to\PH_{k,\ell}(\bullet)\bigr)
        \cong
        \PH_{k,\ell}(SX,\bullet).
    \]
    This is seen by using the map of pairs 
    $(SX,\bullet)\to(\bullet,\bullet)$
    to compare the associated long exact sequences.
    Finally,
    the connecting morphism $\PH_{k,\ell}(CX,X)\to\PH_{k-1,\ell-1}(X)$
    induces an isomorphism
    \[     
        \PH_{k,\ell}(CX,X)
        \cong
        \ker\bigl(\PH_{k-1,\ell-1}(X)\to\PH_{k-1,\ell}(\bullet)\bigr),
    \]
    as one sees using the fact that $CX\to\bullet$ induces an 
    isomorphism in path homology thanks to $1$-contractibility
    of $CX$.
    The three isomorphisms obtained in this paragraph are natural
    in $X$ and, combined, they complete the proof.
\end{proof}

\begin{example}[A family of spheres]\label{eg:spheres}
    Denote the empty graph by \(\emptyset\), and for each \(n \geq 0\) let \(\mathbb{S}^n\) denote the \((n+1)\)-fold suspension \(S^{n+1} \emptyset\). The directed graph \(\mathbb{S}^n\) is then analogous to an `\(n\)-sphere'.
    (Indeed, it is the face poset of the regular CW decomposition of 
    the topological $n$-sphere by hemispheres.)
    Here we depict, from left to right, \(\mathbb{S}^0\), \(\mathbb{S}^1\) and \(\mathbb{S}^2\), with the new vertices \(\pm 1\) labelled in each case:
    
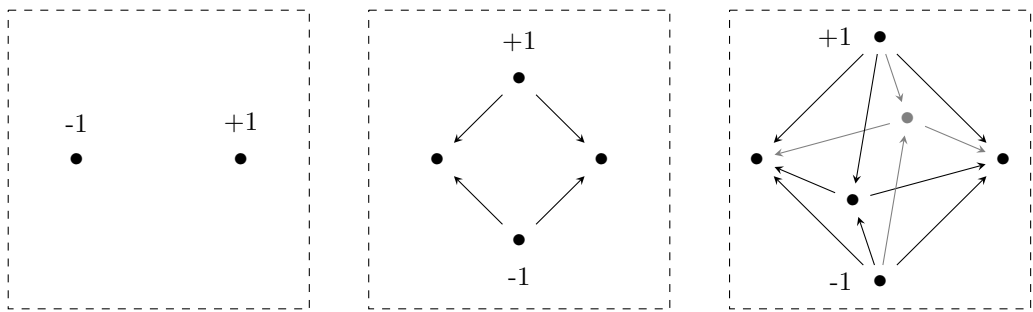
\[
    \begin{tikzpicture}[scale=.9,baseline=0]
        \draw[dashed] (-2.2,2.2) -- (2.2,2.2) -- (2.2,-2.2) -- (-2.2,-2.2) -- (-2.2,2.2);
    
        \node [label={[label distance=.1]90:\small{+1}}] (a) at (0:1.2) {$\bullet$};
        \node [label={[label distance=.1]90:\small{-1}}] (b) at (180:1.2) {$\bullet$};
    \end{tikzpicture}
    \qquad
    \begin{tikzpicture}[scale=.9,baseline=0]
        \draw[dashed] (-2.2,2.2) -- (2.2,2.2) -- (2.2,-2.2) -- (-2.2,-2.2) -- (-2.2,2.2);
        
        \node (a) at (0:1.2) {$\bullet$};
        \node [label={[label distance=.1]90:\small{+1}}] (b) at (90:1.2) {$\bullet$};
        \node (c) at (180:1.2) {$\bullet$};
        \node [label={[label distance=.1]270:\small{-1}}] (d) at (270:1.2) {$\bullet$};

        \draw[stealth-] (a) -- (b);
        \draw[stealth-] (c) -- (b);
        \draw[stealth-] (a) -- (d);
        \draw[stealth-] (c) -- (d);
    \end{tikzpicture}
    \qquad
    \begin{tikzpicture}[scale=.9,baseline=0]
        \draw[dashed] (-2.2,2.2) -- (2.2,2.2) -- (2.2,-2.2) -- (-2.2,-2.2) -- (-2.2,2.2);
        
        \node (a) at (0:1.8) {$\bullet$};
        \node [gray] (b) at (.4,.6) {$\bullet$};
        \node (c) at (180:1.8) {$\bullet$};
        \node (d) at (-.4,-.6) {$\bullet$};
        \node [label={[label distance=.1]180:\small{+1}}] (e) at (90:1.8) {$\bullet$};
        \node [label={[label distance=.1]180:\small{-1}}] (f) at (270:1.8) {$\bullet$};

        \draw[gray, stealth-] (a) -- (b);
        \draw[gray, stealth-] (c) -- (b);
        \draw[stealth-] (a) -- (d);
        \draw[stealth-] (c) -- (d);
        
        \draw[stealth-] (a) -- (e);
        \draw[gray, stealth-] (b) -- (e);
        \draw[stealth-] (c) -- (e);
        \draw[stealth-] (d) -- (e);

        \draw[stealth-] (a) -- (f);
        \draw[gray, stealth-] (b) -- (f);
        \draw[stealth-] (c) -- (f);
        \draw[stealth-] (d) -- (f);
    \end{tikzpicture}
    \]
    The bigraded path homology of \(\mathbb{S}^0\) is concentrated in bidegree \((0,0)\), where it is given by \(R \oplus R\), so that 
    $\ker\bigl(\PH_{\ast,\ast}(\mathbb{S}^0)
    \to\PH_{\ast,\ast}(\bullet)\bigr)$
    is a single copy of $R$ concentrated in bidegree $(0,0)$.
    Applying \Cref{cor:suspension} and inducting on \(n\), 
    one finds that, for $n\geq 0$, 
    $\ker\bigl(\PH_{\ast,\ast}(\mathbb{S}^n)
    \to\PH_{\ast,\ast}(\bullet)\bigr)$
    is a single copy of $R$ concentrated in bidegree $(n,n)$.
    It then follows, for $n>0$,
    that \(\PH_{k, \ell}(\mathbb{S}^n) = 0\) when \(k \neq \ell\), while 
    \[
    \PH_{k,k}(\mathbb{S}^n) \cong
    \begin{cases}
        R & \text{if } k = 0, n\\
        0 & \text{otherwise}.
    \end{cases}
    \]
    That is, for every \(n \geq 0\), the bigraded path homology of \(\mathbb{S}^n\) is concentrated on the diagonal, where it consists of a copy of the singular homology of the topological \(n\)-sphere.
\end{example}

\subsection{Proof of the excision theorem}
\label{sec:proof-excision}

We now embark on the proof of the excision theorem,
\Cref{excision-all}.
Our proof owes a significant debt to the proof in 
Section~9 of~\cite{HepworthWillerton2017},
and indeed Lemma~\ref{A-acyclic} is closely related to
Lemma~9.2 of~\cite{HepworthWillerton2017}.
However, the overall structure of the proof here is essentially
different from that one, in order to cope with the
fact that the map $f\colon A\to Y$ appearing in our pushout
square~\eqref{pushout} is not necessarily the inclusion
of an induced subgraph. 

The proof of~\Cref{excision-all} is rather intricate, 
so we offer a sketch here. The objective is to prove that
\[
    E^r_{\ast,\ast}(X,A)\to E^r_{\ast,\ast}(X\cup_A Y,Y)
\]
is an isomorphism for all $r\geq 1$.
Since an isomorphism of chain complexes induces an isomorphism
on homology, 
it is sufficient to prove that this is an isomorphism when $r=1$,
and this, by the definition of magnitude homology of a pair,
is equivalent to showing that the map
\[
    \frac{\MC_{\ast,\ast}(X)}{\MC_{\ast,\ast}(A)}
    \longrightarrow
    \frac{\MC_{\ast,\ast}(X\cup_A Y)}{\MC_{\ast,\ast}(Y)}
\]
is a quasi-isomorphism.
The map $X\to X\cup_A Y$ induces a map
$X\setminus A\to (X\cup_A Y)\setminus Y$
that is in fact an isomorphism of directed graphs.
By the five-lemma it is therefore sufficient to show that the  map
\[
    \frac{\MC_{\ast,\ast}(X)}
    {\MC_{\ast,\ast}(A)+\MC_{\ast,\ast}(X\setminus A)}
    \longrightarrow
    \frac{\MC_{\ast,\ast}(X\cup_A Y)}{\MC_{\ast,\ast}(Y)
    +\MC_{\ast,\ast}((X\cup_A Y)\setminus Y)}
\]
is a quasi-isomorphism.
In order to prove this, we equip the domain and codomain with
compatible filtrations
\[
    0= F_0\subseteq F_1\subseteq \cdots\subseteq F_{\ell-1}=
    \frac{\MC_{\ast,\ast}(X)}
    {\MC_{\ast,\ast}(A)+\MC_{\ast,\ast}(X\setminus A)}
\]
and 
\[
    0=F'_0\subseteq F'_1\subseteq \cdots\subseteq F'_{\ell-1}=
    \frac{\MC_{\ast,\ast}(X\cup_A Y)}{\MC_{\ast,\ast}(Y)
    +\MC_{\ast,\ast}((X\cup_A Y)\setminus Y)},
\]
so that we are reduced to proving that the induced maps of filtration
quotients
\begin{equation}\label{eq:filtration-quotients}
    F_i/F_{i-1}\longrightarrow F'_i/F'_{i-1}
\end{equation}
are all quasi-isomorphisms.
It turns out that each of these filtration quotients has a decomposition 
as a direct sum of suspensions of more elementary complexes
that we denote 
\[
    A_{\ast,\ell}(-,-)
    \qquad\text{and}\qquad
    A'_{\ast,\ell}(-,-)
\]
respectively.
These complexes all have very 
simple homology that can be computed directly.
With that computation in hand, it is then possible 
to  verify directly that 
\eqref{eq:filtration-quotients} is an isomorphism on homology,
which completes the proof.

We now embark on the proof in detail.

\begin{definition}\label{definition-A}
    Let $A\hookrightarrow X$ be a cofibration,
    let $\ell\geq 0$, let $x\in X\setminus A$
    and let $a\in A$.
    Define $A_{\ast,\ell}(a,x)$ to be the subcomplex of
    $\MC_{\ast,\ell}(X)$ spanned by those tuples
    $(x_0,\ldots,x_k)$ 
    for which $x_0=a$, $x_k=x$ and $x_1,\ldots,x_{k-1}\in A$.
\end{definition}

\begin{lemma}\label{A-not-acyclic}
    Suppose we are in the situation of 
    Definition~\ref{definition-A},
    and that $a=\pi(x)$ and $\ell=d(a,x)$.
    Then the homology of $A_{\ast,\ell}(a,x)$ is a
    single copy of $R$ in degree $1$,
    generated by the homology class of the tuple $(\pi(x),x)$.
\end{lemma}

\begin{proof}
    Any generator of $A_{k,\ell}(\pi(x),x)$ has form 
    $(\pi(x),x_1,\ldots,x_{k-1},x)$.
    The length of such a tuple  satisfies 
    \begin{align*}
        d(\pi(x),x)
        &=
        \ell(\pi(x),x_1,\ldots,x_{k-1},x)
        \\
        &\geq 
        d(\pi(x),x_i)+d(x_i,x)
        \\
        &=
        d(\pi(x),x_i)+d(x_i,\pi(x))+d(\pi(x),x)
    \end{align*}
    where the second line follows from the triangle inequality,
    and the third follows from the second property 
    of a cofibration,
    using the fact that $x_i\in A$.
    It follows that $d(\pi(x),x_i)=d(x_i,\pi(x))=0$, so that 
    $x_i=\pi(x)$. 
    This can only happen if $k=1$, and the  result follows. 
\end{proof}
    
\begin{lemma}\label{A-acyclic}
    Suppose we are in the situation of 
    Definition~\ref{definition-A},
    and that at least one of $a=\pi(x)$ and $\ell=d(a,x)$ fails.
    Then $A_{\ast,\ell}(a,x)$ is acyclic.
\end{lemma}

\begin{proof}
    Define a map
    \[
    	s\colon A_{\ast,\ell}(a,x)
            \longrightarrow A_{\ast+1,l}(a,x)
    \]
    by
    \[
    	s(x_0,\ldots,x_k) 
    	=
    	\begin{cases}
    		(-1)^k(x_0,\ldots,x_{k-1},\pi(x_k),x_k)
    		& \mathrm{if}\ x_{k-1}\neq\pi(x_k),
    		\\
    		0
    		&
    		\mathrm{if}\ x_{k-1}=\pi(x_k).
    	\end{cases}
    \]
    We claim that $s$ is a chain homotopy from the identity
    of $A_{\ast,\ell}(a,x)$ to the zero map,
    or in other words that
    \begin{equation}\label{chain-homotopy}
        \partial\circ s + s\circ\partial=\mathrm{Id}.
    \end{equation}
    This immediately shows that the homology 
    of $A_{\ast,\ell}(a,x)$ is trivial, and the result
    follows.
    
    Equation~\eqref{chain-homotopy} is equivalent to the claim
    that 
    \begin{equation}\label{chain-homotopy-expanded}
    	\sum_{i=1}^k (-1)^i \partial_i s(x_0,\ldots,x_k)
    	+ \sum_{i=1}^{k-1}(-1)^i s\partial_i(x_0,\ldots,x_k)
    	=(x_0,\ldots,x_k)
    \end{equation}
    for each generator $(x_0,\ldots,x_k)$ of $A_{\ast,\ell}(a,x)$.

    To begin we consider the case $k=1$.
    In this case  the only possible generator is $(a,x)$, 
    but this is only present when $\ell = d(a,x)$.
    So by our assumption that we do not have both
    $a=\pi(x)$ and $\ell = d(a,x)$, we must have $a\neq\pi(x)$.
    Then~\eqref{chain-homotopy-expanded} becomes
    $
        -\partial_1 s(a,x)
        =
        (a,x),
    $ 
    which follows from the definition of $s$. 

    For the remainder of the proof we consider the case $k\geq 2$.
    We have $\partial_i s = -s\partial_i$ for $1\leq i\leq k-2$,
    since then $\partial_i$ affects only the first $k-1$ terms
    of a tuple, while $s$ affects, and depends upon, only the 
    remaining terms.
    So~\eqref{chain-homotopy-expanded} reduces to the claim
    that:
    \begin{equation}\label{chain-homotopy-simplified}
        \begin{split}
    	&(-1)^{k-1}\ \partial_{k-1}s(x_0,\ldots,x_k)
            \\
    	+
    	&(-1)^{k\phantom{-1}}\ \partial_{k\phantom{-1}} s(x_0,\ldots,x_k)
            \\
    	+&(-1)^{k-1}\ s\partial_{k-1}(x_0,\ldots,x_k)
            =
            (x_0,\ldots,x_k)
        \end{split}
    \end{equation}
    We now divide into the following cases:
    \begin{itemize}
        \item
        Assume that $x_{k-1}=\pi(x_k)$.
        
        Here we see immediately that  
        the first two terms of~\eqref{chain-homotopy-simplified} 
        vanish by the definition of $s$.
        For the third term, we have
        \begin{align*}
            (-1)^{k-1}s\partial_{k-1}(x_0,\ldots,x_k)
            &=
            (-1)^{k-1}s(x_0,\ldots,x_{k-2},x_k)
            \\
            &=
            (x_0,\ldots,x_{k-2},\pi(x_k),x_k)
            \\
            &=
            (x_0,\ldots,x_k)
        \end{align*}
        as required.
        The first of these equations holds because 
        $d(x_{k-2},x_{k-1})+d(x_{k-1},x_k) = d(x_{k-2},x_k)$
        by the defining property of $\pi(x_k)$,
        and the second holds since  $x_{k-2}\neq x_{k-1}=\pi(x_k)$.
        
        \item
        Assume that $x_{k-1}\neq \pi(x_k)$
        and
        $d(x_{k-2},x_{k-1})+d(x_{k-1},x_k) > d(x_{k-2},x_k)$.
        
        Applying the defining property of $\pi(x_k)$ to the 
        assumed inequality, we find that we also have 
        $
            d(x_{k-2},x_{k-1})+d(x_{k-1},\pi(x_k))
            > d(x_{k-2},\pi(x_k))
        $
        so that we have $\partial_{k-1}(x_0,\ldots,x_{k-1},\pi(x_k),x_k)=0$
        and the first term of~\eqref{chain-homotopy-simplified}
        vanishes.
        The assumed inequality also shows that 
        $\partial_{k-1}(x_0,\ldots,x_k)=0$
        so that the third term 
        of~\eqref{chain-homotopy-simplified} vanishes.
        It remains to show that the second term 
        of~\eqref{chain-homotopy-simplified} is 
        $(x_0,\ldots,x_k)$. 
        Thanks to the assumption $x_{k-1}\neq \pi(x_k)$
        we have
        \[
            (-1)^k \partial_k s(x_0,\ldots,x_k)
            =
            \partial_k (x_0,\ldots,x_{k-1},\pi(x_k),x_k)
            =
            (x_0,\ldots,x_{k-1},x_k)
        \]
        as required.
        
        \item
        Assume that 
        $x_{k-1}\neq \pi(x_k)$
        and 
        $d(x_{k-2},x_{k-1})+d(x_{k-1},x_k) = d(x_{k-2},x_k)$.
        
        By applying the defining property of $\pi(x_k)$ to
        the assumed equation, we get
        $d(x_{k-2},x_{k-1})+d(x_{k-1},\pi(x_k)) 
        = d(x_{k-2},\pi(x_k))$,
        and in particular $x_{k-2}\neq \pi(x_k)$.
        The first term of~\eqref{chain-homotopy-simplified}
        is given by:
        \begin{align*}
    	(-1)^{k-1}\partial_{k-1}s(x_0,\ldots,x_k)
            &=
    	-\partial_{k-1}(x_0,\ldots,x_{k-1},\pi(x_k),x_k)
            \\
            &=
            -(x_0,\ldots,x_{k-2},\pi(x_k),x_k),
        \end{align*}
        where the first equation used the assumption
        $x_{k-1}\neq\pi(x_k)$, 
        and the second equation used the fact that
        $d(x_{k-2},x_{k-1})+d(x_{k-1},\pi(x_k))
        = d(x_{k-2},\pi(x_k))$ and $x_{k-2}\neq\pi(x_k)$.
        The second term of~\eqref{chain-homotopy-simplified} is 
        \begin{align*}
    	(-1)^k \partial_k s(x_0,\ldots,x_k)
            &=
            \partial_k (x_0,\ldots,x_{k-1},\pi(x_k),x_k)
            \\
            &=
    	(x_0,\ldots,x_{k-1},x_k),
        \end{align*}
        where again in the first equation we used the assumption
        $x_{k-1}\neq\pi(x_k)$, and in the second equation
        we used the defining property of 
        $\pi(x_k)$ to see that
        $d(x_{k-1},\pi(x_k))+d(\pi(x_k),x_k) = d(x_{k-1},x_k)$.
        The third term of~\eqref{chain-homotopy-simplified} is  
        \begin{align*}
    	(-1)^{k-1}s\partial_{k-1}(x_0,\ldots,x_k)
            &=
    	(-1)^{k-1}s(x_0,\ldots,x_{k-2},x_k)
            \\
            &=
    	(x_0,\ldots,x_{k-2},\pi(x_k),x_k)
        \end{align*} 
        where in the first equation we used the assumption
        that
        $d(x_{k-2},x_{k-1})+d(x_{k-1},x_k) = d(x_{k-2},x_k)$,
        and in the second we used the fact $x_{k-2}\neq\pi(x_k)$.
        So altogether, the left hand side 
        of~\eqref{chain-homotopy-simplified} is
        \[
            -(x_0,\ldots,x_{k-2},\pi(x_k),x_k)
            +
    	(x_0,\ldots,x_{k-1},x_k),
            +
    	(x_0,\ldots,x_{k-2},\pi(x_k),x_k)
        \]
        which is precisely $(x_0,\ldots,x_k)$ as required.
    \end{itemize}
    In all three cases above, the equation
    \eqref{chain-homotopy-simplified} holds.
    This completes the proof.
\end{proof}

Now we take a cofibration $A\hookrightarrow X$ 
and study the chain complex
\[
    \frac{\MC_{\ast,\ast}(X)}
    {\MC_{\ast,\ast}(A)+\MC_{\ast,\ast}(X\setminus A)}
\]
This quotient has a basis consisting of the tuples
$(x_0,\ldots,x_k)$ of vertices of $X$ that have finite length
and that do not lie entirely in $A$ or entirely in $X\setminus A$.
Since $A\hookrightarrow X$ is a cofibration,
there are no paths from vertices of $X\setminus A$
into $A$, and so any such tuple satisfies
$x_0,\ldots,x_i\in A$ and $x_{i+1},\ldots,x_k\in X\setminus A$
for some $i$ in the range $0\leq i <k$.

\begin{definition}\label{big-quotient-filtration}
    Let $A\hookrightarrow X$ be a cofibration
    and let $\ell\geq 0$.
    Define a filtration
    \[
        0 = F_0\subseteq F_1\subseteq 
        \cdots\subseteq F_{\ell-1}=
        \frac{\MC_{\ast,\ell}(X)}
        {\MC_{\ast,\ell}(A)+\MC_{\ast,\ell}(X\setminus A)}
    \]
    by setting $F_i$ to be the span of all those tuples
    $(x_0,\ldots,x_k)$ for which
    $x_0,\ldots,x_{k-i}\in A$.
    In other words, the final $i$ entries of a tuple in $F_i$
    are allowed to be outside $A$, but no more.
    
    To see that the $F_i$ are closed under the boundary map,
    observe that deleting an entry of a tuple 
    either preserves or reduces
    the number of its entries that lie in $X\setminus A$.
    And to see that $F_{\ell-1}$ does indeed exhaust 
    our chain complex, observe that an arbitrary tuple 
    $(x_0,\ldots,x_k)$ has at least its final term 
    in $X\setminus A$ and therefore lies in $F_{k-1}$,
    but that we always have $k\leq\ell$.
\end{definition}

\begin{lemma}\label{filtration-quotient-isomorphism-lemma}
    There is an isomorphism
    \[
        F_i/F_{i-1}
        \xrightarrow{\ \cong \ }
        \bigoplus_{a,(z_1,\ldots,z_i)}
        \Sigma^i A_{\ast,\ell-\ell'}(a,z_1)
    \]
    where the direct sum is over all $a\in A$ and all
    tuples $(z_1,\ldots,z_i)$ of elements of $X\setminus A$
    with $\ell(z_1,\ldots,z_i)\leq\ell$,
    and $\ell'$ is shorthand for $\ell(z_1,\ldots,z_i)$.
\end{lemma}

\begin{proof}
    Since $F_i$ and $F_{i-1}$ are spans of generators,
    with those for the latter 
    being a subset of those for the former,
    the quotient $F_i/F_{i-1}$ has basis consisting of all those
    generators that lie in $F_i$ but not $F_{i-1}$.
    These are precisely the tuples for which
    $x_0,\ldots,x_{k-i}$ lie in $A$ (so that the tuple is in
    $F_i$), while $x_{k-i+1}$ and all subsequent entries of the
    tuple do not (so that the tuple does not lie in $F_{i-1}$).
    So we may write an arbitrary generator of $F_i/F_{i-1}$ in the
    form $(x_0,\ldots,x_{k-i},z_1,\ldots,z_i)$ 
    where $x_0,\ldots,x_{k-i}\in A$ 
    and $z_1,\ldots,z_i\in X\setminus A$.
    
    The boundary map on $F_i/F_{i-1}$ is, as usual, given by
    the alternating sum of all ways of removing an element from
    a tuple, ignoring any terms for which the length is decreased.
    If, in a tuple $(x_0,\ldots,x_{k-i},z_1,\ldots,z_i)$,
    it is one of the $x_j$ that is removed,
    then we find that we still have the first $k-i-1=(k-1)-i$
    terms in $A$. 
    On the other hand, if one of the $z_j$ is removed, then
    in this new tuple the first $k-i  = (k-1)-(i-1)$ 
    terms lie in $A$, so that the tuple lies in $F_{i-1}$ and 
    therefore vanishes in $F_i/F_{i-1}$.
    From this, we see that the differential of $F_i/F_{i-1}$ is
    given by the alternating sum of all ways to delete one
    of the $x_j$ from a tuple 
    $(x_0,\ldots,x_{k-i},z_1,\ldots, z_i)$, omitting any 
    terms where the length is decreased.
    
    The last two paragraphs now enable us to form the
    required isomorphism
    \[
        F_i/F_{i-1}
        \xrightarrow{\ \cong \ }
        \bigoplus_{a,(y_1,\ldots,y_i)}
        \Sigma^i A_{\ast,\ell-\ell'}(a,y_1)
    \]
    by sending 
    $(x_0,\ldots,x_{k-i},z_1,\ldots,z_i)\in F_i/F_{i-1}$
    to the element  $(x_0,\ldots,x_{k-i},z_1)$
    in the summand indexed by $x_0$ and $(z_1,\ldots,z_i)$.
    The first paragraph shows that this is a bijection on 
    generators, and therefore an isomorphism of 
    graded $R$-modules,
    while the second paragraph shows 
    that it respects the differentials on both sides.
\end{proof}

Suppose now that we are given a pushout diagram
\[\begin{tikzcd}
    A \arrow{r}{f} \arrow[hook,d,"i"'] & Y \arrow{d}{j} \\
    X \arrow[r,"g"'] & X \cup_A Y
\end{tikzcd}\]
in which $i$, and consequently $j$, are cofibrations.
Since the square commutes, 
$g$ sends $A\subseteq X$ into $Y\subseteq X\cup_A Y$.
And since (by construction---see Lemma~2.10 of~\cite{CDKOSW}) 
$g$ identifies 
$X\setminus A$ with  $(X\cup_A Y)\setminus Y\subseteq X\cup_A Y$, 
we have an induced map 
\begin{equation}\label{big-quotient-map}
    \frac{\MC_{\ast,\ast}(X)}
    {\MC_{\ast,\ast}(A)+\MC_{\ast,\ast}(X\setminus A)}
    \longrightarrow
    \frac{\MC_{\ast,\ast}(X\cup_A Y)}
    {\MC_{\ast,\ast}(Y)+\MC_{\ast,\ast}(X\cup_A Y \setminus Y)}.
\end{equation}

\begin{lemma}\label{big-quotient-quasi-isomorphism}
    The map~\eqref{big-quotient-map}
    is a quasi-isomorphism.
\end{lemma}

\begin{proof}
    Let us fix some \(\ell \geq 0\) and restrict to length \(\ell\)---that is, fix the second grading to be \(\ell\).
    
    The domain and codomain of~\eqref{big-quotient-map}
    both admit the filtration of 
    Definition~\ref{big-quotient-filtration}.
    For clarity, we denote the filtration of the domain by
    $\{F_i\}$ and that of the codomain by
    $\{F'_i\}$.
    Similarly, both $A\hookrightarrow X$ 
    and $Y\hookrightarrow X\cup_A Y$
    admit the complexes $A_{\ast,\ell}(-,-)$ of 
    Definition~\ref{definition-A}.
    Again for clarity, we will write the the complexes
    associated with $Y\hookrightarrow X\cup_A Y$ 
    as $A'_{\ast,\ell}(-,-)$.
    
    By construction, the induced map~\eqref{big-quotient-map}
    preserves the filtration, 
    and thus induces maps on filtration quotients
    \begin{equation}\label{filtration-quotient-map}
        F_i/F_{i-1}\longrightarrow F'_i/F'_{i-1}.
    \end{equation}
    Since both filtrations terminate after $\ell-1$ steps,
    to prove the lemma it will be sufficient to prove that each
    map~\eqref{filtration-quotient-map} is a quasi-isomorphism.
    The domain and codomain 
    of~\eqref{filtration-quotient-map}
    both admit the isomorphism of 
    Lemma~\ref{filtration-quotient-isomorphism-lemma},
    so that~\eqref{filtration-quotient-map}
    becomes a map
    \begin{equation}\label{filtration-quotient-map-two}
        \bigoplus_{a,(z_1,\ldots,z_i)}
        \Sigma^i A_{\ast,\ell-\ell'}(a,z_1)
        \longrightarrow
        \bigoplus_{y,(z_1,\ldots,z_i)}
        \Sigma^i A'_{\ast,\ell-\ell'}(y,z_1).
    \end{equation}
    It is straightforward to see---since the map
    $g\colon X\to X\cup_A Y$ identifies
    $X\setminus A$ with $X\cup_A Y \setminus Y$---that this map now sends the summand corresponding to
    $a$ and $(z_1,\ldots,z_i)$ to precisely the summand
    corresponding to $f(a)$ and $(z_1,\ldots,z_i)$,
    and that on this summand it is given by the evident induced
    map
    $A_{\ast,\ell}(a,z_1)\to A'_{\ast,\ell}(f(a),z_1)$.
    Lemma~\ref{A-acyclic} shows that the only summands that do
    not have vanishing homology are those of the form
    $A_{\ast,d(\pi(x),x)}(\pi(x),x)$
    and
    $A'_{\ast,d(f(\pi(x),x)}(f(\pi(x)),x)$,
    and Lemma~\ref{A-not-acyclic} 
    shows that in these cases the induced map
    $A_{\ast,d(\pi(x),x)}(\pi(x),x)
    \to 
    A'_{\ast,d(f(\pi(x)),x}(\pi(x),x)$
    is a quasi-isomorphism.
    This is sufficient to complete the proof.
\end{proof}

\begin{proof}[Proof of Theorem~\ref{excision-all}]
    We wish to prove that the maps 
    $E^r_{\ast,\ast}(X,A)\to E^r_{\ast,\ast}(X\cup_A Y,Y)$
    are isomorphisms for $r\geq 1$.
    These maps are induced by a map of filtered complexes,
    and so it is enough to show that there is a quasi-isomorphism
    between the $E^0$-pages,
    or in other words that the map
    $\MC_{\ast,\ast}(X,A)\to \MC_{\ast,\ast}(X\cup_A Y,Y)$
    is a quasi-isomorphism.
    The domain and codomain contain within them a
    copy of $\MC_{\ast,\ast}(X\setminus A$)
    and $\MC_{\ast,\ast}(X\cup_A Y \setminus Y)$
    respectively,
    and these subcomplexes are identified
    by the map in question.
    It is therefore sufficient to prove that the map
    of quotients
    \[
        \frac{\MC_{\ast,\ast}(X,A)}{\MC_{\ast,\ast}(X\setminus A)}
        \longrightarrow
        \frac{\MC_{\ast,\ast}(X\cup_A Y,Y)}
        {\MC_{\ast,\ast}(X\cup_A Y\setminus Y)}
    \] 
    is a quasi-isomorphism.
    But this map is precisely~\eqref{big-quotient-map},
    which is a quasi-isomorphism by
    Lemma~\ref{big-quotient-quasi-isomorphism}.
    This completes the proof.
\end{proof}
    

\section{A cofibration category of directed graphs}\label{sec:cofibration category}

In the drive to develop the formal homotopy theory of directed graphs, results so far have mainly been negative: for various natural notions of weak equivalence, it is known that no model structure can exist \cite{GoyalSanthanam2021, GoyalSanthanam2023, CarranzaKapulkinKim2023}. There has, however, been one positive result. Carranza \textit{et al} \cite{CDKOSW} exhibit a \emph{cofibration category} structure on \(\DiGraph\) for which the weak equivalences are maps inducing isomorphisms on path homology. A cofibration category structure is, roughly speaking, `one half' of a model category structure: it comprises a class of weak equivalences and a class of cofibrations, satisfying axioms that enable the construction of homotopy colimits (but not homotopy limits). This type of structure was first introduced in its dual form (then known as a \emph{category of fibrant objects}) by Brown in 1973 \cite{Brown1973}; for a classical account, see Baues' book \cite{Baues1989}. Various definitions of cofibration categories can now be found in the literature---we adopt the one used in \cite{CDKOSW}.

\begin{definition}[Definition 2.33 in \cite{CDKOSW}]\label{def:cofibration-category}
    A \demph{cofibration category} is a category \(\cl{C}\) together with two distinguished classes of morphisms, the class of \demph{weak equivalences} and the class of \demph{cofibrations}, which satisfy axioms (C1)--(C6) below. An \demph{acyclic cofibration} is a morphism that is both a cofibration and a weak equivalence.
    \begin{itemize}
        \item [(C1)] The class of cofibrations and the class of weak equivalences are each closed under composition, and for every object \(X\) in \(\cl{C}\), the identity morphism \(\mathrm{Id}_X\) is an acyclic cofibration. 
        \item [(C2)] The class of weak equivalences satisfies the \emph{2-out-of-6 property}: given a triple of composable morphisms \(f \colon X \to Y\), \(g\colon Y \to Z\) and \(h\colon Z \to W\), if \(g \circ f\) and \(h \circ g\) are weak equivalences, then so are \(f,g,h\) and \(h \circ g \circ f\).
        \item [(C3)] 
        The category \(\cl{C}\) admits an initial object \(\emptyset\), and every object \(X\) in \(\cl{C}\) is \demph{cofibrant}, meaning that the unique morphism from the initial object to \(X\) is a cofibration.
        \item [(C4)] 
        The category \(\cl{C}\) admits pushouts along cofibrations, and the pushout of an (acyclic) cofibration is an (acyclic) cofibration.
        \item [(C5)] For every object \(X\) in \(\cl{C}\), the codiagonal map \(X \sqcup X \to X\) can be factored as a cofibration followed by a weak equivalence.
        \item [(C6)] The category \(\cl{C}\) admits all small coproducts.
        \item [(C7)] The transfinite composite of (acyclic) cofibrations is again an (acyclic) cofibration.
    \end{itemize}
\end{definition}

Axioms (C1)--(C6) imply that the cofibrations and weak equivalences in a cofibration category \(\cl{C}\) satisfy various properties one would expect to hold in a model category. In particular, every morphism in \(\cl{C}\) can be factored as a cofibration followed by a weak equivalence \cite[p.~421]{Brown1973}, and the pushout of a weak equivalence along a cofibration is a weak equivalence \cite[Lemma I.4.2]{Brown1973}. For further discussion of the definition, we refer the reader to Section 2 of \cite{CDKOSW}. 

Theorem 5.1 in \cite{CDKOSW} says that \(\DiGraph\) carries a cofibration category structure in which the cofibrations are those of \Cref{def:cofib} and the weak equivalences are maps inducing isomorphisms on path homology. Equipped with the results of the previous sections, we will prove that that structure has a refinement, in which the weak equivalences are maps inducing isomorphisms on \emph{bigraded} path homology.

\begin{theorem}[A cofibration category for bigraded path homology]
\label{thm:cofib_cat}
    Fix a ground ring \(R\) which is a P.I.D. The category \(\cn{DiGraph}\) admits a cofibration category structure in which the cofibrations are those in \Cref{def:cofib} and the weak equivalences are morphisms inducing isomorphisms on bigraded path homology.
\end{theorem}

Before proving the theorem, we record a remark, an example, and two auxiliary statements that will be required for the proof.

\begin{remark}\label{rmk:compare-cofibration-categories}
    The maps that induce isomorphisms on bigraded path homology
    are strictly finer than those that induce isomorphisms
    on ordinary path homology, as the examples
    of Section~\ref{sec:directed_cycles} demonstrate.
    Consider the directed cycles
    $Z_m$, for $m\geq 3$ (see~\Cref{fig:cycles}).
    \Cref{theorem-directed-cycles-omnibus} shows that 
    for each $i\geq 0$ there is exactly one $j\geq 0$ for which
    $\PH_{i,j}(Z_m)$ is nonzero, 
    and the first three cases are 
    $\PH_{0,0}(Z_m)$, $\PH_{1,1}(Z_m)$ and $\PH_{2,m}(Z_m)$.
    The first two of these are precisely the (ordinary) path
    homology groups $\PH_0(Z_m)$ and $\PH_1(Z_m)$,
    and are the only nonzero path homology groups.
    On the other hand, the third, $\PH_{2,m}(Z_m)$,
    is not an ordinary path homology group at all.
    Now take $n>m\geq 3$ and consider any map
    \[
        Z_n\longrightarrow Z_m
    \]
    that contracts precisely $n-m$ edges.
    This does \emph{not} induce an isomorphism 
    of bigraded path homology groups, 
    because $\PH_{2,m}(Z_n)$ vanishes while $\PH_{2,m}(Z_m)$
    does not.
    On the other hand, it \emph{does} induce an isomorphism 
    on ordinary path homology groups.
    (The latter can be seen by direct computation.
    Tracking through the proof of~\Cref{bph-Zm} shows that
    $\PH_{0,0}(Z_m)$ is a single copy of $R$
    represented by the reachability chain $(v)$ for any vertex
    $v\in V(Z_m)$, 
    while $\PH_{1,1}(Z_m)$ is again a single copy of $R$
    represented by the reachability chain 
    $\sum_{(a,b)\in E(Z_m)}(a,b)$,
    and similarly for $Z_n$.
    These representatives are preserved by the map $Z_n\to Z_m$.)
\end{remark}

The proof of \Cref{thm:cofib_cat} follows the structure of the proof of Theorem 5.1 in \cite{CDKOSW} closely, and makes use of several facts established in that paper. It also depends on the K\"unneth theorem and the excision theorem for bigraded path homology proved in Sections \ref{sec:kunneth} and \ref{sec:BPH_MV} of this paper, and on the fact that bigraded path homology is a \emph{finitary} functor on the category of directed graphs: it preserves filtered colimits.

\begin{proposition}
\label{thm:MPSS_filtered_colimits}
    Fix a commutative ground ring \(R\). For each \(r \geq 1\) and every \(p,q \in \mathbbZ\), the functor 
    \(\MPSS^r_{pq}(-)\colon \cn{DiGraph} \to \cn{Mod}_R\)
    preserves filtered colimits. In particular this holds for magnitude homology and bigraded path homology.
\end{proposition}

\Cref{thm:MPSS_filtered_colimits} is proved as Proposition 1.14 in~\cite{DIMZ}. Before proceeding to the proof of \Cref{thm:cofib_cat}, let us consider an application of the Proposition.

\begin{example}[An infinite sphere]
\label{eg:infinite_sphere}
    Recall from \Cref{eg:spheres} the `sphere'-like directed graphs \(\mathbb{S}^n\). By construction, each \(\mathbb{S}^n\) is the suspension of \(\mathbb{S}^{n-1}\); in particular, \(\mathbb{S}^{n-1}\) includes into \(\mathbb{S}^{n}\), for each \(n > 0\). This sequence of inclusions gives a filtered diagram
    \begin{equation}
    \label{eq:infinite_sphere}
        \mathbb{S}^0 \to \mathbb{S}^{1} \to \mathbb{S}^2 \to \cdots
    \end{equation}
    in \(\DiGraph\), whose colimit we denote by
    \[\mathbb{S}^\infty = \mathrm{colim}_{\mathbb{N}} \mathbb{S}^n.\]
    We can apply \Cref{thm:MPSS_filtered_colimits} to see that the bigraded path homology of \(\mathbb{S}^\infty\) is
    \[
    \PH_{k,\ell}(\mathbb{S}^\infty) = \begin{cases}
    R & \text{if } k = \ell = 0 \\
    0 &\text{otherwise,}
    \end{cases}
    \]
    analogous to the vanishing in positive dimensions of the singular homology of the infinite topological sphere.

    Indeed, we saw in~\Cref{eg:spheres} that the bigraded path
    homology of $\mathbb{S}^n$ consists of two copies of $R$,
    one in bidegree $(0,0)$ and a second in bidegree $(n,n)$.
    Thus $\PH_{0,0}(\mathbb{S}^n)$ is a copy of $R$ 
    for all $n\geq 1$, and the maps $\mathbb{S}^{n-1}\to \mathbb{S}^n$ induce the identity on this, so that we obtain
    $\PH_{0,0}(\colim_\mathbb{N}\mathbb{S}^n)=
    \colim_\mathbb{N}\PH_{0,0}(\mathbb{S}^n)=R$.
    And in any other bidegree $(i,j)$ there is at most one value
    of $n$ for which $\PH_{i,j}(\mathbb{S}^n)$ is nonzero,
    so that the induced maps 
    $\PH_{i,j}(\mathbb{S}^{n-1})\to \PH_{i,j}(\mathbb{S}^n)$
    all vanish, and 
    $\PH_{i,j}(\colim_\mathbb{N}\mathbb{S}^n)=
    \colim_\mathbb{N}\PH_{i,j}(\colim_n\mathbb{S}^n)=0$.
\end{example}

We require one more lemma. To state it, we recall that a directed graph \(G\) is called \emph{diagonal} if \(\MH_{k\ell}(G) = 0\) whenever \(k \neq \ell\). (This definition was first made for undirected graphs in \cite{HepworthWillerton2017}, and is used in the context of directed graphs in \cite{Asao-path}, for instance.) Observe that the terminal directed graph \(\bullet\) is also the unit object for the box product: for every directed graph \(H\), we have \(\bullet \square H \cong H \cong H \square \bullet\).

\begin{lemma}
\label{lem:contractible_kunneth}
    Fix a ground ring \(R\) which is a P.I.D. Suppose \(G\) is 1-contractible and diagonal. Let \(t\colon G \to \bullet\) denote the terminal map. Then, for any directed graph \(H\), the map of graphs \(t \square \mathrm{Id}\colon G \square H \to H\) induces an isomorphism on bigraded path homology.
\end{lemma}

\begin{proof}
    By \Cref{lem:diag_free}, the assumption that \(G\) is diagonal implies that its magnitude homology is free. We can therefore apply the K\"unneth theorem for bigraded path homology (\Cref{thm:kunneth_box_bigrad}) to get, for every \(H\), a short exact sequence involving \(\PH(G) \otimes \PH(H)\) and \(\PH(G \square H)\) and a \(\tor\) term. The assumption that \(G\) is 1-contractible implies that its bigraded path homology is concentrated in bidegree \((0,0)\), where it is a single copy of the ground ring \(R\). Thus, the \(\tor\) term vanishes, leaving the isomorphism \(\alpha\colon \PH(G) \otimes \PH(H) \xto{\cong} \PH(G \square H)\). 
    The naturality of that isomorphism tells us that this square commutes:
    \[
    \begin{tikzcd}
      \PH(G) \otimes \PH(H) \arrow{r}{\cong} \arrow[swap]{d}{\PH(t) \otimes \PH(\mathrm{Id})} & \PH(G \square H) \arrow{d}{\PH(t \square \mathrm{Id})} \\
      I \otimes \PH(H) \arrow{r}{\cong} & \PH(H)
    \end{tikzcd}
    \]
    where \(I = \PH(\bullet)\) is is the bigraded module given by a single copy of \(R\) in bidegree \((0,0)\). Since the left leg is an isomorphism, the right leg is an isomorphism too.
\end{proof}

\begin{proof}[Proof of \Cref{thm:cofib_cat}]
First, recall that \(\cn{DiGraph}\) is cocomplete.
So axiom (C6) certainly holds, and in particular \(\cn{DiGraph}\) admits an initial object (the empty graph) and pushouts along cofibrations.

That the class of cofibrations is closed under composition and contains all identities is proved as Proposition~3.12 in \cite{CDKOSW}; the corresponding fact for weak equivalences follows immediately from the functoriality of bigraded path homology. This establishes axiom (C1). Axiom (C2) is immediate from the corresponding fact for isomorphisms and the functoriality of bigraded path homology. Axiom (C3)---that all directed graphs are cofibrant---follows immediately from the definitions.

For axiom (C4), consider a pushout diagram
\[\begin{tikzcd}
    A \arrow{r}{f} \arrow[hook,swap]{d}{i} & Y \arrow{d}{j} \\
    X \arrow[swap]{r}{g} & X \cup_A Y
\end{tikzcd}\]
in which \(i\) is a cofibration. Then \(j\) is also a cofibration, by Proposition~3.15 in \cite{CDKOSW}. In particular, our \Cref{thm:les-pair} applies to both \((X,A)\) and \((X \cup_A Y, Y)\), giving a long exact sequence on bigraded path homology in both cases. If \(i\) is an \emph{acyclic} cofibration, so \(\PH_{\ast,\ast}(A) \cong \PH_{\ast,\ast}(X)\), then the long exact sequence of the pair \((X,A)\) tells us the relative homology \(\PH_{\ast,\ast}(X,A)\) vanishes. The excision theorem (\Cref{excision-all}) says that \(\PH_{\ast,\ast}(X \cup_A Y, X) \cong \PH_{\ast,\ast}(X,A) = 0\) and it follows, by considering the long exact sequence of the pair \((X \cup_A Y, Y)\), that \(\PH_{\ast,\ast}(X \cup_A Y) \cong \PH_{\ast,\ast}(Y)\). This establishes (C4).

To prove that axiom (C5) holds, let \(J\) denote the directed graph
\[
\begin{tikzcd}
    -2 & -1 & 0 & 1 & 2 \\ [-20]
    \bullet \arrow{r} & \bullet & \bullet \arrow{l} \arrow{r} & \bullet & \bullet \arrow{l} 
\end{tikzcd}
\]
and let \(\partial J\) denote the subgraph consisting of just the vertices labelled -2 and 2. Then the inclusion \(\iota\colon \partial J \to J\) is a cofibration, and, for every \(X\), the codiagonal map \(X \sqcup X \to X\) factors as
\begin{equation}
\label{eq:codiag}
    X \sqcup X = \partial J \square X \xto{\iota \square \mathrm{Id}_X} J \square X \xto{t \square \mathrm{Id}_X} \bullet \: \square X = X
\end{equation}
where \(t\) is the terminal map.
The first morphism here is a cofibration by Proposition~3.14 in \cite{CDKOSW}, which says that the box product of cofibrations is a cofibration. To see that the second morphism is a weak equivalence, observe that \(J\) is both 1-contractible and diagonal. (Since it contains no paths of length greater than 1, its magnitude chains are concentrated in bidegrees \((0,0)\) and \((1,1)\).) It follows from \Cref{lem:contractible_kunneth}, then, that \(t \square \mathrm{Id}_X\) is a weak equivalence. This establishes (C5).

That the transfinite composite of cofibrations is again a cofibration is proved as Proposition~3.18 in \cite{CDKOSW}; that the transfinite composite of weak equivalences is a weak equivalence follows from \Cref{thm:MPSS_filtered_colimits} together with the fact that a transfinite composite of isomorphisms is an isomorphism. 
This establishes (C7).
\end{proof}


\section{Directed cycles}\label{sec:directed_cycles}

In this section we will explore in detail the magnitude-path spectral 
sequence of the directed cycles.
These examples, together with those in~\Cref{sec:mn_cycles}, 
will clearly demonstrate the strength of the  bigraded theory.

\begin{definition}
    For $m\geq 1$,
    let $Z_m$ denote the \emph{directed cycle of length
    $m$}, i.e.~a graph with $m$ cyclically ordered vertices,
    and with a single directed edge between adjacent vertices, 
    consistently oriented:
    
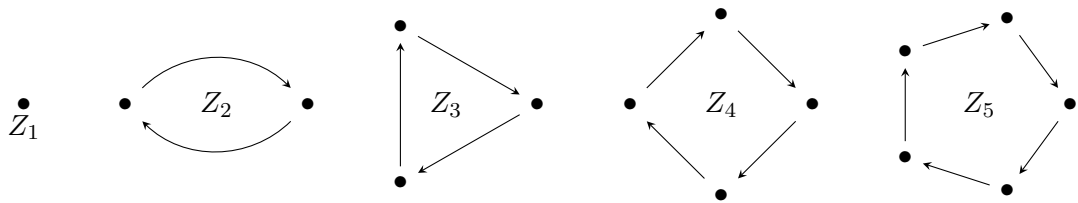
\[
        \begin{tikzpicture}[baseline=0]
            \node (a) at (0:0) {$\bullet$};
            \node[below] () at (0:0) {$Z_1$};
        \end{tikzpicture}
        \qquad
        \begin{tikzpicture}[baseline=0]
            \node (a) at (0:1.2) {$\bullet$};
            \node (b) at (180:1.2) {$\bullet$};
    
            \draw[stealth-] (a) to [out=north west,in=north east] (b);
            \draw[stealth-] (b) to [out= south east, in=south west] (a);
    
            \node () at (0:0) {$Z_2$};
        \end{tikzpicture}
        \qquad
        \begin{tikzpicture}[baseline=0]
            \node (a) at (0:1.2) {$\bullet$};
            \node (b) at (120:1.2) {$\bullet$};
            \node (c) at (240:1.2) {$\bullet$};
    
            \draw[stealth-] (a) -- (b);
            \draw[stealth-] (b) -- (c);
            \draw[stealth-] (c) -- (a);
    
            \node () at (0:0) {$Z_3$};
        \end{tikzpicture}
        \qquad
        \begin{tikzpicture}[baseline=0]
            \node (a) at (0:1.2) {$\bullet$};
            \node (b) at (90:1.2) {$\bullet$};
            \node (c) at (180:1.2) {$\bullet$};
            \node (d) at (270:1.2) {$\bullet$};
    
            \draw[stealth-] (a) -- (b);
            \draw[stealth-] (b) -- (c);
            \draw[stealth-] (c) -- (d);
            \draw[stealth-] (d) -- (a);
    
            \node () at (0:0) {$Z_4$};
        \end{tikzpicture}
        \qquad
        \begin{tikzpicture}[baseline=0]
            \node (a) at (0:1.2) {$\bullet$};
            \node (b) at (72:1.2) {$\bullet$};
            \node (c) at (144:1.2) {$\bullet$};
            \node (d) at (216:1.2) {$\bullet$};
            \node (e) at (288:1.2) {$\bullet$};
    
            \draw[stealth-] (a) -- (b);
            \draw[stealth-] (b) -- (c);
            \draw[stealth-] (c) -- (d);
            \draw[stealth-] (d) -- (e);
            \draw[stealth-] (e) -- (a);
    
            \node () at (0:0) {$Z_5$};
        \end{tikzpicture}
        \]
\end{definition}

Observe that each $Z_m$ has diameter $m-1$, so that 
all $Z_m$ are long homotopy equivalent to one another 
and to the singleton, and in particular they all
have trivial reachability homology.
In contrast, when we consider $1$-homotopy, it is not difficult
to see that for \(m \geq 2\) no two of the $Z_m$ are $1$-homotopy equivalent
(when $m<n$ the only maps $Z_m\to Z_n$ are constant, and for $n>2$
the constant maps $Z_n\to Z_n$ are only $1$-homotopic to other 
constant maps).
Nevertheless, for $m\geq 3$ the $Z_m$ all have the same path homology,
namely a single copy of $R$ in degrees $0$ and $1$.
So path homology cannot distinguish 
the different oriented cycles.
In contrast, we will see that \emph{bigraded} path homology 
can distinguish all except for $Z_1$ and $Z_2$.

\begin{theorem}[Magnitude homology, bigraded path homology, and MPSS of directed cycles]
\label{theorem-directed-cycles-omnibus}
    Let $m\geq 3$.
    Then the magnitude homology $\MH_{\ast,\ast}(Z_m)$
    and  bigraded path homology $\PH_{\ast,\ast}(Z_m)$
    are both concentrated in bidegrees of the form 
    $(2i,mi)$ and $(2i+1,mi+1)$, in each of which they are free of rank $m$
    and rank $1$ respectively.
    Moreover, the MPSS of $Z_m$ satisfies  $E^2(Z_m)=\cdots=E^{m-1}(Z_m)$
    while $E^m(Z_m)$ is trivial, consisting of a single copy of $R$ in
    bidegree $(0,0)$.
\end{theorem}

In the case $m=1$ the magnitude and bigraded path homology of $Z_1$ are
both trivial, consisting of a single copy of $R$ in bidegree $(0,0)$.
In the case $m=2$ the description of the magnitude homology
given in the theorem remains true,
while the bigraded path homology becomes trivial.

\begin{corollary}
    The \(E^r\)-page of the MPSS distinguishes all of the directed cycles \(Z_m\) for \(m \geq r\), and is trivial for \(m \leq r\). In particular, bigraded path homology distinguishes all the directed cycles \(Z_m\) for \(m \geq 2\).
\end{corollary}

Let us describe the MPSS of $Z_m$ and preview the work of the rest
of the section; we assume $m\geq 3$ for simplicity.
In Theorem~\ref{directed-cycle-MH} we will see that
the magnitude homology of $Z_m$ vanishes except for copies of
$R^m$ that occur in pairs, one in each total degree of the 
spectral sequence, each pair being connected by
the differential $d^1$, as depicted in \Cref{fig:E1_Zm}.
\begin{figure}[h!]
    \centering
    \begin{tikzpicture}[baseline=0,xscale=1.3]
        \draw[-stealth, ultra thick,white!90!black] (-1,0) to (10,0);
        \draw[-stealth, ultra thick,white!90!black] (0,1) to (0,-5);
        \node (a) at (0,0) {$R^m$};
        \node (b) at (1,0) {$R^m$};
        \node (c) at (4,-2) {$R^m$};
        \node (d) at (5,-2) {$R^m$};
        \node (e) at (8,-4) {$R^m$};
        \node (f) at (9,-4) {$R^m$};
        \draw[-stealth] (b) to node[below]{$\scriptstyle d^1$} (a);
        \draw[-stealth] (d) to node[below]{$\scriptstyle d^1$}(c);
        \draw[-stealth] (f) to node[below]{$\scriptstyle d^1$}(e);
        \node() at (0,0.5) {$\scriptstyle 0$};
        \node() at (1,0.5) {$\scriptstyle 1$};
        \node() at (4,0.5) {$\scriptstyle m$};
        \node() at (5,0.5) {$\scriptstyle m+1$};
        \node() at (8,0.5) {$\scriptstyle 2m$};
        \node() at (9,0.5) {$\scriptstyle 2m+1$};
        \node() at (-0.5,0) {$\scriptstyle 0$};
        \node() at (-0.5,-2) {$\scriptstyle -(m-2)$};
        \node() at (-0.5,-4) {$\scriptstyle -2(m-2)$};
        \node[draw,below left]() at (10,-0.2) {$E^1_{\ast,\ast}(Z_m)$};
    \end{tikzpicture}
    \caption{The magnitude homology of the directed \(m\)-cycle.}
    \label{fig:E1_Zm}
\end{figure}
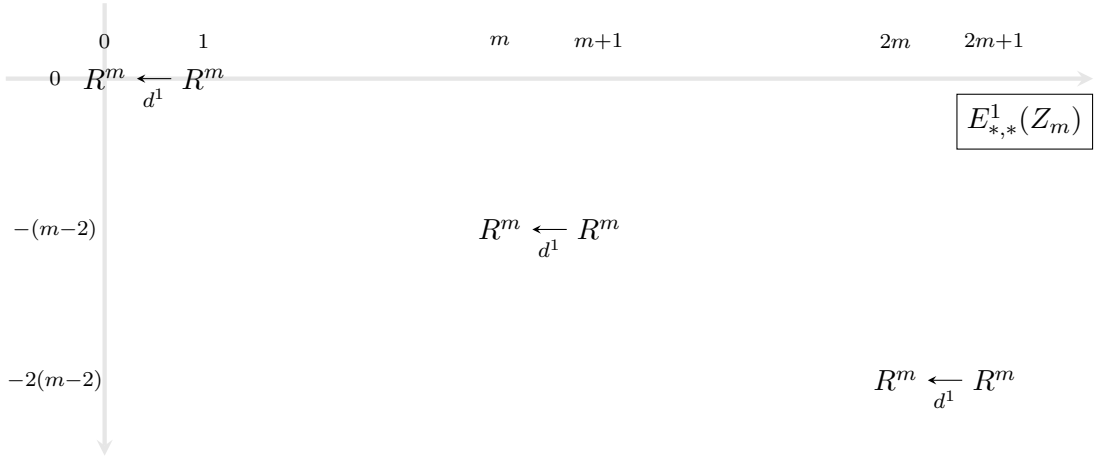
This result is based on Theorem~\ref{OP-homology},
which computes the homology
of an auxiliary chain complex, which we call the 
\emph{complex of partitions with upper bound}.
The magnitude homology of $Z_m$ is then a direct sum of
copies of the homology groups of these complexes of partitions.

Returning to the $E^1$ term,
we will see in Theorem~\ref{bph-Zm} that by computing the
differentials $d^1$,
it follows that the $E^2$-term---or in other words, the bigraded path homology---has the form depicted in \Cref{fig:E2_Zm}.
\begin{figure}[h!]
    \centering
    \begin{tikzpicture}[baseline=0,xscale=1.3]
        \draw[-stealth, ultra thick,white!90!black] (-1,0) to (10,0);
        \draw[-stealth, ultra thick,white!90!black] (0,1) to (0,-5);
        \node (a) at (0,0) {$R$};
        \node (b) at (1,0) {$R$};
        \node (c) at (4,-2) {$R$};
        \node (d) at (5,-2) {$R$};
        \node (e) at (8,-4) {$R$};
        \node (f) at (9,-4) {$R$};
        \draw[-stealth] (c) to node[below left]{$\scriptstyle d^{m-1}$} (b);
        \draw[-stealth] (e) to  node[below left]{$\scriptstyle d^{m-1}$} (d);
        \node() at (0,0.5) {$\scriptstyle 0$};
        \node() at (1,0.5) {$\scriptstyle 1$};
        \node() at (4,0.5) {$\scriptstyle m$};
        \node() at (5,0.5) {$\scriptstyle m+1$};
        \node() at (8,0.5) {$\scriptstyle 2m$};
        \node() at (9,0.5) {$\scriptstyle 2m+1$};
        \node() at (-0.5,0) {$\scriptstyle 0$};
        \node() at (-0.5,-2) {$\scriptstyle -(m-2)$};
        \node() at (-0.5,-4) {$\scriptstyle -2(m-2)$};
        \node[draw, below left]() at (10,-0.2) {$E^2_{\ast,\ast}(Z_m)=E^{m-1}_{\ast,\ast}(Z_m)$};
    \end{tikzpicture}
    \caption{The bigraded path homology of the directed \(m\)-cycle.}
    \label{fig:E2_Zm}
\end{figure}
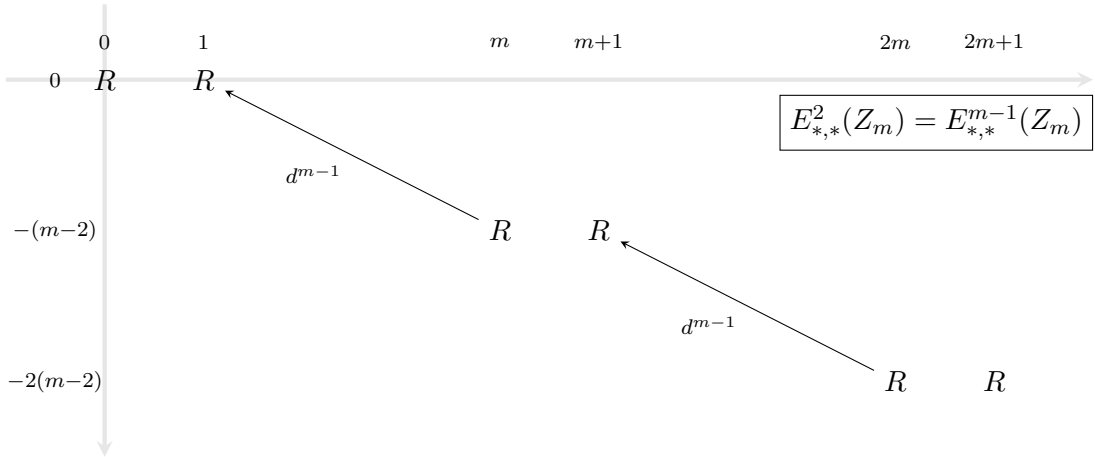
The nonzero groups here lie in bidegrees of the form
$(im,-i(m-2))$ and $(im+1,-i(m-2))$, and in particular these
bidegrees mean that the bigraded path homology of $Z_m$
determines $m$ for $m\geq 3$.
Finally, the position of the nonzero terms means that the only
subsequent differential can be in the $(m-1)$-page of the 
spectral sequence.
Since we know that the homology of the reachability complex is trivial, the same 
must be true of the term $E^m=E^\infty$, so that each $d^{m-1}$
depicted must be an isomorphism.

\subsection{The complex of ordered partitions with upper bound}

\begin{definition}
    Let $\ell\in\mathbbZ$.
    An \emph{ordered partition} of $\ell$
    is an ordered tuple $(a_1,\ldots,a_k)$
    of $k\geq 0$ positive integers $a_i$ for which
    $a_1+\cdots+a_k=\ell$.
\end{definition}

In the last definition, 
observe that there is a unique ordered partition of $\ell=0$,
namely the empty tuple $()$, but that any ordered partition of 
$\ell>0$ must have $k>0$ entries.
Note also that the definition admits the cases $\ell<0$,
but that in these cases there are no ordered partitions.
(This apparently pointless decision will be useful 
in formulating a later lemma.)

\begin{definition}
    Given $\ell,m\in\mathbbZ$ with $m\geq 2$,  
    we let $\OP_\ast(\ell,m)$ be the chain complex 
    of $R$-modules defined as follows.
    In degree $k$,
    $\OP_k(\ell,m)$ has basis given by the ordered 
    partitions $(a_1,\ldots,a_k)$ of $\ell$
    such that $a_i<m$ for all $i$.
    The differential $d$ vanishes for $k\leq 1$,
    and is given in degrees $k\geq 2$ 
    by summing  adjacent entries,
    \[
        d(a_1,\ldots,a_k)
        =
        \sum_{i=1}^{k-1} (-1)^i
        (a_1,\ldots,a_i+a_{i+1},\ldots,a_k),
    \]
    and omitting any terms in which the summed entry
    fails the requirement $a_i+a_{i+1}<m$.
\end{definition}

In the last definition, note that
$\OP_\ast(\ell,m)$ is defined---and vanishes---for $\ell\leq -1$,
while $\OP_\ast(0,m)$ has a single basis element
$()\in\OP_0(0,m)$, 
and $\OP_\ast(1,m)$ has the single basis element
$(1)\in\OP_1(1,m)$.

We will see later that the magnitude homology of a directed cycle
of size $m$ has a description in terms of 
the homology of the complexes $\OP_\ast(\ell,m)$.
Our aim now is to compute that homology.

\begin{theorem}\label{OP-homology}
    Let $\ell,m,k\in\mathbbZ$ with $m\geq 2$ and $k\geq 0$.
    Then $H_k(\OP(\ell,m))=0$, except in the following cases.
    \begin{itemize}
        \item
        Let $i\geq 0$.
        Then $H_{2i}(\OP(mi,m))$ is a copy of $R$ generated by
        the element 
        \[
            [
                {(1,m-1,\ldots,1,m-1)}
            ]
            =
            [
                {(m-1,1,\ldots,m-1,1)}
            ],
        \]
        where each tuple has $2i$ entries.
        \item
        Let $i\geq 0$.
        Then $H_{2i+1}(\OP(mi+1,m))$ 
        is a copy of $R$ generated by the element 
        \[
            [
                {(1,m-1,\ldots,1,m-1,1)}
            ],
        \]
        where the tuple has $2i+1$ entries.
    \end{itemize}
\end{theorem}

This theorem will be used as a black box in our computation
of the magnitude homology of the directed $m$-cycle.
The remainder of this subsection is dedicated to the proof
of the theorem, which the reader might therefore wish to skip.

We will see that, modulo a degree shift of $2$,
$H_\ast(\OP(\ell,m))$ is in fact periodic in $\ell$ 
with period $m$.
This will allow us to prove the theorem by induction.
We begin with the base cases.

\begin{lemma}\label{OP-base-cases}
    For $\ell\leq 1$, the homology of 
    $\OP_\ast(\ell,m)$ is as follows:
    \begin{itemize}
        \item
        For $\ell<0$, $H_\ast(\OP(\ell,m))$
        vanishes.
        \item
        For $\ell=0$, $H_\ast(\OP(\ell,m))$ 
        is a single copy of $R$ concentrated
        in degree $0$, generated by the class of the
        empty word $()$.
        \item
        For $\ell=1$, $H_\ast(\OP(\ell,m))$
        is a single copy of $R$ concentrated in degree
        $1$, generated by the class of the word $(1)$.
    \end{itemize}
\end{lemma}

\begin{proof}
    These claims are immediate: $\OP_\ast(\ell,m)$
    vanishes for $\ell<0$, while each of $\OP_\ast(0,m)$
    and $\OP_\ast(1,m)$ has a single basis element,
    namely $()$ and $(1)$ respectively.
\end{proof}

\begin{definition}
    Let $\ell\geq 2$ and $m\geq 2$, and  define a chain map 
    $\phi\colon\OP_{\ast-2}(\ell-m,m)\to\OP_{\ast}(\ell,m)$
    by 
    \[
        \phi(a_1,\ldots,a_k)
        =
        (1,m-1,a_1,\ldots,a_k).
    \]
    It is straightforward to check that this is a chain map,
    thanks to the fact that $1+(m-1)\geq m$ and
    $(m-1)+a_1\geq m$.
\end{definition}

\begin{lemma}\label{OP-step}
    Let $\ell,m\geq 2$.
    Then $\phi\colon\OP_{\ast-2}(\ell-m,m)\to\OP_{\ast}(\ell,m)$ 
    is a quasi-isomorphism.
\end{lemma}

Note that we include the cases $2\leq\ell\leq m$,
where the lemma tells us that $\OP_\ast(\ell,m)$
is acyclic.

\begin{proof}
    Since $\phi$ is injective, it will suffice to show that
    its cokernel, which we denote by $C_\ast$, is acyclic.
    Observe that $C_k$ has basis given by the basis elements
    $(a_1,\ldots,a_k)$ of $\OP_k(\ell,m)$
    for which either $a_1>1$, or $a_1=1$ and $a_2<m-1$.
    To prove that $C_\ast$ is acyclic, we use the map
    $s\colon C_\ast\to C_{\ast+1}$ defined on basis elements as
    follows.
    \[
        s(a_1,\ldots,a_k)
        =
        \begin{cases}
            0
            &
            \text{if }a_1=1
            \\
            -(1,a_1-1,a_2,\ldots,a_k)
            &
            \text{if }a_1>1
        \end{cases}
    \]
    We will show that $s$ is a chain contraction, i.e.~that
    $sd+ds$ is the identity map, and this will complete 
    the proof.
    Let us therefore fix a basis element $(a_1,\ldots,a_k)$
    and show that $(sd+ds)(a_1,\ldots,a_k)=(a_1,\ldots,a_k)$.
    We will do this in three separate cases, as follows.

    \emph{Case 1.}
    We begin with the case $k=1$, so that the only possible
    basis element of $C_1$ is $(\ell)$. 
    (Even this element will not be present when $\ell\geq m$.)
    Then $sd(\ell)=s(0)=0$ and $ds(\ell)=-d(1,\ell-1)=(\ell)$,
    so that $(sd+ds)(\ell)=(\ell)$ as required.
    (Note that the restriction to $\ell\geq 2$ 
    was necessary since case 1 fails when $\ell=1$.)
    
    \emph{Case 2.}
    Let us assume now that $k>1$, that $a_1=1$
    and that $a_2<m-1$.
    Then $ds(a_1,\ldots,a_k)=0$, while
    \begin{align*}
        sd(a_1,\ldots,a_k)
        &=
        \sum_{i=1}^{k-1} (-1)^i
        s(a_1,\ldots,a_i+a_{i+1},\ldots,a_k)
        \\
        &=
        -s(a_1+a_2,a_3,\ldots,a_k)
        \\
        &=
        (a_1,\ldots,a_k).
    \end{align*}
    In the first line, the $i$-th term of the sum is omitted if
    $a_i+a_{i+1}\geq m$. 
    This omission only happens for $i>1$,
    and in these cases the resulting partition has initial
    entry $a_1=1$ and therefore vanishes under $s$,
    giving us the second line, 
    and the third then follows immediately since $a_1+a_2>1$.
    Thus in this case $sd+ds$ sends $(a_1,\ldots,a_k)$
    to itself as required.

    \emph{Case 3.}
    Let $(a_1,\ldots,a_k)$ be a basis element of $C_k$,
    and assume that $1< a_1\leq m-1$.
    Then 
    \begin{align*}
        sd(a_1,\ldots,a_k)
        &=
        \sum_{i=1}^{k-1} (-1)^i
        s(a_1,\ldots,a_i+a_{i+1},\ldots,a_k)
        \\
        &=
        -s(a_1+a_2,a_3,\ldots,a_k)
        +
        \sum_{i=2}^{k-1} (-1)^i 
        s(a_1,\ldots,a_i+a_{i+1},\ldots,a_k)
        \\
        &=
        (1,a_1+a_2-1,a_3,\ldots,a_k)
        +
        \sum_{i=2}^{k-1} (-1)^{i+1} 
        (1,a_1-1,\ldots,a_i+a_{i+1},\ldots,a_k)
    \end{align*}
    where the first term is omitted if $a_1+a_2\geq m$,
    and the $i$-th term of each sum is omitted if
    $a_i+a_{i+1}\geq m$.
    And
    \begin{align*}
        ds(a_1,\ldots,a_k)
        &=
        -d(1,a_1-1,a_2,\ldots,a_k)
        \\
        &=
        (a_1,\ldots,a_k)
        -
        (1,a_1+a_2-1,a_3,\ldots,a_k)
        \\
        &\qquad\qquad
        +\sum_{i=2}^{k-1} (-1)^i
        (1,a_1-1,\ldots,a_i+a_{i+1},\ldots,a_k)
    \end{align*}
    where the second term is omitted if $a_1+a_2-1\geq m$,
    and the $i$-th term of the sum is omitted if
    $a_i+a_{i+1}\geq m$.
    Thus when we compute $(sd+ds)(a_1,\ldots,a_k)$
    we find that the sums cancel out,
    the question of which of their terms are omitted being
    the same in each case,
    and that $(a_1,\ldots,a_k)$ always remains.
    Thus $(sd+ds)(a_1,\ldots,a_k)=(a_1,\ldots,a_k)$
    will hold so long as we can show that 
    the two potential terms $(1,a_1+a_2-1,a_3,\ldots,a_k)$ 
    contribute $0$ overall, 
    the question being exactly when each one is omitted.
    If $a_1+a_2\geq m+1$ then both of these terms are omitted;
    if $a_1+a_2\leq m-1$ then both are present and cancel out;
    and if $a_1+a_2=m$ then one is present and the other is not.
    But in this last case we in fact have
    $(1,a_1+a_2-1,a_3,\ldots,a_k) =(1,m-1,a_3,\ldots,a_k)=0$
    because we work in the cokernel of $\phi$.
    This completes the proof that $(ds+sd)(a_1,\ldots,a_k)
    =(a_1,\ldots,a_k)$ in this case.
\end{proof}

\begin{proof}[Proof of Theorem~\ref{OP-homology}]
    The two cycles listed in the first bullet point
    are homologous. 
    Indeed, let us define
    $c\in\OP_{2i+1}(mi,m)$ to be the sum
    \[
        c
        =
        \sum_{j=1}^i
        (1,m-1,\ldots,1,m-2,1,\ldots,m-1,1),
    \]
    where each tuple has $2i+1$ entries,
    and $m-2$ appears in position $2j$ in the $j$-th term.
    Then
    \[
        dc
        =
        (1,m-1,\ldots,1,m-1)
        -
        (m-1,1,\ldots,m-1,1)
    \]
    where now each tuple has $2i$ entries.
    
    For $\ell\leq 1$ the theorem follows from 
    Lemma~\ref{OP-base-cases},
    while for $\ell\geq 2$ we may repeatedly apply
    Lemma~\ref{OP-step} to one of the cases from 
    Lemma~\ref{OP-base-cases} to obtain the result,
    including the description of the  generators.
\end{proof}

\subsection{The magnitude homology of $Z_m$}

If $x$ is a vertex of $Z_m$, then we write $x^+$ for
the next vertex in the cyclic order, and $x^-$ for
the previous vertex.
In other words, $x^\pm$ are characterised by 
the existence of edges $x^-\to x\to x^+$.

\begin{definition}
    We define two families of classes in the magnitude
    homology of $Z_m$ as follows.
    \begin{itemize}
        \item
        Given $x\in V(Z_m)$ and $i\geq 0$,
        define $\kappa_x^i\in\MH_{2i,mi}(Z_m)$
        by
        \[
            \kappa_x^i
            =
            [
                (x,x^+,\ldots,x,x^+,x)
            ]
            =
            [
                (x,x^-,\ldots,x,x^-,x)
            ],
        \]
        where the tuple has a total of $2i+1$ entries.

        \item
        Given $e\in E(Z_m)$ and $i\geq 0$,
        define $\lambda^i_e\in\MH_{2i+1,mi+1}(Z_m)$
        by 
        \[
            \lambda_e^i
            =
            [
                (x,x^+,\ldots,x,x^+)
            ],
        \]
        where $e=xx^+$ and the tuple has a total of 
        $2i+2$ entries.
    \end{itemize}
    In the spectral sequence grading, 
    these classes lie in positions
    \[
        \kappa^i_x\in E^1_{mi,(2-m)i}(Z_m)
        \quad
        \text{and}
        \quad
        \lambda^i_e\in E^1_{mi+1,(2-m)i}(Z_m).
    \]
\end{definition}

\begin{theorem}\label{directed-cycle-MH}
    Let $m\geq 2$.
    The magnitude homology $\MH_{\ast,\ast}(Z_m)$
    is the free $R$-module with basis given by the 
    elements $\kappa_x^i$ and $\lambda_e^i$
    for $i\geq 0$, $x\in V(Z_m)$ and $e\in E(Z_m)$.
    In particular, it is free of rank
    $m$ in bidegrees of the form 
    $(2i,mi)$ and $(2i+1,mi+1)$, and it is zero
    in all other bidegrees.
\end{theorem}

\begin{proof}
    Observe that if we fix a vertex $x\in V(Z_m)$
    and a length $\ell\geq 0$,
    then the tuples of the form
    $(x_0,\ldots,x_k)$,
    where $x_0=x$, $k\geq 0$
    and $\ell(x_0,\ldots,x_k)=\ell$,
    span a subcomplex $\MC_\ast(x,\ell)$
    of $\MC_{\ast,\ell}(Z_m)$.
    Moreover, observe that 
    $\MC_{\ast,\ell}(Z_m)$ is the direct 
    sum of the $\MC_\ast(x,\ell)$ for $x\in V(Z_m)$.
    
    There is an isomorphism
    of chain complexes
    \[
        \MC_\ast(x,\ell)
        \xrightarrow{\ \cong\ }
        \OP_\ast(\ell,m)
    \]
    defined by
    \[
        (x,x_1,\ldots,x_k)
        \longmapsto
        (d(x,x_1),d(x_1,x_2),\ldots,d(x_{k-1},x_k)).
    \]
    That this is an isomorphism of $R$-modules follows
    immediately
    from the fact that, given $x\in V(Z_m)$ and
    $a\in\{1,\ldots,m-1\}$, there is a unique
    vertex $y\neq x$ with $d(x,y)=a$.
    That it is a chain map follows from the fact that,
    given $x,y,z\in V(Z_m)$,
    we have $d(x,y)+d(y,z)=d(x,z)$ if 
    and only if $d(x,y)+d(y,z)<m$.

    The isomorphism above identifies
    the elements 
    \[
        (x,x^+,\ldots,x,x^+,x),
        (x,x^-,\ldots,x,x^-,x)
        \in\MC_{2i}(x,mi)
    \]
    with the elements
    \[
        (1,m-1,\ldots,1,m-1),
        (m-1,1,\ldots,m-1,1)
        \in
        \OP_{2i}(mi),
    \]
    and identifies the element
    \[
        (x,x^+,\ldots,x,x^+)\in\MC_{2i+1}(x,mi+1)
    \]
    with
    \[
        (1,m-1,\ldots,m-1,1)\in\OP_{2i+1}(mi+1).
    \]
    The theorem now follows directly 
    from
    Theorem~\ref{OP-homology}.
\end{proof}

\subsection{Bigraded path homology}

\begin{theorem}\label{bph-Zm}
    Let $m\geq 3$.
    The bigraded path homology of $Z_m$ is the free
    $R$-module with basis given by classes
    $\alpha^{2i},\beta^{2i+1}$ for all $i\geq 0$, where
    \[
        \alpha^{2i}\in E^2_{mi,-(m-2)i}(Z_m),
        \qquad
        \beta^{2i+1}\in E^2_{mi+1,-(m-2)i}(Z_m)
    \]
    or equivalently 
    \[
        \alpha^{2i}\in \PH_{2i,mi}(Z_m),
        \qquad
        \beta^{2i+1}\in \PH_{2i+1,mi+1}(Z_m).
    \]
    In particular, $\PH_{\ast,\ast}(Z_m)$
    is free of rank $1$ in bidegrees of form $(2i,mi)$ and $(2i+1,mi+1)$
    for $i\geq 0$, and vanishes in all other bidegrees.
\end{theorem}

\begin{lemma}\label{d-one-lemma}
    Let $m\geq 3$.
    If $e=xy\in E(Z_m)$, then 
    $d^1(\lambda^i_e)
    = \kappa^i_y - \kappa^i_x$.
    And if $x\in V(Z_m)$ then $d^1(\kappa^i_x)=0$.
\end{lemma}

\begin{proof}
    We have $y=x^+$, so that
    \[
        \lambda^i_e 
        =
        [(x,x^+,\ldots,x,x^+)]
        \in
        \MH_{2i+1,mi+1}(Z_m)
    \] 
    and consequently
    \[
        d^1(\lambda^i_e)
        =
        [d(x,x^+,\ldots,x,x^+)]
        \in
        \MH_{2i,mi}(Z_m).
    \]
    Here $d(x,x^+,\ldots,x,x^+)\in\MC_{2i,mi}(Z_m)$
    is the alternating sum of the tuples obtained
    by omitting terms from  $(x,x^+,\ldots,x,x^+)$.
    When we delete anything other than the first or last
    term, we obtain a tuple with repeated consecutive
    entries, which therefore vanishes in the magnitude
    chain group.
    Thus
    \begin{align*}
        d(x,x^+,\ldots,x,x^+)
        &=
        (x^+,x,\ldots ,x,x^+)
        - (x,x^+,\ldots,x^+,x)
        \\
        &=
        (y,y^-,\ldots ,y,y^-)
        - (x,x^+,\ldots,x^+,x)
        \\
    \end{align*}
    so that, taking classes in $\MH_{2i,mi}(Z_m)$,
    we find 
    $d^1(\lambda_e^i) = \kappa^i_y-\kappa^i_x$
    as required.
    Finally, $d^1(\kappa^i_x)=0$ because it lies in a bidegree
    where magnitude homology vanishes.
\end{proof}

\begin{proof}[Proof of Theorem~\ref{bph-Zm}]
    Lemma~\ref{d-one-lemma} shows that the homology of 
    $E^1_{\ast,\ast}(Z_m)$ with respect to $d^1$
    is freely spanned over $R$ by the
    following homology classes, for each $i\geq 0$:
    \begin{enumerate}
        \item
        $[\kappa^{2i}_x]$ for any chosen vertex $x$ of $Z_m$.
        A different choice of vertex defines the same 
        homology class.
        \item
        $\left[\sum_{e\in E(Z_m)} \lambda^{2i+1}_e\right]$.
    \end{enumerate}
    We may now define $\alpha^{2i}$ and $\beta^{2i+1}$ to be
    the elements of $E^2_{\ast,\ast}(Z_m)$ corresponding to 
    these classes under the isomorphism
    $H(E^1_{\ast,\ast}(Z_m))\cong E^2_{\ast,\ast}(Z_m)$.
\end{proof}


\section{Bi-directed cycles}
\label{sec:mn_cycles}

In this section we determine the magnitude-path spectral sequence
for the bi-directed cycles.
Again, we find that the bigraded path homology contains strictly
more information than the path homology itself,
but that in contrast with the case of oriented cycles, it is not
sufficient to determine graphs of this class up to isomorphism.

\begin{definition}
    Let $m$ and $n$ be integers with $m,n\geq 1$.
    The \emph{bi-directed cycle} or \emph{$(m,n)$-cycle} 
    $C_{m,n}$ is obtained  by taking 
    directed intervals of length $m$ and $n$, 
    and then identifying their initial points to a single point,
    and their final points to a single point.
    Put differently, it is obtained from an unoriented $m+n$
    cycle by orienting a set of $m$ contiguous edges 
    in one direction,
    and the remaining $n$ edges in the opposite direction.
    We often think of $C_{m,n}$ as having its two intervals 
    oriented from left to right, with the interval of length $m$
    across the top, and the one of length $n$ across the bottom,
    as follows:
    
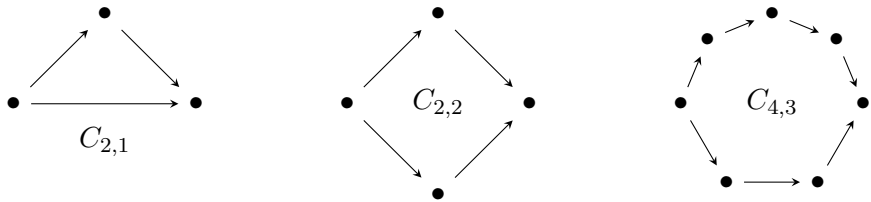
\[
        \begin{tikzpicture}[baseline=0]
            \node(x) at (180:1.2) {$\bullet$};
            \node(y) at (0:1.2) {$\bullet$};
    
            \node (a2) at (90:1.2) {$\bullet$};
            
            \draw[-stealth] (x) -- (a2);
            \draw[-stealth] (a2) -- (y);
    
            \draw[-stealth] (x) -- (y);
            
            \node () at (-90:0.5) {$C_{2,1}$};
        \end{tikzpicture}
        \qquad\qquad
        \begin{tikzpicture}[baseline=0]
            \node(x) at (180:1.2) {$\bullet$};
            \node(y) at (0:1.2) {$\bullet$};
    
            \node (a2) at (90:1.2) {$\bullet$};
            \node (b2) at (-90:1.2) {$\bullet$};
            
            \draw[-stealth] (x) -- (a2);
            \draw[-stealth] (a2) -- (y);
    
            \draw[-stealth] (x) -- (b2);
            \draw[-stealth] (b2) -- (y);
            
            \node () at (0:0) {$C_{2,2}$};
        \end{tikzpicture}
        \qquad\qquad
        \begin{tikzpicture}[baseline=0]
            \node(x) at (180:1.2) {$\bullet$};
            \node(y) at (0:1.2) {$\bullet$};
    
            \node (a2) at (135:1.2) {$\bullet$};
            \node (a3) at (90:1.2) {$\bullet$};
            \node (a4) at (45:1.2) {$\bullet$};
    
            \node (b2) at (-120:1.2) {$\bullet$};
            \node (b3) at (-60:1.2) {$\bullet$};
            
            \draw[-stealth] (x) -- (a2);
            \draw[-stealth] (a2) -- (a3);
            \draw[-stealth] (a3) -- (a4);
            \draw[-stealth] (a4) -- (y);
    
            \draw[-stealth] (x) -- (b2);
            \draw[-stealth] (b2) -- (b3);
            \draw[-stealth] (b3) -- (y);
            
            \node () at (0:0) {$C_{4,3}$};
        \end{tikzpicture}
    \]
    Observe that $C_{m,n}$ has both an initial and a terminal
    vertex; 
    in our diagrams these are the ones at the left and right
    respectively.
    Moreover, $C_{m,n}$ has diameter $\m-1$ where
    $\m=\max(m,n)$.
\end{definition} 

Throughout the section we will use the letter $\m$ to denote
$\max(m,n)$.
The following result tells us that 
the homological behaviour of $C_{m,n}$
is largely determined by $\m$ alone.

\begin{proposition}\label{proposition-independent}
    The bigraded path homology of $C_{m,n}$ 
    depends only on $\m=\max(m,n)$.
    More precisely, choose the shorter of the two directed intervals
    of $C_{m,n}$ (or either one if $m=n$), 
    and contract all of its edges except the last,
    to obtain a map $C_{m,n}\to C_{\m,1}$.
    This map is an isomorphism on bigraded path homology.
\end{proposition}

This proposition will be used to reduce our task to computing the
MPSS for the graphs $C_{m,1}$.
Although it is possible to compute the MPSS for all $C_{m,n}$
directly, that is significantly more onerous,
and in particular requires the tedious separation of the 
cases $m=n$ and $m\neq n$.

\begin{theorem}[Bigraded path homology and MPSS of bi-directed cycles]
\label{theorem-mn-cycles-omnibus}
    Let $m,n\geq 1$ and assume further that $\m=\max(m,n)\geq 3$.
    Then the bigraded path homology $\PH_{\ast,\ast}(C_{m,n})$
    is concentrated in bidegrees $(0,0)$, $(1,0)$ and $(\m,-(\m-2))$,
    in each of which it is free of rank $1$.
    Moreover, the MPSS of $C_{m,n}$ satisfies  
    $E^2(C_{m,n})=\cdots=E^{\m-1}(C_{m,n})$
    while $E^\m(C_{m,n})$ is trivial, 
    consisting of a single copy of $R$ in bidegree $(0,0)$.
    If $\m=1,2$ then the bigraded path homology is concentrated in bidegree
    $(0,0)$, where it is free of rank $1$.
\end{theorem}

\begin{corollary}
    The MPSS of $C_{m,n}$ depends only on the value of $\m=\max(m,n)$.
    The \(E^r\)-page of the MPSS of $C_{m,n}$
    determines $\m$ for \(\m \geq r\), 
    and is trivial for \(\m \leq r\). 
    In particular, bigraded path homology 
    determines the value of $\m$ for  \(\m \geq 2\).
\end{corollary}

In this section we will describe the  MPSS for $C_{m,n}$
from the $E^2$-term onwards.
In the case $\m=2$, the bigraded path homology of $C_{m,n}$ is
trivial, consisting of a single copy of $R$ in degree $(0,0)$.
For $\m\geq 3$, the bigraded path homology 
for $C_{m,n}$ is given as  follows:
\[
    \begin{tikzpicture}[baseline=0,scale=2]
        \draw[-stealth, ultra thick,white!90!black] (-0.75,0) to (3.5,0);
        \draw[-stealth, ultra thick,white!90!black] (0,0.5) to (0,-1.5);
        \node (a) at (0,0) {$R$};
        \node (b) at (1,0) {$R$};
        \node (c) at (3,-1) {$R$};
        \node() at (0,0.3) {$\scriptstyle 0$};
        \node() at (1,0.3) {$\scriptstyle 1$};
        \node() at (3,0.3) {$\scriptstyle \m$};
        \node() at (-0.5,0) {$\scriptstyle 0$};
        \node() at (-0.5,-1) {$\scriptstyle -(\m-2)$};
        \draw[-stealth] (c) to node[below left]{$\scriptstyle d^{\m-1}$} (b);
        \node[draw]() at (-2,0) {$E^2_{\ast,\ast}(C_{m,n})$};
    \end{tikzpicture}
\]
The MPSS is then determined by the value of the lone remaining
differential, which is $d^{\m-1}$ as shown.
Since $C_{m,n}$ has both initial and terminal vertices,
its reachability homology is trivial, 
and therefore $E^\infty$ vanishes in positive total degrees.
Thus in $E^{\m-1}$, the terms in degrees
$(1,0)$ and $(\m,-(\m-2))$ cannot survive, 
and $d^{\m-1}$ must therefore be an isomorphism.
(The apparent separation between the cases $\m=2$ and $\m\geq 3$
arises because, when $\m=2$, the $d^{\m-1}$ differential
takes place on the $E^1$ page.)
Thus, we see that $E^2_{\ast,\ast}(C_{m,n})$ determines
the value of $\m$ via the placement of its nonzero groups.
The MPSS of $C_{m,n}$ also determines $\m$ as the first
page that is trivial, i.e.~concentrated in degree
$(0,0)$.

We now move to the proof of~\Cref{proposition-independent}.
We now assume without loss that $m\geq n$ so that $\m=m$.
Let us write $A_{m,n}$ for the subgraph of $C_{m,n}$
consisting of the first $n-1$ edges of the 
directed interval of length $n$, which in our diagrams we think
of as going across the bottom.
To make it easier to visualise the arguments to come,
it can help to redraw $C_{m,n}$ and $A_{m,n}$
in the following manner.
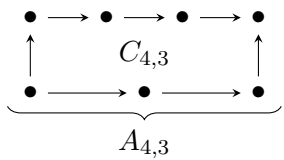
\[
    \begin{tikzpicture}[baseline=0]
        \node(x) at (0,0) {$\bullet$};
        \node(y) at (3,1) {$\bullet$};

        \node (a2) at (0,1) {$\bullet$};
        \node (a3) at (1,1) {$\bullet$};
        \node (a4) at (2,1) {$\bullet$};

        \node (b2) at (1.5,0) {$\bullet$};
        \node (b3) at (3,0) {$\bullet$};
        
        \draw[-stealth] (x) -- (a2);
        \draw[-stealth] (a2) -- (a3);
        \draw[-stealth] (a3) -- (a4);
        \draw[-stealth] (a4) -- (y);

        \draw[-stealth] (x) -- (b2);
        \draw[-stealth] (b2) -- (b3);
        \draw[-stealth] (b3) -- (y);
        
       \node () at (1.5,0.5) {$C_{4,3}$};
 
       \draw [
        decorate, 
        decoration = 
            {
                brace,
                raise=5pt,
                amplitude=5pt,
            }
        ] 
            (3.3,0) --  (-0.3,0)
            node[pos=0.5,below=10,black]{$A_{4,3}$};
    \end{tikzpicture}
\]

\begin{lemma}
    Let $m,n\geq 1$ with $m\geq n$ and $m\geq 2$.
    Then the inclusion $A_{m,n}\hookrightarrow C_{m,n}$
    is a cofibration.
\end{lemma}

\begin{proof}
    There are certainly no edges into $A_{m,n}$ from vertices
    not in $A_{m,n}$, and so the first condition of a cofibration
    holds.
    Observe that the reach of $A_{m,n}$ is then the whole of
    $C_{m,n}$, since the initial vertex of $A_{m,n}$ can
    reach every other.
    We may now define $\pi$ as follows:
    \begin{itemize}
        \item Each vertex of $A_{m,n}$ is sent to itself.
        \item The terminal vertex of $C_{m,n}$
        is sent to the terminal vertex of $A_{m,n}$.
        \item The remaining vertices of $C_{m,n}$ are sent to the initial vertex of $A_{m,n}$.
    \end{itemize}
    Observe that in all cases there is a path from $\pi(x)$
    to $x$.
    For example, in the case of $C_{4,3}$ the map $\pi$ 
    identifies all vertices of the same colour in this diagram:
    
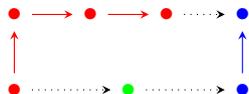
\[
        \begin{tikzpicture}[baseline=0]
            \node[red](x) at (0,0) {$\bullet$};
            \node[blue](y) at (3,1) {$\bullet$};
    
            \node[red] (a2) at (0,1) {$\bullet$};
            \node[red] (a3) at (1,1) {$\bullet$};
            \node[red] (a4) at (2,1) {$\bullet$};
    
            \node[green] (b2) at (1.5,0) {$\bullet$};
            \node[blue] (b3) at (3,0) {$\bullet$};
            
            \draw[red,-stealth] (x) -- (a2);
            \draw[red,-stealth] (a2) -- (a3);
            \draw[red,-stealth] (a3) -- (a4);
            \draw[dotted,-stealth] (a4) -- (y);
    
            \draw[dotted,-stealth] (x) -- (b2);
            \draw[dotted,-stealth] (b2) -- (b3);
            \draw[blue,-stealth] (b3) -- (y);
        \end{tikzpicture}
    \]
    
    With respect to the given map $\pi$, we may now verify
    the second condition of a cofibration, namely that
    \[
        d(a,x)=d(a,\pi(x))+ d(\pi(x),x)
    \]
    for $x\in C_{m,n}$ and $a\in A_{m,n}$.
    If $a$ and $x$ are, respectively, the initial and final 
    vertices of $C_{m,n}$, then there are two distinct paths
    from $a$ to $x$, but the shortest distance is always given
    by the path that travels through $\pi(x)$ (the lower path),
    and consequently the condition holds in this case.
    In all other cases there is at most one path
    from $a$ to $x$, and that path passes through $\pi(x)$,
    so that the condition holds in these cases too.
\end{proof}

\begin{proof}[Proof of~\Cref{proposition-independent}]
    We may assume that $m\geq n$ and $m\geq 2$.
    Then the map from the statement takes the form
    $C_{m,n}\to C_{m,1}$,  
    and fits into a pushout diagram
    \[\begin{tikzcd}
        A_{m,n} \arrow{r} \arrow{d} & A_{m,1} \arrow{d} \\
        C_{m,n} \arrow{r} & C_{m,1}
    \end{tikzcd}\]
    whose upper map collapses $A_{m,n}$ to $A_{m,1}$,
    the latter being a single vertex.
    The maps 
    $\PH_{\ast,\ast}(A_{m,n})\to \PH_{\ast,\ast}(A_{m,1})$
    are isomorphisms since $A_{m,n}\to A_{m,1}$ 
    is a $1$-homotopy equivalence.
    And the maps 
    $\PH_{\ast,\ast}(C_{m,n},A_{m,n})
    \to 
    \PH_{\ast,\ast}(C_{m,1},A_{m,1})$
    are isomorphisms by
    the excision theorem, which applies since the 
    vertical maps of our pushout are cofibrations.
    Then the long exact sequences of the pairs 
    $(C_{m,n},A_{m,n})$ 
    and  
    $(C_{m,1},A_{m,1})$
    are related by a commutative ladder,
    and the last sentence allows us to 
    apply the five~lemma to this ladder to
    obtain the result.
\end{proof}

Let us now turn to the MPSS of $C_{m,1}$ for $m\geq 2$.
To do so we need to establish some notation.
We write the vertices of $C_{m,1}$ as $a_0,\ldots,a_m$, 
with edges $a_{i-1}\to a_i$ for $i=1,\ldots,m$,
and $a_0\to a_m$.
So for $C_{4,1}$ we have the following:
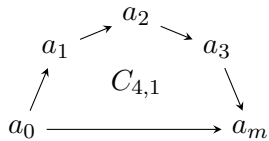
\[
    \begin{tikzpicture}[baseline=0]
        \node(x) at (180:1.5) {$a_0$};
        \node(y) at (0:1.5) {$a_m$};

        \node (a2) at (135:1.5) {$a_1$};
        \node (a3) at (90:1.5) {$a_2$};
        \node (a4) at (45:1.5) {$a_3$};

        \draw[-stealth] (x) -- (a2);
        \draw[-stealth] (a2) -- (a3);
        \draw[-stealth] (a3) -- (a4);
        \draw[-stealth] (a4) -- (y);

        \draw[-stealth] (x) -- (y);
        
        \node () at (0,0.6) {$C_{4,1}$};
    \end{tikzpicture}
\]
Next, we define some magnitude homology classes.

\begin{definition}
    We define magnitude homology classes 
    \[
        \kappa_x\in\MH_{0,0}(C_{m,1}),
        \qquad
        \lambda_e\in\MH_{1,1}(C_{m,1}),
        \qquad
        \mu\in\MH_{2,m}(C_{m,1}),
    \]
    for $x\in V(C_{m,1})$ and $e\in E(C_{m,1})$,
    by the rules
    \[
        \kappa_x = [(x)], 
        \qquad 
        \lambda_e = [(a,b)],
        \qquad
        \text{and\ }
        \mu = [(a_0,a_i,a_m)]
    \]
    where $e=ab$.
    In the definition of $\mu$ one can make any choice of
    $1<i<m$; the value of $\mu$ does not change.
    Note that in the spectral sequence grading we have
    \[
        \kappa_x\in E^1_{0,0}(C_{m,1}),
        \qquad
        \lambda_e\in E^1_{1,0}(C_{m,1}),
        \qquad
        \mu\in E^1_{m,-m+2}(C_{m,1}).
    \] 
\end{definition}

We can now state our main result.

\begin{theorem}\label{theorem-mn-cycle}
    Let $m\geq 2$.
    The magnitude homology $\MH_{\ast,\ast}(C_{m,1})$
    is the free bigraded $R$-module with basis given by the
    $\kappa_x$, $\lambda_e$ and $\mu$ for $x\in V(C_{m,1})$
    and $e\in E(C_{m,1})$.
    The differential $d^1$ acts on these classes by the rules
    $d^1(\kappa_x)=0$,
    $d^1(\lambda_e) = \kappa_b - \kappa_a$
    when $e=ab$,
    and, if $m=2$,
    $d^1(\mu) = 
    \lambda_{a_0,a_1}+\lambda_{a_1,a_2}
    -\lambda_{a_0,a_2}$.
\end{theorem}

We immediately obtain the following description of the bigraded
path homology.

\begin{corollary}
    If $m>2$ then the bigraded path homology of $C_{m,1}$ consists
    of a single copy of $R$ in bidegrees $(0,0)$, $(1,0)$
    and $(m,-m+2)$.
    If $m=2$ then the bigraded path homology of $C_{m,1}$ 
    vanishes in positive degrees.
\end{corollary}

Let us turn to the proof of Theorem~\ref{theorem-mn-cycle}.
For the final sentence of the statement we have the following.

\begin{lemma}
    The description of $d^1$ in
    Theorem~\ref{theorem-mn-cycle} holds.
\end{lemma}

\begin{proof}
    The map $d^1$ is obtained by applying the differential
    of the reachability complex to any representative.
    Thus, for example, when $m=2$ we have
    $d^1(\mu)=d^1[(a_0,a_1,a_2)]
    =[d(a_0,a_1,a_2)]
    =[(a_1,a_2)-(a_0,a_2)+(a_0,a_1)]
    =[(a_1,a_2)]-[(a_0,a_2)]+[(a_0,a_1)]
    =\lambda_{a_1a_2}-\lambda_{a_0a_2}+\lambda_{a_0a_1}$
    as claimed.
    The other cases are similar and left to the reader.
\end{proof}

In order to give our proof of~\Cref{theorem-mn-cycle}, 
we use the decomposition of
magnitude chains and magnitude homology by initial and final
points of tuples, which we now recall.
(The decomposition is well known,
and to the best of our knowledge its first 
explicit use is in Section~5.4 of~\cite{KanetaYoshinaga}.)
Let $G$ be a directed graph. 
Let $a,b$ be vertices of $G$, possibly equal,
and let $\ell\geq 0$.
Then we write $\MC_{\ast,\ell}(a,b)$ for the subcomplex of
$\MC_{\ast,\ell}(G)$ spanned by all tuples of the form
$(x_0,\ldots,x_k)$ with $x_0=a$ and $x_k=b$,
and we write $\MH_{\ast,\ell}(a,b)$ for its homology.
Note that $\MC_{\ast,\ell}(a,b)$
is indeed a subcomplex of $\MC_{\ast,\ell}(G)$, 
thanks to the fact that omitting the
first or last entries of a nondegenerate tuple always strictly
decreases its length.
Then the magnitude chains and magnitude homology of $G$ both
decompose as direct sums:
\[
    \MC_{\ast,\ast}(G)
    =
    \bigoplus_{a,b\in V(G)}\MC_{\ast,\ast}(a,b),
    \qquad\qquad
    \MH_{\ast,\ast}(G)
    =
    \bigoplus_{a,b\in V(G)}\MH_{\ast,\ast}(a,b)
\]
Note that $\MC_{\ast,\ast}(a,b)$ is nonzero if and only if
$d(a,b)<\infty$.

The description of $\MH_{\ast,\ell}(C_{m,n})$ for $\ell=0,1$
is straightforward and left to the reader. 
Compare, for example, with Proposition~2.9 of~\cite{HepworthWillerton2017}
or~Theorems~4.1 and~4.3 of~\cite{LeinsterShulman}.
Therefore it remains to prove~\Cref{theorem-mn-cycle} 
in the case $\ell\geq 2$, which is given by the following.

\begin{lemma}
    Let $\ell\geq 2$.
    Then the groups 
    $\MH_{\ast,\ell}(a_i,a_j)$
    for $0\leq i\leq j\leq m$ 
    all vanish, with the single exception of 
    $\MH_{2,m}(a_0,a_m)$,
    which is a copy of the ring $R$ spanned by the
    class $\mu$.
\end{lemma}

\begin{proof}
    We first address the case where $(i,j)\neq (0,m)$.
    The generators of 
    $\MC_{k,\ast}(a_i,a_j)$ are the tuples
    \[
        (a_i,a_{i_1},\ldots,a_{i_k},a_j)
    \]
    for $k\geq 0$ and $i<i_1<\cdots<i_k<j$,
    and in particular these all have length $j-i$.
    So we may assume $\ell=j-i$.
    Define 
    $s\colon\MC_{\ast,j-i}(a_i,a_j)
    \to \MC_{\ast+1,j-i}(a_i,a_j)$
    by the rule
    \[
        s(a_i,a_{i_1},\ldots,a_{i_k},a_j)
        =\begin{cases}
            -(a_i,a_{i+1},a_{i_1},\ldots,a_{i_k},a_j)
            &
            \text{if }i_1\neq i+1 
            \\
            0
            &
            \text{if }i_1=i+1.
        \end{cases}
    \]
    Then $(sd+ds)(a_i,a_{i_1},\ldots,a_{i_k},a_j)=
    (a_0,a_{i_1},\ldots,a_{i_k},a_m)$,
    and it follows that the homology of 
    $\MC_{\ast,j-i}(a_i,a_j)$ vanishes in all degrees.

    Next we address the case of  $\MC_{\ast,\ast}(a_0,a_m)$,
    whose generators are the tuples of the form
    $(a_0,a_{i_1},\ldots,a_{i_k},a_m)$ for $k\geq 1$.
    (We no longer allow the case $k=0$ because
    $\ell(a_0,a_m)=1$ and we have assumed $\ell\geq 2$.)
    Any one of these has length $m$, and so we may assume that
    $\ell=m$.
    Define 
    $s\colon\MC_{\ast,m}(a_0,a_m)
    \to \MC_{\ast+1,m}(a_0,a_m)$
    by the rule
    \[
        s(a_0,a_{i_1},\ldots,a_{i_k},a_m)
        =\begin{cases}
            -(a_0,a_1,a_{i_1},\ldots,a_{i_k},a_m)
            &
            i_1\neq 1
            \\
            0
            &
            i_1=1.
        \end{cases}
    \]
    Then $(sd+ds)(a_0,a_{i_1},\ldots,a_{i_k},a_m)$
    is equal to $(a_0,a_{i_1},\ldots,a_{i_k},a_m)$ 
    except in the case $k=1$, $i_1=1$, when it vanishes.
    It follows that the homology of $\MC_{\ast,m}(a_0,a_m)$
    vanishes except for $\MC_{2,m}(a_0,a_m)$,
    which is generated by the class of $(a_0,a_1,a_m)$,
    this class being precisely $\mu$.
\end{proof}

\bibliographystyle{plain} 
\bibliography{Bigraded_PH_and_the_MPSS} 

\end{document}